\documentclass[11pt,oneside,reqno]{amsart}
\usepackage[T1]{fontenc}
\usepackage[utf8]{inputenc}
\usepackage[a4paper,margin=27mm]{geometry}
\usepackage{amsmath,amssymb,lmodern,microtype}
\usepackage{enumitem}
\usepackage{graphicx,pgfplots}
\usepackage{todonotes}

\usepgfplotslibrary{groupplots}
\pgfplotsset{compat=1.18}
\definecolor{phaseblue}{RGB}{0,92,145}
\definecolor{phaseorange}{RGB}{174,85,0}
\usepackage{hyperref}
\usepackage{refcheck}
\numberwithin{equation}{section}
\newtheorem{theorem}{Theorem}[section]
\newtheorem{proposition}[theorem]{Proposition}
\newtheorem{corollary}[theorem]{Corollary}
\newtheorem{lemma}[theorem]{Lemma}
\theoremstyle{remark}
\newtheorem{remark}[theorem]{Remark}
\newcommand{\R}{\mathbb R}
\newcommand{\dd}{\,\mathrm d}

\definecolor{myred}{rgb}{0.84,0.07,0.14}

\newcommand{\N}{\mathbb N}

\def\dd{\,{\rm d}}

\newcommand{\D}{\mathcal D_N}

\newcommand{\hu}{\widehat u}
\newcommand{\hv}{\widehat v}
\newcommand{\tw}{\widetilde w}
\newcommand{\tmu}{\widetilde\mu}
\newcommand{\tnu}{\widetilde\nu}
\newcommand{\tU}{\widetilde U}
\newcommand{\tV}{\widetilde V}
\newcommand{\tdelta}{\widetilde\delta}

\newcommand{\cA}{{\mathcal A}}

\newcommand{\cD}{{\mathcal D}}
\newcommand{\cE}{{\mathcal E}}
\newcommand{\cF}{{\mathcal F}}

\newcommand{\cH}{{\mathcal H}}

\newcommand{\cK}{{\mathcal K}}
\newcommand{\cL}{{\mathcal L}}
\newcommand{\cM}{{\mathcal M}}

\newcommand{\cP}{{\mathcal P}}

\newcommand{\cR}{{\mathcal R}}
\newcommand{\cS}{{\mathcal S}}
\newcommand{\cT}{{\mathcal T}}

\newcommand{\cW}{{\mathcal W}}

\title[Normalised Hamiltonian Elliptic Systems]{Normalised Hamiltonian Elliptic Systems:\\ a Gagliardo-Nirenberg inequality for bilinear mass}
\author[]{Daniele Cassani\textsuperscript{*}}
\thanks{\textsuperscript{*}Corresponding author: Daniele Cassani (daniele.cassani@uninsubria.it).}
\address[Daniele Cassani]{Dipartimento di Scienza e Alta Tecnologia, Universit\`{a} degli Studi dell'Insubria and RISM-Riemann International School of Mathematics, Villa Toeplitz, Via G.B. Vico, 46 - 21100 Varese, Italy}
\email{daniele.cassani@uninsubria.it}

\author[]{Giulio Romani}%
\address[Giulio Romani]{Universit\`{a} degli Studi di Udine, Via delle Scienze 206 - 33100 Udine, Italy}
\email{giulio.romani@uniud.it}

\date{\today}
\hypersetup{pdftitle={Normalised Hamiltonian Elliptic Systems: a Gagliardo-Nirenberg inequality for bilinear mass},
 pdfauthor={Daniele Cassani and Giulio Romani}}

\begin{document}
	
\raggedbottom
\maketitle

\begin{abstract}
	We study Hamiltonian elliptic systems in $\R^N$ under the prescribed bilinear mass constraint $\int_{\R^N}uv=a>0$. We establish a Gagliardo–Nirenberg inequality involving the crossed gradient pairing and the bilinear mass on scaling-invariant classes subject to componentwise bounds. This inequality identifies the critical curve $\frac1p+\frac1q=\frac{N}{N+2}$, which also governs the variational behaviour of the energy functional associated to the Hamiltonian system. For the corresponding fixed-frequency system with power nonlinearities, we prove uniqueness of positive solutions up to common translations and nondegeneracy throughout the Sobolev-subcritical range. By exploiting the exact scaling laws, we classify positive solutions satisfying the prescribed mass constraint and identify the unique critical bilinear mass.
	In dimension two, we develop a bilinear exponential Gagliardo--Nirenberg estimate with a sharp gradient threshold, together with a truncated scalar refinement with the optimal quartic leading coefficient. This allows to prove existence of solutions with prescribed bilinear mass for a planar Hamiltonian system with critical exponential nonlinearities. Building on ideas from Cassani and Tarsi \cite{CassaniTarsi2015}, we establish existence for every prescribed mass below the cubic limiting mass. 
	Finally, in both dimension two and higher dimensions, we use local perturbative arguments to obtain branches of positive solutions with prescribed bilinear mass away from criticality, and to derive explicit mass-response formulas at criticality.
\end{abstract}

\section{Introduction}
For $N\ge2$, we study positive solutions of the Hamiltonian system
\begin{equation}\label{system}
	\begin{cases}
		-\Delta u+\lambda u=|v|^{p-2}v&\text{in }\R^N,\\
		-\Delta v+\lambda v=|u|^{q-2}u&\text{in }\R^N,\\
		\displaystyle\int_{\R^N}uv=a\,,
	\end{cases}
\end{equation}
with a common frequency $\lambda>0$, powers $p,q>2$ and a new normalisation which is the prescribed bilinear mass $a>0$. In the planar case $N=2$, we also discuss its exponential counterpart. Two features distinguish these problems from scalar equations with prescribed mass: first, the quadratic energy is indefinite, and moreover our constraint measures an interaction between the components rather than their separate sizes.
Main goals are to identify interpolation inequalities adapted to
these quantities and to classify the positive normalised solutions from analytical and variational point of view.

If $p=q=m$ and $u=v$, \eqref{system} reduces to the scalar equation
\begin{equation}\label{scalar}
	-\Delta w+\lambda w=|w|^{m-2}w\,,\qquad w\in H^1(\R^N)\,,\qquad \int w^2=a\,.
\end{equation}
Here and below an integral without a domain is over $\R^N$, and mass
means the $L^2$-norm squared. We write $2^*=\frac{2N}{N-2}$ for $N\ge3$ and $2^*=\infty$ for $N=2$. Equation \eqref{scalar} is obtained by investigating standing waves of the form $e^{i\lambda t}w(x)$ for the focusing nonlinear Schr\"odinger equation
\begin{equation}\label{scalar_evolution}
	i\partial_t\varphi=-\Delta\varphi-|\varphi|^{p-2}\varphi\,.
\end{equation}
At fixed frequency, the mass is an unknown of the problem and solutions are critical points of the scalar \textit{action functional}
$$I_\lambda^{\rm sc}(w)=\frac12\int|\nabla w|^2+\frac\lambda2\int w^2-\frac1m\int|w|^m\,.$$
This problem has a long history and, for $2<m<2^*$, variational and concentration-compactness methods produce positive profiles; radial symmetry and uniqueness identify a single centred profile $Q_m$ at frequency $\lambda=1$ \cite{Lions1,Lions2,Kwong,Cazenave}.

The alternative is to prescribe the mass and let the frequency arise as a Lagrange multiplier. This approach has its root in the fact that the $L^2$-mass is conserved by the evolution in \eqref{scalar_evolution}. The classical Gagliardo--Nirenberg inequality \cite{N,FFRS}
\begin{equation}\label{scalarGN}
	\|w\|_m^m\le C_{N,m}\|\nabla w\|_2^{N(m-2)/2}\|w\|_2^{m-N(m-2)/2}
\end{equation}
then compares the nonlinear energy with the Dirichlet energy on the
mass sphere. The exponent $m=2+\frac4N$ marks their balance under mass-preserving dilations. Since the nonlinearity is homogeneous, all solutions $w_\lambda$ can be derived from a unique positive centred scalar profile $Q_m$, which is the unique solution of the unit-frequency equation
\begin{equation*}
	-\Delta w+w=|w|^{p-2}w\,,
\end{equation*}
see \cite{Kwong}. In fact, for $\lambda>0$,
\begin{equation*}
	w_\lambda(x):=\lambda^\frac1{m-2}Q_m(\sqrt\lambda x)\,,\qquad\int w_\lambda^2=\lambda^{\frac2{m-2}-\frac N2}\|Q_m\|_2^2\,.
\end{equation*}
If $m\neq2+\frac4N$ each mass selects one frequency, whereas at criticality the scaling leaves the mass invariant, hence all frequencies carry the same mass. This implies that for $m\in(2,2^*)\setminus\{2+\frac4N\}$, for every prescribed mass $a>0$ there exists a positive normalised solution, 
while, in the \textit{mass-critical} case $m=2+\frac4N$, 
the mass of a positive solution is necessarily equal to the mass of $Q_m$, and normalised solutions exist only for this distinguished value of the mass, \cite{Weinstein}.

Even the variational structure of the problem is influenced by the mass-critical exponent. In fact, solutions correspond to critical points of the \textit{energy functional}
$$J^{\rm sc}(w)=\frac12\int|\nabla w|^2-\frac1m\int|w|^m$$
constrained to the $L^2-$sphere $\Sigma_a^{\rm sc}=\{w\in H^1(\R^N)\,|\,\int w^2=a\}$. By means of \eqref{scalarGN}, $J^{\rm sc}$ changes geometry on $\Sigma_a^{\rm sc}$ according to the mutual position of $m$ and $2+\frac4N$. If $m<2+\frac4N$ (mass-subcritical regime), $J$ is bounded below and critical points are obtained by minimisation arguments. 
Above it (mass-supercritical regime), $J$ is unbounded below and more sophisticated min-max techniques, which make use of dilation fibres and Poho\v zaev constraints,
provide a variational route to solutions \cite{Jeanjean,BS,BScorrection}.

This scalar variational viewpoint also underlies the orbital-stability
method of Cazenave and Lions \cite{CazenaveLions1982} and the connection
between constrained coercivity and Lyapunov stability developed by
Weinstein \cite{Weinstein1986}. At the critical mass, the ground state
also organises concentration: critical-mass blow-up and continuation
were studied by Merle \cite{Merle1992}, and near-ground-state blow-up
dynamics by Merle and Rapha\"el \cite{MerleRaphael2005}. These results
explain the broader significance of identifying an interpolation
constant, a critical mass and the linearised kernel. Our Hamiltonian
system \eqref{system} demands for their counterparts in an indefinite variational
setting.

\vskip0.2truecm
The normalised viewpoint is closely connected with wave localisation.
In nonlinear optics, diffraction competes with self-focusing, and a
stationary transverse profile describes a beam propagating without
changing shape. Its squared $L^2$ norm represents optical power, while
the phase frequency represents a propagation constant. The planar cubic
ground state $Q_4$ which solves
\begin{equation}\label{Q4}
	-\Delta Q_4+Q_4=Q_4^3\qquad\text{in }\R^2.
\end{equation}
is the normalised Townes profile associated with
critical self-trapping \cite{ChiaoGarmireTownes1964}, and its mass also
determines the sharp quartic GN constant \cite{Weinstein}. Thus a
functional-analytic best constant identifies a physically meaningful
power scale. A similar interpretation arises for multicomponent Bose--Einstein condensates, where coupled Gross--Pitaevskii equations describe interacting matter waves. The component masses represent particle numbers, while the frequency parameters correspond to chemical potentials. In this setting, the interaction between components can lead to localised coupled states, and their existence and properties depend crucially on the nature and strength of the coupling \cite{KevrekidisFrantzeskakis2016}.
For planar Bose fluids, Cassani, Tavares, and Zhang
\cite{CassaniTavaresZhang2020} study weakly coupled systems on bounded
domains with critical exponential self-interaction and coupling terms.
They obtain positive vector ground states in suitable cooperative
regimes and positive vector solutions for weak competition. Their
model gives a concrete precedent for exponential interactions with a
cubic small-amplitude limit, within a gradient-system setting.

\vskip0.2truecm
These applications motivate a common mathematical question: which
coherent profiles are compatible with a prescribed conserved quantity?
The crossed system \eqref{system} studied here has its own Hamiltonian structure and
charge, distinct from the usual particle-number constraints of
Gross--Pitaevskii models. Its standing-wave interpretation is supplied
by the following evolution system
\begin{equation}\label{crossflow}
	\begin{cases}
		i\partial_t\psi=-\Delta\psi-|\varphi|^{p-2}\varphi\,,\\
		i\partial_t\varphi=-\Delta\varphi-|\psi|^{q-2}\psi\,.
	\end{cases}
\end{equation}
For smooth decaying complex fields $z=(\psi,\varphi)$, define
\[
 \begin{aligned}
 \mathcal H(z)&=\operatorname{Re}\int\nabla\psi\cdot\overline{\nabla\varphi}
       -\frac1q\int|\psi|^q-\frac1p\int|\varphi|^p\,,\qquad\mathcal Q(z)&=\operatorname{Re}\int\psi\overline\varphi\,.
 \end{aligned}
\]
Setting $S(\psi,\varphi)=(\varphi,\psi)$ and taking gradients with
respect to the real $L^2$ pairing, \eqref{crossflow} becomes
\[
 i\partial_tz=S\nabla\mathcal H(z)\,,\qquad S\nabla\mathcal Q(z)=z\,.
\]
Consequently $\mathcal Q$ generates common phase rotations for this
Hamiltonian structure: it is the charge associated with its gauge
symmetry. The ansatz $z=e^{i\lambda t}(u,v)$ gives
$\mathcal H'(u,v)+\lambda\mathcal Q'(u,v)=0$, hence \eqref{system}
when $\mathcal Q=a$.
For smooth sufficiently decaying solutions, integration by parts gives
\[
 \frac{\dd}{\dd t}\operatorname{Re}\int\psi\overline\varphi
 =\operatorname{Re}\left(i\int|\varphi|^p-i\int|\psi|^q\right)=0\,.
\]
Thus the normalised elliptic problem selects standing-wave profiles with prescribed bilinear charge.

The change of variables $z_\pm=(\psi\pm\varphi)/\sqrt2$ makes the
meaning and the limitation of this charge explicit:
\[
 \mathcal Q(z)=\frac12\bigl(\|z_+\|_2^2-\|z_-\|_2^2\bigr)\,.
\]
It fixes the difference between the intensities of the symmetric and
antisymmetric modes, leaving their sum uncontrolled. This explains why
the scalar mass-sphere theory cannot simply be applied componentwise,
and why the component control in our bilinear interpolation inequality
is essential. The direct application developed here is the construction
and classification of standing waves for \eqref{crossflow} at fixed
interaction charge.

\vskip0.2truecm
At fixed frequency, Hamiltonian elliptic systems have a developed
existence theory \cite{FY,Sirakov,dF,BMR,BMT}. Their natural exponent
range is bounded by the Sobolev critical hyperbola of Mitidieri
\cite{Mitidieri1993,Mitidieri},
\begin{equation}\label{critical_curve}
 \frac1p+\frac1q=\frac{N-2}{N}\,.
\end{equation}
One main difficulty is that, although the problem possesses a variational structure, the \textit{action functional} associated to \eqref{system}
\begin{equation*}
	I_\lambda(u,v)=\int\nabla u\cdot\nabla v+\lambda\int uv-\frac1q\int|u|^q-\frac1p\int|v|^p
\end{equation*}
is strongly indefinite. A least-action positive profile minimises $I_\lambda$ among positive solutions at that frequency $\lambda$. Dual formulations handle the indefiniteness without minimising the original action on the entire function space \cite{BMT}.

Only few works deal with Hamiltonian systems of the kind \eqref{system} under
a mass-type constraint. Mass-constrained gradient systems with separate component masses lead to separate frequency multipliers \cite{BartschJeanjean,BJS,GJ,CZZ}. Looking for critical points of the \textit{energy functional}
\begin{equation*}
	J(u,v)=\int\nabla u\cdot\nabla v-\frac1q\int|u|^q-\frac1p\int|v|^p
\end{equation*}
under a single total $L^2$ constraint in \eqref{system} produces a different Hamiltonian systems. Indeed, if
\begin{equation}\label{total_mass}
	M(u,v)=\frac12\int(u^2+v^2)\,,
\end{equation}
then $J'+\lambda M'=0$ gives
\begin{equation}\label{totalmassEL}
 -\Delta u+\lambda v=|v|^{p-2}v\,,\qquad
 -\Delta v+\lambda u=|u|^{q-2}u\,.
\end{equation}
In contrast, denoting by 
\begin{equation}\label{bilinear_mass}
	B(u,v)=\int uv
\end{equation}
our \textit{bilinear mass constraint}, $J'+\lambda B'=0$ gives the diagonal frequency terms in
\eqref{system}. The familiar scalar rule, ``an $L^2$ constraint produces
a frequency term'', therefore requires care for a crossed energy:
the derivative of the constraint determines which component carries
the multiplier. Remark~\ref{quadraticconstraint} records this argument
for every quadratic constraint with constant coefficients.
This computation is relevant to \cite{MaiaMiyagaki},
whose stated problem combines diagonal frequency terms with a total
$L^2$ constraint: constraining the crossed energy $J$ by that total
mass yields \eqref{totalmassEL}, a different, meaningful normalisation problem, not its displayed diagonal-frequency system. 
The crossed-multiplier formulation appears explicitly in
\cite[system (1), constraint (2)]{ZhangYuZhang} and is consistent with
its total-mass constraint. Likewise, separate masses are appropriate
for the gradient systems cited above. 
We stress once more that our choice of bilinear mass keeps the frequency terms diagonal and the coupling in the nonlinearities.

\vskip0.2truecm
Throughout the paper we assume
\begin{equation}\label{range}
	N\geq2\,,\qquad p,q>2\,,\qquad\frac1p+\frac1q>\frac{N-2}{N}\,.
\end{equation}
In particular for $N=2$ every $p,q>2$ is admissible. By working in products of suitable conjugated Sobolev spaces, \eqref{bilinear_mass} yields to the definition of a new \textit{bilinear-mass critical hyperbola}, which extends the concept of critical mass to the Hamiltonian setting of \eqref{system} from both analytical and variational point of view. This new curve is given by
\begin{equation}\label{criticalcurve}
	\gamma_N(p,q):=\frac1p+\frac1q-\frac{N}{N+2}=0\,.
\end{equation}
It lies strictly inside the Sobolev subcritical region and passes through the scalar mass-critical exponent $p=q=2+\frac4N$, see Figure \ref{figcriticalcurves}. 

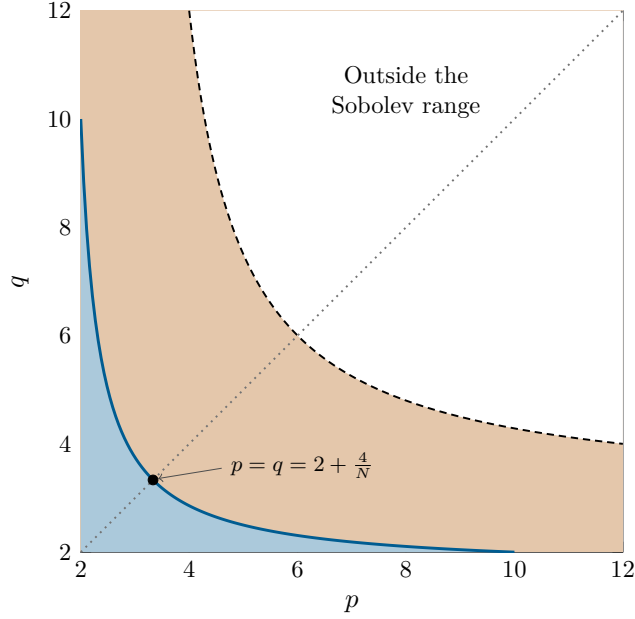
\begin{figure}[htbp]
	\centering
	\resizebox{0.53\linewidth}{!}{%
		\begin{tikzpicture}
			\begin{axis}[width=9.4cm,height=9.4cm,axis equal image,
				xmin=2,xmax=12,ymin=2,ymax=12,
				xlabel={$p$},ylabel={$q$},xtick={2,4,6,8,10,12},ytick={2,4,6,8,10,12},
				tick label style={font=\small},label style={font=\small},
				axis lines=box,clip=true]
				\path[fill=phaseorange!33] (axis cs:2,2) rectangle (axis cs:12,12);
				\addplot[draw=none,fill=phaseblue!33,domain=2:10,samples=150] {5*x/(3*x-5)} \closedcycle;
				\addplot[draw=none,fill=white!13,domain=4:12,samples=150] {3*x/(x-3)} -- (axis cs:12,12) -- (axis cs:4,12) --cycle;
				\addplot[black,densely dashed,thick,domain=4:12,samples=150] {3*x/(x-3)};
				\addplot[phaseblue,very thick,domain=2:10,samples=150] {5*x/(3*x-5)};
				\addplot[black!55,dotted,thick,domain=2:12] {x};
				\node[align=center,font=\small,fill=white!13,inner sep=2pt] at (axis cs:8,10.5) {Outside the\\Sobolev range};
				\addplot[only marks,mark=*,mark size=2pt,black] coordinates {(3.333333,3.333333)};
				\draw[->,black!70] (axis cs:4.6,3.55)--(axis cs:3.43,3.36);
				\node[anchor=west,font=\footnotesize] at (axis cs:4.6,3.6) {$p=q=2+\frac4N$};
			\end{axis}
		\end{tikzpicture}
	}
	\caption{Sobolev and bilinear-mass critical curves in the $(p,q)$ plane. The blue and the orange region correspond to the bilinear-mass subcritical and supercritical regimes, respectively. The solid line represents the bilinear-mass critical curve $\gamma_N(p,q)=0$, whereas the dashed line is the Sobolev critical hyperbola. The picture is quantitatively drawn in the case $N=3$. 
	The diagonal corresponds to the scalar problem.}
	\label{figcriticalcurves}
\end{figure}

As first main result, we show that, for Sobolev subcritical exponents which lie outside this curve, the constrained system \eqref{system} always admits a positive solution for any value of the mass, while, if $\gamma_N(p,q)=0$, positive solutions can be found only for a specific value of the bilinear mass. Such solutions can be obtained by suitable scaling from positive solutions of the unit-frequency system
\begin{equation}\label{system1}
	\begin{cases}
		-\Delta U+U=|V|^{p-2}V\quad\mbox{in}\ \,\R^N\,,\\
		-\Delta V+V=|U|^{q-2}U\quad\mbox{in}\ \,\R^N\,.
	\end{cases}
\end{equation}
Let $\mathcal S$ denote the set of positive classical solutions of \eqref{system1}, which decay at
infinity. Lemma \ref{profileregularity} identifies this class with the positive
$H^1(\R^N)^2$ weak solutions and establishes their exponential decay. Fixed-frequency ground-state existence in the entire range \eqref{range}, including the plane, is classical, see \cite[Theorems 1.8--1.9 and Lemma 3.5]{BMR} and
\cite[Section 6.2, Theorem 6.4]{BMT}, hence in the next result, the existence and dual-attainment assertions for \eqref{system1} are not claimed as new.

For the sake of a compact notation, set
\begin{equation*}
	\sigma:=\frac1p+\frac1q\in\left(\frac{N-2}N,1\right)
\end{equation*}
because we work in the Sobolev subcritical regime and $p,q>2$, and
\begin{equation}\label{parameters}
	D:=pq-p-q=pq(1-\sigma)>0\quad\mbox{and}\quad\gamma:=\frac{\sigma}{1-\sigma}-\frac N2=\frac{N+2}{2(1-\sigma)}\gamma_N(p,q)\,.
\end{equation}
Set also the conjugate exponents $p'=\frac p{p-1}$, $q'=\frac q{q-1}$, and
$\mathcal R=(-\Delta+1)^{-1}$. The dual constant is
\begin{equation}\label{Cpq}
	\mathfrak C_{p,q}:=\sup_{\|f\|_{q'}=\|g\|_{p'}=1}\int f\,\mathcal Rg
\end{equation}
and the least-action is denoted by
\begin{equation}\label{c_+}
	c_+:=\inf_{(U,V)\in\mathcal S}I_1(U,V)\,.
\end{equation}

\begin{theorem}[Existence, scaling and classification]\label{introexistencescaling}
	Suppose $p$, $q$, and $N$ satisfy \eqref{range}.
	For every $(U,V)\in\mathcal S$ and $\lambda>0$, the pair
	\begin{equation}\label{introfrequencyscaling}
		u_\lambda(x)=\lambda^{p/D}U(\sqrt\lambda x),\qquad
		v_\lambda(x)=\lambda^{q/D}V(\sqrt\lambda x)
	\end{equation}
	solves the two equations in \eqref{system} and has mass
	\begin{equation}\label{masslaw}
		B(u_\lambda,v_\lambda)=\lambda^\gamma B(U,V).
	\end{equation}
	Conversely, every positive-frequency solution in the decay class
	$\mathcal S$ is obtained uniquely from a unit-frequency profile
	in this way, at its given frequency.
	If $\gamma_N(p,q)\ne0$, every $a>0$ is attained, with frequency
	\begin{equation}\label{frequencychoice}
		\lambda=\left(\frac a{B(U,V)}\right)^\frac1\gamma\,.
	\end{equation}
	Moreover, every $a>0$ admits exactly one positive normalised solution triple
	$(\lambda,u,v)$ modulo translation.
	
	If $\gamma_N(p,q)=0$, the entire positive mass spectrum consists
	of the single value
	\begin{equation}\label{introcriticalmass}
	 a_*=c_+=\frac{2}{N+2}\mathfrak C_{p,q}^{-\frac{N+2}2}.
	\end{equation}
	At this mass, all positive solutions in $\mathcal S$ are precisely
	the frequency-scaled copies of the unique unit-frequency profile.
	The constant $\mathfrak C_{p,q}$ is finite, positive and attained
	by a nonnegative radially decreasing pair. It yields a positive
	least-action profile in $\mathcal S$.
\end{theorem}
The uniqueness parts of Theorem \ref{introexistencescaling} are based on a new global uniqueness result below, that we establish for the unit-frequency system \eqref{system1}. In particular, in Theorem \ref{introexistencescaling} it permits to identify the entire positive mass spectrum, not only its least value.

Although uniqueness of positive radial solutions for the scalar equation \eqref{scalar} has been proved by Kwong more than thirty years ago in \cite{Kwong}, the same issue for the Hamiltonian system \eqref{system1} was never fully answered. We are aware only of results in balls of $\R^N$ in the sublinear range $p\in(1,2)$ by \cite{CCST}, and for the one-dimensional version of \eqref{system1} in \cite{ACS}, where uniqueness is found for all $p,q>1$. However, as the authors emphasise, \textit{dimension $N=1$ is a critical assumption} therein. We also point out that in \cite{MC} the authors claim to having proved uniqueness for \eqref{system1} for $N\geq2$ and exponents $2<p,q<2^*$, but their proof has a gap.

Note that the symmetry result in \cite{BuscaSirakov} provides a common centre for every
element of $\mathcal S$. We call a pair \textit{centred} when this centre is the origin. The next theorem classifies all positive profiles throughout the full Sobolev-subcritical regime.

\begin{theorem}[Global uniqueness and nondegeneracy]\label{thm:main}
	Under \eqref{range}, every positive $H^1(\R^N)^2$ weak solution of
	\eqref{system1} is smooth and exponentially decaying.
	\begin{enumerate}
		\item There is exactly
		one centred positive weak solution, and every positive weak solution is a common translate of this pair.
		\item At any positive weak solution $(U,V)$, the linearised system
		\begin{equation}\label{eq:linearized}
			\begin{cases}
				(-\Delta+1)\phi=(p-1)V^{p-2}\psi\,,\\
				(-\Delta+1)\psi=(q-1)U^{q-2}\phi
			\end{cases}
		\end{equation}
		has, in each of $L^\infty(\R^N)^2$, $L^2(\R^N)^2$,
		$H^1(\R^N)^2$ and $H^2(\R^N)^2$, precisely the kernel
		\begin{equation*}
			\operatorname{span}\{(\partial_{x_j}U,\partial_{x_j}V):1\le j\le N\}\,.
		\end{equation*}
	\end{enumerate} 
\end{theorem}

The theorem concerns all positive profiles, without an a priori action
bound or a proximity assumption. We dedicate Section \ref{secuniqueness} to its complete proof, whose strategy is adapted from a similar result in the context of critical Lane--Emden systems \cite{CassaniXuZhang}. The main point of the proof is to show nondegeneracy: to this aim, we combine the fact that the radial Green operator preserves monotonicity with a comparison by means of spherical harmonics, which leaves only the translations in the kernel of the linearised system. 
Once nondegeneracy is established, compactness of all centred profiles and continuation from the diagonal then give global uniqueness. 

\vskip0.2truecm
Combining Theorems \ref{introexistencescaling} and \ref{thm:main} one finds the complete characterisation of normalised solutions off bilinear criticality. Further results on frequency and energy of those positive normalised solutions will be given in Section \ref{Section_normalised}, see Proposition \ref{introleastenergy}.

\vskip0.2truecm

As already described, the mass law \eqref{masslaw} identifies the critical curve \eqref{criticalcurve}. 
A further question is whether this same curve separates different geometries of the constrained
energy functional $J$. In the scalar problem, the Gagliardo-Nirenberg inequality \eqref{scalarGN} supplies precisely
this connection: on a fixed mass sphere it compares the nonlinear
potential with a power of the Dirichlet energy, explaining the passage
from an energy $J^{\rm sc}$ which is bounded below to a critical mass threshold and then to
a min-max geometry \cite{Weinstein,Jeanjean}. This variational information
is essential, especially for problems involving nonlinearities which are not pure powers, and for which a prescribed mass
can no longer be obtained simply by rescaling a fixed-frequency
solution. The purpose of our new bilinear Gagliardo-Nirenberg type inequality which we are now introducing, is to recover that
comparison in the quantities actually occurring in $J$ and $B$.

A main obstruction is that for a pair $(u,v)$ of positive solutions, neither $B(u,v)$ nor
$$A(u,v):=\int\nabla u\cdot\nabla v$$
controls the corresponding component
norms: reciprocal changes of amplitude preserve both pairings, see Proposition \ref{obstruction} below.
The vector-valued GN inequalities for gradient systems already used in the literature, exemplified by \cite{MZ}, use sums of component masses and Dirichlet energies. 
Here one needs an estimate in the crossed
pairings themselves, together with conditions making that estimate
possible.

Our new bilinear GN-type inequality is formulated on a scaling-invariant admissible class, with the exponents of \(A\) and \(B\) uniquely determined by scaling. The analytic foundation is provided by scalar interpolation inequalities, while its structure, specific for system \eqref{system}, emerges from the bilinear formulation and the resulting constrained energy geometry.
Let us describe this inequality in details. Choose conjugate exponents $r,s>1$ as in Lemma \ref{setting} below, so that
$W^{1,r}(\R^N)\hookrightarrow L^q(\R^N)$ and $W^{1,s}(\R^N)\hookrightarrow L^p(\R^N)$, and set
$X=W^{1,r}(\R^N)\times W^{1,s}(\R^N)$ and
\begin{equation}\label{Sigma_a}
	\Sigma_a=\{(u,v)\in X:B(u,v)=a\}\,.
\end{equation}
For $(u,v)\in X$ for which $A(u,v)>0$ and $B(u,v)>0$, let $(\widehat u,\widehat v)$ be the unique pair obtained by the transformation
\[
 T_{\mu,t}(u,v)=\bigl(\mu^{1/q}u(t\,\cdot)\,,\mu^{1/p}v(t\,\cdot)\bigr),\qquad \mu,t>0\,,
\]
which satisfies $A(\widehat u,\widehat v)=B(\widehat u,\widehat v)=1$.
For $K>0$, the admissible class $\mathcal M_K$ consists of these 
pairs for which
\begin{equation}\label{MK}
	\max\{\|\widehat u\|_r,\|\nabla\widehat u\|_r,\|\widehat v\|_s,\|\nabla\widehat v\|_s\}\le K\,.
\end{equation}
Thus the component bounds are imposed after bilinear normalisation; moreover, they are preserved by both parameters of $T_{\mu,t}$.

\begin{theorem}[Bilinear Gagliardo--Nirenberg inequality]\label{introinterpolation}
	Under \eqref{range}, for all sufficiently large $K$, the set $\mathcal M_K$ is nonempty,
	invariant under every $T_{\mu,t}$ and common translation, and
	intersects every $\Sigma_a$, $a>0$. On this class,
	\begin{equation}\label{introGN}
		\frac1q\int|u|^q+\frac1p\int|v|^p\le C_K\left(\int\nabla u\cdot\nabla v\right)^\alpha\left(\int uv\right)^\beta,
	\end{equation}
	where
	\begin{equation}\label{alphabeta}
		\alpha:=\frac N2\left(\left(\frac1p+\frac1q\right)^{-1}-1\right)\,,\qquad\beta:=\left(\frac1p+\frac1q\right)^{-1}-\alpha\,,
	\end{equation}
	and one may take
	\[
	 C_K=\frac{C_{N,r,q}^qK^q}{q}+\frac{C_{N,s,p}^pK^p}{p}\,,
	\]
	with the scalar interpolation constants of \eqref{GNrs}.
\end{theorem}
Note that, if $p=q=m$ and $u=v=w$, one retrieves the classical scalar GN inequality \eqref{scalarGN}. 
In Proposition~\ref{equivalent} we also establish on the same
invariant class, the equivalence between our bilinear Gagliardo-Nirenberg inequality \eqref{introGN} and the inhomogeneous inequality
$$\int\nabla u\cdot\nabla v+\int uv\geq c_K\left(\frac1q\int|u|^q+\frac1p\int|v|^p\right)^\sigma$$
for some $c_K>0$.


The critical curve now has an energy interpretation as well as a
frequency-scaling interpretation. If $\gamma_N(p,q)=0$ one has $\alpha=1$ and $\beta=2/N$, that is, the nonlinear energy
is bounded by a mass-dependent multiple of the crossed Dirichlet
energy. On either side of the critical curve, the comparison is with a sublinear or a
superlinear power of $A$. The scalar limit, this critical balance,
and invariance under both amplitude and spatial scaling are the
three features that determine the form of \eqref{introGN}. The next
theorem makes its geometric consequences precise.

Set $\cM_K^a=\cM_K\cap\Sigma_a$, henceforth called ``admissible mass slice'', and denote 
\begin{equation*}
	F(u,v):=\frac1p\int|u|^p+\frac1q\int|u|^q\,.
\end{equation*}
Define also $C_K^*\in(0,C_K]$ as the best constant of the inequality \eqref{introGN}, namely
\begin{equation}\label{C_K*}
	C_K^*=\sup_{(u,v)\in\mathcal M_K}\frac{F(u,v)}{A(u,v)^\alpha B(u,v)^\beta}\,.
\end{equation}
The mass-preserving dilation and its exact energy law are
\begin{equation*}
	S_t(u,v)=\bigl(t^{N/(q\sigma)}u(t\,\cdot), t^{N/(p\sigma)}v(t\,\cdot)\bigr),
\end{equation*}
\begin{equation*}
	J(S_t(u,v))=t^2A(u,v)-t^{2\alpha}F(u,v)\,.
\end{equation*}
Hence, we call \textit{fibre} of $(u,v)$ the set
\begin{equation}\label{fiber_set}
	\{S_t(u,v):t>0\}.
\end{equation}
These fibres give the following precise analogue of the scalar
energy trichotomy, on the stated admissible sets.

\begin{theorem}[Energy trichotomy]\label{introenergytrichotomy}
	Fix $a>0$ and let $K>0$ such that $\mathcal M_K\ne\emptyset$.
	\begin{enumerate}
	\item If $\gamma_N(p,q)>0$, equivalently $0<\alpha<1$, then
	\begin{equation}\label{subinf}
		\inf_{\mathcal M_K^a}J=-(1-\alpha)\alpha^{\alpha/(1-\alpha)}(C_K^*a^\beta)^{1/(1-\alpha)}\in(-\infty,0)\,.
	\end{equation}
	Moreover $J(u,v)\to+\infty$ along any sequence in $\mathcal M_K^a$
	with $A(u,v)\to+\infty$.
	\item If $\gamma_N(p,q)=0$, equivalently $\alpha=1$, $\beta=2/N$, define
	\begin{equation}\label{threshold}
		a_K=(C_K^*)^{-\frac N2}\,.
	\end{equation}
	The infimum of $J$ on $\cM_K^a$ is $0$ for $a\le a_K$ and
	$-\infty$ for $a>a_K$. For $a<a_K$ every element in $\cM_K^a$ has
	strictly positive energy, so that the infimum is not attained. At $a=a_K$, zero energy is attained if and only if the best $C_K^*$ constant is attained.
	\item If $\gamma_N(p,q)<0$, equivalently $\alpha>1$, then
	$\inf_{\mathcal M_K^a}J=-\infty$. Moreover, there exists $\rho>0$, such that
	\begin{equation}\label{barrier}
		J(u,v)\geq\tfrac12A(u,v)\quad\text{whenever }(u,v)\in\cM_K^a\quad\mbox{and}\quad 0<A(u,v)\leq\rho\,.
	\end{equation}
	For each $(u,v)\in\cM_K^a$, its fibre set has a unique critical point, a strict global maximum at
	$$t_* =\left(\frac{A(u,v)}{\alpha F(u,v)}\right)^{1/(2\alpha-2)}.$$
	\end{enumerate}
\end{theorem}

\begin{figure}[tbp]
	\centering
	\resizebox{\linewidth}{!}{%
		\begin{tikzpicture}
			\begin{groupplot}[group style={group size=3 by 1,horizontal sep=1.05cm},
				width=5.1cm,height=5.5cm,xmin=0,xmax=1.5,
				axis lines=middle,
				ticks=none,
				xlabel={$t$},ylabel={$J(S_t(u,v))$},
				xlabel style={at={(axis description cs:1.02,.42)},anchor=west},
				ylabel style={at={(axis description cs:0,1.03)},anchor=south},
				tick label style={font=\footnotesize},label style={font=\small},
				title style={font=\small,align=center,yshift=14pt},
				xtick={.5,1,1.5},samples=160,domain=0:1.5,clip=false]
				\nextgroupplot[title={Subcritical regime\\ $0<\alpha<1$},ymin=-.35,ymax=.85,ytick={-.25,.5},
				xlabel style={at={(axis description cs:1.02,.292)},anchor=west}]
				\addplot[phaseblue,very thick] {x^2-x};
				\addplot[only marks,mark=*,mark size=1.8pt,phaseblue] coordinates {(.5,-.25)};
				\nextgroupplot[title={Critical regime \\$\alpha=1$},ymin=-1,ymax=1,ytick={-.8,.8},xtick={1,1.5},
				xlabel style={at={(axis description cs:1.02,.5)},anchor=west}]
				\addplot[phaseblue,very thick] {.4*x^2};
				\addplot[black,densely dashed,thick] {0};
				\addplot[phaseorange,very thick,dashdotted] {-.4*x^2};
				\node[font=\footnotesize,anchor=east] at (axis cs:1.35,.77) {$A>F$};
				\node[font=\footnotesize,anchor=south] at (axis cs:1.35,.035) {$A=F$};
				\node[font=\footnotesize,text=phaseorange,anchor=east] at (axis cs:1.35,-.77) {$A<F$};
				\nextgroupplot[title={Supercritical regime\\$\alpha>1$},ymin=-1.2,ymax=.4,ytick={-1,.25},
				xlabel style={at={(axis description cs:1.02,.75)},anchor=west}]
				\addplot[phaseblue,very thick] {x^2-x^3};
				\addplot[only marks,mark=*,mark size=1.8pt,phaseblue] coordinates {(.6666667,.1481481)};
			\end{groupplot}
		\end{tikzpicture}
	}
	\caption{Qualitative behaviour of the energy on mass-preserving fibres $J(S_t(u,v))=A(u,v)t^2-F(u,v)t^{2\alpha}$, for $t>0$.
	}
	\label{figenergyfibres}
\end{figure}
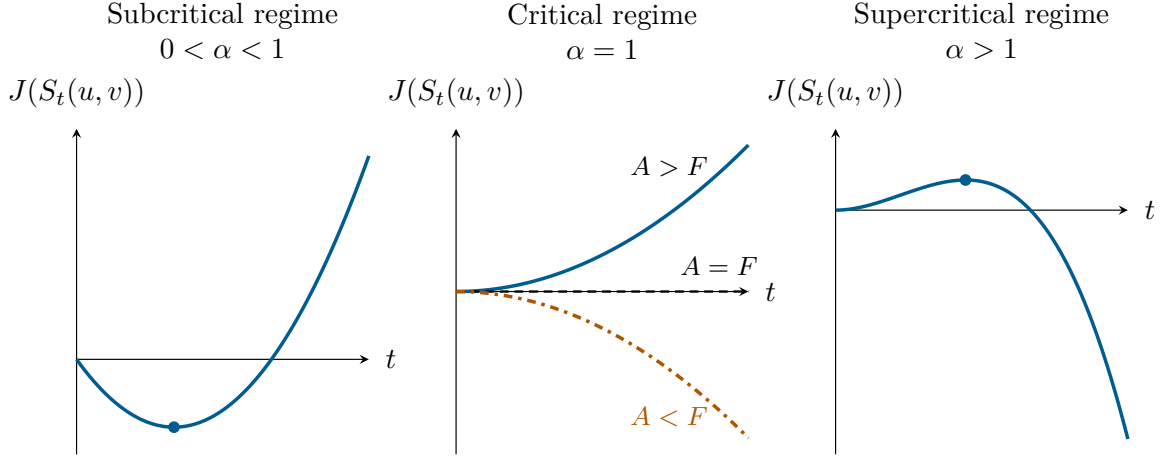

The proof is given in Section~\ref{secgeometry}. 
Figure~\ref{figenergyfibres}
illustrates the different behaviour of the energy on the fibre sets. 
\vskip0.2truecm
Next, we discuss the attainment of the best constant $C_K^*$ in \eqref{C_K*} and its dependence on the parameter $K>0$.
\begin{theorem}[Restricted optimisers and attainment]\label{introattainment}
	For every $a>0$, $\inf_{\Sigma_a}J=-\infty$.
	For any choice of admissible $N,p,q,r,s$, every optimiser of $C_K^*$ has at
	least one normalised component norm in \eqref{MK} equal to $K$. 
	Moreover, for sufficiently large $K$, there exists constants $c,C>0$ such that
	\begin{equation}\label{CK_bounds}
		cK^{\max\{p,q\}}\le C_K^*\le CK^{\max\{p,q\}}\,.
	\end{equation}
	On the critical curve, $a_K$ consequently tends to zero as $K\to+\infty$.
	
	If $2<p,q<2^*$ and $r=s=2$, the minimal nonempty class is
	$\mathcal M_1$. Its best constant $C_1^*$ is attained and equals
	\begin{equation}\label{C1formula}
	 C_1^*=\max_{\|w\|_2=\|\nabla w\|_2=1}\left(\frac1q\int|w|^q+\frac1p\int|w|^p\right).
	\end{equation}
	A maximiser can be chosen positive, radial and nonincreasing, and solves
	\begin{equation}\label{mixedoptimizer}
		-\eta\Delta w+\nu w=w^{q-1}+w^{p-1}
	\end{equation}
	for some $\eta,\nu>0$. Thus the subcritical infimum on $\mathcal M_1^a$ and zero energy at
	its critical threshold $a_1$ are attained. No positive pair in
	$\mathcal M_1$ solves \eqref{system} when $p\ne q$; instead, when
	$p=q=2+\frac4N$, one has $a_1=\|Q_{2+4/N}\|_2^2=a_*$.
	In this Hilbert setting the constants are continuous from the right
	at the minimal bound: for $1\le K\le2$,
	\begin{equation}\label{nearminimalestimate}
		0\le C_K^*-C_1^*\le C\sqrt{K^2-1}\,,
	\end{equation}
	with $C$ depending only on $N,p,q$.
\end{theorem}
Section \ref{Section_attainment} is dedicated to its proof. 
We stress that the restriction $K=1$ concerns only attainment on the entire invariant
class, and does not nullify the validity of the interpolation inequality \eqref{introGN}. In Proposition~\ref{proportionalattainment} we discuss the attainment on the proportional subclass for every $K\ge1$, and provide an explicit formula when $p=q$. Indeed, equality in Cauchy--Schwarz forces the two normalised components to coincide at this value when $K=1$, while this scalar reduction is unavailable for $K>1$. 

\vskip0.2truecm
In addition to the classification result in Theorem \ref{introexistencescaling}, nondegeneracy plays a crucial role in establishing the persistence of solutions under local perturbations of the nonlinearities through implicit function arguments. Next, we extend Theorem \ref{thm:main} studying the nondegeneracy of full constrained system \eqref{system} and obtaining a unique branch of solutions which bifurcates at any positive radial solution of \eqref{system} away from the bilinear mass critical curve.


\begin{theorem}[Fixed-mass nondegeneracy and power continuation]\label{intropowercontinuation}
	In the range \eqref{range}, at any centred positive frequency family
	$(u_\lambda,v_\lambda)$ the radial linearisation
	of \eqref{system} with fixed bilinear mass $a>0$ in the variables $(u,v,\lambda)$ is an isomorphism $(H^1_{\rm rad})^2\times\R\longrightarrow(H^{-1}_{\rm rad})^2\times\R$ if and only if $\gamma_N(p,q)\ne0$. When $\gamma_N(p,q)=0$ its kernel is $\operatorname{span}\{(\partial_\lambda u_\lambda, \partial_\lambda v_\lambda,1)\}$.
	
	Let $a>0$ and $2<p_0,q_0<2^*$ with $\gamma_N(p_0,q_0)\neq0$. Let $(u_0,v_0,\lambda_0)$ be the centred positive triple with bilinear mass $a$ given by Theorem \ref{introexistencescaling}. For any $\ell_1,\ell_2\in(2,2^*)$, there exists a unique local $C^1$ branch near this triple solving
	\begin{equation}\label{perturbedsystem}
		\begin{cases}
			-\Delta u+\lambda u=v_+^{p-1}+\varepsilon_1v_+^{\ell_1-1},\\
			-\Delta v+\lambda v=u_+^{q-1}+\varepsilon_2u_+^{\ell_2-1},\\
			B(u,v)=a\,,
		\end{cases}
	\end{equation}
	for $(p,q,\varepsilon_1,\varepsilon_2)$ near $(p_0,q_0,0,0)$.
	For $\varepsilon_1,\varepsilon_2\ge0$ the solutions are positive,
	radial, classical and exponentially decaying, with $\lambda>0$.
	Uniqueness is local in $(H^1_{\rm rad})^2\times\R$.
\end{theorem}
The fixed-mass linearisation and the continuation argument are presented in Section \ref{Section_branches_powers}. The individual Sobolev bounds, although not required for the linear nondegeneracy assertion, allow to prove the bifurcation result in $H^1$ spaces.

	Let us come back to the 
	frequency-scaling exponent $\gamma$ defined in \eqref{parameters}.
	Theorem \ref{introexistencescaling} gives the exact law
	\begin{equation*}
		B(u_\lambda,v_\lambda)=\lambda^\gamma B(U,V)\,.
	\end{equation*}
	This mass law has an interesting interpretation. 
	Set
	$a(\lambda)=B(u_\lambda,v_\lambda)$ and
	$\mathcal D(\lambda)=I_\lambda(u_\lambda,v_\lambda)$ along a positive
	frequency family. Since $I_\lambda'(u_\lambda,v_\lambda)=0$,
	differentiation gives
	\begin{equation}\label{introactionresponse1}
		\mathcal D'(\lambda)=\partial_\lambda I_\lambda(u_\lambda,v_\lambda)+\langle I'_\lambda(u_\lambda,v_\lambda),(\nabla u_\lambda,\nabla v_\lambda)\rangle=B(u_\lambda,v_\lambda)=a(\lambda)\,,
	\end{equation}
	hence
	\begin{equation}\label{introactionresponse2}
		\mathcal D''(\lambda)=a'(\lambda)=\frac{\gamma a(\lambda)}{\lambda}\,.
	\end{equation}
	Off criticality the inverse response is
	$\mathrm d\lambda/\mathrm da=\lambda/(\gamma a)$: prescribing the
	charge determines the frequency, and therefore the spatial scale of
	the wave. At criticality the charge is insensitive to changes of
	frequency, so this selection mechanism degenerates. Theorem
	\ref{intropowercontinuation} identifies the corresponding
	kernel in the constrained linearisation. Away from that degeneracy,
	our continuation result quantifies the persistence of centred profiles
	under power perturbations (see also the planar case of exponential perturbations treated in Theorem \ref{introexponentialcontinuation} and Proposition \ref{introcriticalresponse} below).
	
	In scalar solitary-wave theory, the charge--frequency slope and the
	spectrum of the linearised action enter the stability analysis
	\cite{Weinstein1986,GrillakisShatahStrauss1987}. Formulas \eqref{introactionresponse1}-\eqref{introactionresponse2}
	and our nondegeneracy result identify analogous stationary data for the
	crossed problem. 
	\vskip0.2truecm
	To complete the picture, in Section \ref{Section_massresponse} we analyse how the bilinear mass of branches of solutions of \eqref{perturbedsystem} behaves at bilinear-mass critical exponents $(p,q)$. As usual, in the next result $(U,V)$ denotes the centred positive solution of \eqref{system1}, while $(u_\lambda,v_\lambda)$ its rescaled counterpart at frequency $\lambda$ via the scaling \eqref{introfrequencyscaling}, whose mass is $a_*=\int UV$.
	
	\begin{proposition}[Mass response in the critical case]\label{introcriticalresponseNgeq3}
		Let $N\ge2$, $2<p,q<2^*$ such that $\gamma_N(p,q)=0$, and $2<\ell_1,\ell_2<2^*$. For every $\lambda_0>0$, the perturbed system \eqref{perturbedsystem}, without the mass constraint, has a unique local $C^1$ radial branch $(\lambda,\varepsilon_1,\varepsilon_2)\to(u_{\lambda,\varepsilon},v_{\lambda,\varepsilon})$ near $(\lambda_0,0,0)$, satisfying $(u_{\lambda,0},v_{\lambda,0})=(u_\lambda,v_\lambda).$
		Moreover, as $\varepsilon=(\varepsilon_1,\varepsilon_2)\to0$, its mass satisfies
		\begin{equation}\label{criticalexpansion}
			B(u_{\lambda,\varepsilon},v_{\lambda,\varepsilon})=a_*-\varepsilon_1\frac{e_{\ell_1}^{(v)}}{\ell_1\lambda}\int v_\lambda^{\ell_1}-\varepsilon_2\frac{e_{\ell_2}^{(u)}}{\ell_2\lambda}\int u_\lambda^{\ell_2}+o(|\varepsilon|)\,,
		\end{equation}
		locally uniformly for $\lambda$ near $\lambda_0$, where
		\[
		e_\ell^{(v)}=\frac{\ell q}{D}-\frac N2\,,\qquad
		e_\ell^{(u)}=\frac{\ell p}{D}-\frac N2\,.
		\]
	\end{proposition}

\vskip0.2truecm

\paragraph{\textbf{The planar exponential counterpart.}}

For problems in dimension two, the classical Trudinger--Moser permits to consider nonlinearities with exponential growth. However, note that the scalar threshold does not exhaust the range of admissible nonlinearities for planar Hamiltonian systems. Indeed, Cassani and Tarsi obtained in \cite{CassaniTarsi2015} finite-energy linking solutions with growth beyond this scalar growth, using complementary Lorentz--Sobolev spaces: faster exponential growth in one component can be balanced by slower growth in the other, see also Remark~\ref{supercriticalLorentzscope} below. As a first contribution, here we restrict to consider the quadratic-exponential Trudinger--Moser model which can be treated within the standard Hilbert framework, namely 
\begin{equation}\label{system_exp}
	\begin{cases}
		-\Delta u+\lambda u=v(e^{\vartheta_1v^2}-1)\quad&\text{in }\ \R^2\,,\\
		-\Delta v+\lambda v=u(e^{\vartheta_2u^2}-1)\quad&\text{in }\ \R^2\,,\\
		B(u,v)=\displaystyle\int uv=a\,,
	\end{cases}
\end{equation}
for $\vartheta_1,\vartheta_2,a>0$. This system plays the role of a planar exponential counterpart of \eqref{system}. For Hamiltonian systems on the whole plane with a fixed frequency or potential, Cassani and Zhang established in \cite{CassaniZhang2017}, under critical-growth
conditions, compactness modulo translations of the ground-state set, uniform bounds and decay, and, with additional
assumptions, positivity and radial symmetry. 
Our prescribed-mass problem requires instead a control of the
unknown frequency and recovery of the bilinear constraint $B$ along
variational sequences; these issues motivate the exponential
interpolation and compactness estimates which we investigate in this paper.
\vskip0.2truecm

When $N=2$, the bilinear Gagliardo-Nirenberg inequality \eqref{introGN} remains valid with
exponents $\alpha=\left(\frac1p+\frac1q\right)^{-1}-1$ and $\beta=1$. Exponential integrability provided by Ruf's theorem and its rescaled form \cite{Ruf,CST}, however, suggests a bilinear mass-weighted limiting formulation. Let
\begin{equation}\label{Rufconstant}
	C_{\rm R}=\sup_{\|w\|_{H^1(\R^2)}^2\le1}\int_{\R^2}(e^{4\pi w^2}-1)<\infty\,,
\end{equation}
where finiteness is provided by Ruf in \cite{Ruf}. For $K\ge1$, let $\mathcal E_K$ consist of pairs $(u,v)\in H^1(\R^2)^2$ with $A(u,v)>0$, $B(u,v)>0$ and
\[
 \|\nabla u\|_2^2,\|\nabla v\|_2^2\le K^2A(u,v)\quad\ \mbox{and}\ \quad \|u\|_2^2,\|v\|_2^2\le K^2B(u,v)\,.
\]
These quadratic component bounds are invariant under common amplitude changes and spatial dilations. As a planar counterpart of Theorem \ref{introinterpolation}, we first establish new Gagliardo--Nirenberg-type inequalities for Trudinger-Moser nonlinearities in the class $\cE_K$.

\begin{theorem}[Exponential Gagliardo--Nirenberg inequalities]\label{introexponentialGN}
	Let $N=2$, $K\ge1$ and $\vartheta>0$. For $(u,v)\in\mathcal E_K$
	with $\vartheta K^2A(u,v)<4\pi$, one has
	\begin{equation}\label{introexponentialbound}
		\int\big[(e^{\vartheta u^2}-1)+(e^{\vartheta v^2}-1)\big]\le\frac{2C_{\rm R}\vartheta K^2B(u,v)}{4\pi-\vartheta K^2A(u,v)}\,.
	\end{equation}
	The threshold is sharp: at fixed $A=A_0>0$, $B=B_0>0$, the
	supremum of the left-hand side over $\mathcal E_K$ is finite if and
	only if $\vartheta K^2A_0<4\pi$.
	
	Set now $\Psi(s)=e^s-1-s$. For $0<\rho<1$ and $\vartheta K^2A(u,v)\le4\pi\rho$,
	\begin{equation}\label{introtruncatedbound}
		\int[\Psi(\vartheta u^2)+\Psi(\vartheta v^2)]\le C_\rho\vartheta^2K^4A(u,v)B(u,v)\quad\ \mbox{with}\quad\ C_\rho=\frac{C_{\rm R}}{2\pi^2(1-\rho^2)}\,.
	\end{equation}
\end{theorem}

The two statements of Theorem \ref{introexponentialGN} are proved separately in Sections \ref{secTMbilinear} and \ref{secTMtruncated}. 
Observe that taking $\vartheta=\theta/A$ in \eqref{introexponentialbound} gives a bound proportional to $B/A$, for $0<\theta<4\pi/K^2$. The truncated estimate has the quartic mass-critical scaling $AB$. 

Let us first comment on the consequences of Theorem \ref{introexponentialGN} in the scalar case. Set $u=v=w$ and $K=1$. 
The untruncated mass-weighted estimate is the classical rescaled inequality by
Ruf \cite{Ruf} and Adachi--Tanaka \cite{AdTa}, see \cite{CST}. For the nonlinear potential we obtain the following scalar refinement.
\begin{corollary}[Scalar exponential GN with the sharp quartic coefficient]\label{scalartruncatedrefinement}
	Let $\vartheta>0$ and $0<\rho<1$. Then
	\begin{equation}\label{introscalartruncated}
		\int\Psi(\vartheta w^2)\le\frac{C_\rho}{2}\vartheta^2\|w\|_2^2\|\nabla w\|_2^2\,,
		\qquad \vartheta\|\nabla w\|_2^2\le4\pi\rho<4\pi\,.
	\end{equation}
	It retains the same mass--gradient product as the familiar planar
	quartic inequality $\int w^4\le(2/\|Q_4\|_2^2)\|w\|_2^2\|\nabla w\|_2^2$, where $Q_4$ is defined in \eqref{Q4}. 
	The optimal quartic leading coefficient is isolated by
	\begin{equation}\label{introscalarquartic}
		\int\Psi(\vartheta w^2)\le\frac{\vartheta^2}{\|Q_4\|_2^2}\|w\|_2^2\|\nabla w\|_2^2+\frac{C_{\rm R}\vartheta^3}{8\pi^3}\| w\|_2^2\|\nabla w\|_2^4\,,
		\qquad\vartheta \|\nabla w\|_2^2\le2\pi\,.
	\end{equation}
	The coefficient $1/\|Q_4\|_2^2$ is optimal among estimates with a
	remainder $O(\|w\|_2^2\|\nabla w\|_2^4)$ as $\|\nabla w\|_2\to0$, at fixed positive mass and
	fixed $\vartheta$.
\end{corollary}
The proof is given in Section~\ref{secTMtruncated}. These scalar consequences of classical exponential integrability
are stated separately because they provide the bridge from scalar
GN energy estimates to our bilinear exponential theory. We do not
claim a new scalar embedding theorem: the new system formulation,
its sharp gradient threshold on $\mathcal E_K$, and the application
of the quartic remainder to the full Hamiltonian problem are the
key differences from the scalar case. 

\vskip0.2truecm
The exponential Gagliardo-Nirenberg inequalities in Theorem \ref{introexponentialGN} are an essential ingredient to prove existence of solutions to the exponential system \eqref{system_exp} with prescribed bilinear mass, while we refer to Proposition~\ref{exponentialcomparison} for pointwise properties of $(u,v)$ and mass bounds.

\begin{theorem}[Exponential solutions below the cubic limiting mass]\label{introfullfibreexistence}
	Let $\vartheta_1,\vartheta_2>0$; for every
	\begin{equation}\label{pureexpsmallmass}
		0<a<a_{\rm c}:=\frac{\|Q_4\|_2^2}{\sqrt{\vartheta_1\vartheta_2}}
	\end{equation}
	there exist $\lambda>0$ and positive radial classical exponentially
	decaying $u,v\in H^1(\R^2)$ satisfying
	\begin{equation}\label{TMpurenormalisedsystem}
		\begin{cases}
			-\Delta u+\lambda u=v(e^{\vartheta_1v^2}-1)\quad&\text{in }\ \R^2\,,\\
			-\Delta v+\lambda v=u(e^{\vartheta_2u^2}-1)\quad&\text{in }\ \R^2\,,\\
			B(u,v)=a\,.
		\end{cases}
	\end{equation}
	The solution has least energy among positive radial solutions of this
	mass, and its energy level $c_a$ satisfies
	\begin{equation}\label{estimate_energy_level}
		0<\beta_a\le c_a\le\overline c<\frac{4\pi}{\sqrt{\vartheta_1\vartheta_2}}\,.
	\end{equation}
\end{theorem}

The main ingredients of its proof are essentially two: a linking geometry on the bilinear-mass constraint, and compactness of Poho\v zaev--Cerami sequences. These results are the core of the proof of Theorem \ref{introfullfibreexistence}, hence we state them separately here in Theorems \ref{introtildelinking} and \ref{introtildecompactness} below, which are proved in Section \ref{sectildelinking} and \ref{sectildecompactness}, respectively. The proof of Theorem \ref{introfullfibreexistence} is then provided in Section \ref{secfullfibre}.
\vskip0.2truecm
For $\vartheta_1,\vartheta_2>0$ we define
\begin{equation}\label{TMnonlinearity}
	h_\vartheta(t)=t_+(e^{\vartheta t_+^2}-1)\quad\mbox{and}\quad H_\vartheta(t)=\frac{\Psi(\vartheta t_+^2)}{2\vartheta}=\frac{e^{\vartheta t_+^2}-1-\vartheta t_+^2}{2\vartheta}\,,
\end{equation}
so that $H_\vartheta'=h_\vartheta$. Set also
\begin{equation}\label{TMpurefunctionals}
	\begin{aligned}
		F_{\rm e}(u,v)&=\int[H_{\vartheta_2}(u)+H_{\vartheta_1}(v)]\,,
		&\quad S_{\rm e}(u,v)&=\int[u h_{\vartheta_2}(u)+v h_{\vartheta_1}(v)]\,,\\
		J_{\rm e}(u,v)&=A(u,v)-F_{\rm e}(u,v)\,,
		&\quad P_{\rm e}(u,v)&=2A(u,v)-S_{\rm e}(u,v)+2F_{\rm e}(u,v)\,.
	\end{aligned}
\end{equation}
First, the tilde-map method in the formulation developed by Cassani and Tarsi \cite{CassaniTarsi2015} (see Section \ref{sectildepreliminaries}) leads to a linking geometry on the mass constraint itself. The next result settles the geometry and the upper level bound for finite-dimensional negative approximations. 
\begin{theorem}[Mass-preserving tilde linking]\label{introtildelinking}
	Let $\vartheta_0=\sqrt{\vartheta_1\vartheta_2}$ and choose
	$0<\rho<1$, $a>0$ such that
	\begin{equation}\label{tildebarrierconstants}
		\delta=1-\frac{C_\rho(\vartheta_1+\vartheta_2)a}4>0\,,\qquad R_0=\frac{4\pi\rho}{\vartheta_0}\,,\qquad \beta_0=\delta R_0\,.
	\end{equation}
	For any finite-dimensional radial negative space $W_m$, the
	mass-preserving chart \eqref{tildechart} induces a linking class
	on $\Sigma_{a,m}=\mathcal X_a(\mathbb S\times W_m)$, where $\mathcal X_a$ is defined in \eqref{tildechart}, with minimax
	level $c_m$ satisfying
	\[
	0<\beta_0\le c_m\le\overline c<\frac{4\pi}{\vartheta_0}\,.
	\]
	The cylinder parameters and $\overline c$ are independent of $m$. Moreover, $J_{\rm e}|_{\Sigma{a,m}}$ admits a Cerami sequence, hence a Palais--Smale sequence, at level $c_m$.
\end{theorem}
An important step of the linking construction is to find a suitable upper bound for the linking level by means of the Moser function, see Lemma~\ref{tildemoser} below. 
The chart $\mathcal X_a$ preserves $B(u,v)=a$ throughout the linking deformation.

The following result provides uniform compactness under the same upper level.
\begin{theorem}[Uniform compactness for Cerami-Poho\v zaev sequences]\label{introtildecompactness}
	Let $a,\vartheta_1,\vartheta_2>0$, $\vartheta_0=\sqrt{\vartheta_1\vartheta_2}$,
	and $0<c<4\pi/\vartheta_0$. Choose nested finite-dimensional
	$W_n\subset C^\infty_{c,\rm rad}(\R^2)$ with dense union in
	$H^1_{\rm rad}(\R^2)$, and let $E_n$ be the spaces defined in \eqref{tildeGalerkinspace}. Suppose $(u_n,v_n)\in E_n$
	and $\lambda_n\ge0$ satisfy
	\begin{equation}\label{tildeCPconditions}
		\begin{gathered}
			B(u_n,v_n)=a\,,\qquad J_{\rm e}(u_n,v_n)\to c\,,\qquad P_{\rm e}(u_n,v_n)\to0\,,\\
			(1+\|(u_n,v_n)\|_{H^1\times H^1})\left\|\left.(J_{\rm e}'(u_n,v_n)+\lambda_nB'(u_n,v_n))\right|_{E_n}\right\|_{E_n^*}\to0\,.
		\end{gathered}
	\end{equation}
	Then $(u_n,v_n)$ is uniformly bounded in $H^1(\R^2)^2$, and there exist $\lambda_\pm>0$ such that $\lambda_-<\lambda_n<\lambda_+$. Moreover, up to a subsequence, $\lambda_n\to\lambda$ and $(u_n,v_n)$ converges strongly in $H^1(\R^2)^2$ to a solution of \eqref{TMpuresystem} of mass $a$ and energy $c$, which is positive, radial, classical, and exponentially decaying.
\end{theorem}
At the concentration level $c=\frac{4\pi}{\vartheta_0}$, Corollary~\ref{tildeendpoint} proves compactness provided the multipliers stay away from zero. The positive weak limit consumes a strictly positive amount of energy, leaving a subthreshold gradient defect. Thus zero-frequency degeneration is the only remaining obstruction at that level under the 
Poho\v zaev--Cerami conditions \eqref{tildeCPconditions}.

The proof of Theorem \ref{introtildecompactness} is given in Section \ref{sectildecompactness}.  
Using the tilde spaces and the Cerami-Poho\v zaev conditions \eqref{tildeCPconditions}, 
we first bound $(\lambda_n)_n$ and $(u_n,v_n)$ in $H^1(\R^2)^2$. A quartic estimate under the Trudinger–Moser gradient threshold is used to excludes vanishing multipliers. Radial compactness and truncation arguments then allow us to pass to the limit in the equations and show that the limit is a positive solution with the prescribed mass. To prove strong convergence, we compare the limiting quadratic identities and use the equality case in the arithmetic–geometric mean inequality to identify the two weighted gradient remainders. The resulting strict bound below $4\pi$ permits an application of the Adachi–Tanaka inequality, yielding the exponential integrability needed to pass to the limit and conclude strong convergence in $H^1(\R^2)^2$.

\vskip0.2truecm
Theorem \ref{introfullfibreexistence} provides existence for the exponential system \eqref{system_exp} for bilinear masses in the full range $(0,a_c)$, where the threshold parameter $a_c$, defined in \eqref{pureexpsmallmass}, depends on the $L^2$-mass of the Townes profile $Q_4$, namely the ground state of the cubic planar Schr\"odinger equation \eqref{Q4}. Next, we investigate the behaviour of problem \eqref{system_exp} for bilinear masses close to $a_c$ from the left, branching exactly from $Q_4$: note indeed that $h_\theta(t)\sim\vartheta t^3$ as $t\to0^+$.
	
	\begin{theorem}[Pure exponential normalised solutions]\label{intropureexponential}
		Let $\vartheta_1,\vartheta_2>0$ and set $\tau=(\vartheta_1/\vartheta_2)^{1/4}$. 
		For all sufficiently small $\lambda>0$ there exists a positive radial
		classical exponentially decaying branch of solutions of
		\begin{equation}\label{TMpuresystem}
			\begin{cases}
				-\Delta u+\lambda u=v(e^{\vartheta_1v^2}-1)\quad&\text{in }\ \R^2\,,\\
				-\Delta v+\lambda v=u(e^{\vartheta_2u^2}-1)\quad&\text{in }\ \R^2\,.
			\end{cases}
		\end{equation}
		Its mass extends $C^1$ to zero frequency and satisfies
		\begin{equation}\label{TMpuremass}
			B(u_\lambda,v_\lambda)=a_{\rm c}-\kappa\lambda+o(\lambda)\,,
			\qquad \frac\dd{\dd\lambda}B(u_\lambda,v_\lambda)\to-\kappa
		\end{equation}
		as $\lambda\to0$, where
		$$\kappa=\frac{(\tau+\tau^{-1})}{6\sqrt{\vartheta_1\vartheta_2}}\int Q_4^6>0\,.$$
		Hence every $a\in(a_{\rm c}-\delta,a_{\rm c})$, for some $\delta>0$,
		is realised, with $\lambda(a)=(a_{\rm c}-a)/\kappa+o(a_{\rm c}-a)$.
		Uniqueness holds locally among small-frequency solutions whose cubic
		rescalings \eqref{TMpurescaling} are close to $(Q_4,Q_4)$ in $H^2_{\rm rad}$.
	\end{theorem}
	Its proof is given in Section \ref{Section_11.2}. This limit connects the exponential theory with planar critical self-focusing: after balancing the two cubic coefficients, the	rescaled components converge to the same Townes profile $Q_4$. The original profiles have amplitudes of order $\sqrt\lambda$ and
	widths of order $\lambda^{-1/2}$, retaining a finite limiting bilinear
	mass. The next exponential term breaks the cubic scale invariance
	and gives the strictly negative mass slope in \eqref{TMpuremass}.
	Thus the expansion measures how higher-order focusing changes the
	charge needed to sustain a nearby localised profile.
	The connection with the optical cubic problem is through this controlled small-amplitude limit.
	
	The same cubic profile determines the optimal leading coefficient in
	\eqref{introscalarquartic} and the variational mass range in
	Theorem~\ref{introfullfibreexistence}. The exponential GN refinement
	therefore connects the local response near the cubic limit to the
	energy estimates used to construct solutions. Note that the local branch theorem
	does not require small exponential coefficients. 
	
	In the next result, by a local mapping argument we find a positive branch at fixed prescribed mass, starting from solutions of \eqref{system}, when the power nonlinearities (which do not lie on the mass-critical curve \eqref{criticalcurve}) are perturbed by small exponential critical terms. On the other hand, at criticality the augmented linearisation is singular, but its fixed-frequency part remains invertible, and it is still possible to analyse the resulting mass response. Hence, our last two results, the proof of which can be found in Sections~\ref{Section_exp_response}-\ref{Section_exp_response}, are the planar exponential counterparts of Theorem \ref{intropowercontinuation} and Proposition \ref{introcriticalresponseNgeq3}.
	
	\begin{theorem}[Normalised exponential continuation]\label{introexponentialcontinuation}
		Let $p_0,q_0>2$, $\frac1{p_0}+\frac1{q_0}\neq\frac12$, $a>0$ and
		$\vartheta_1,\vartheta_2>0$. Near the centred positive triple $(\lambda_0,u_0,v_0)$ solving \eqref{system} given by Theorem \ref{introexistencescaling}, there exists a unique local
		$C^1$ branch solving
		\begin{equation}\label{TMsystem}
			\begin{cases}
				-\Delta u+\lambda u=v_+^{p-1}+\varepsilon_1h_{\vartheta_1}(v)\quad&\text{in }\ \R^2\,,\\
				-\Delta v+\lambda v=u_+^{q-1}+\varepsilon_2h_{\vartheta_2}(u)\quad&\text{in }\ \R^2\,,\\
				B(u,v)=a\,,
			\end{cases}
		\end{equation}
		for parameters $(p,q,\varepsilon_1,\varepsilon_2)$ near $(p_0,q_0,0,0)$. For nonnegative perturbation
		coefficients the branch consists of positive radial classical solutions
		with positive frequency and exponential decay. Uniqueness is local in
		$(H^1_{\rm rad}(\R^2))^2\times\R$.
	\end{theorem}
	The argument of the proof uses local differentiability near bounded profiles, see Section \ref{Section_11.3}. Both continuation neighbourhoods may depend on the mass, starting exponents, and perturbing nonlinearities.

	\begin{proposition}[Exponential mass response on the critical curve]\label{introcriticalresponse}
		Let $p,q>2$ such that $\frac1{p_0}+\frac1{q_0}=\frac12$, and $\vartheta_1,\vartheta_2>0$. Let $(U,V)$ be the centred positive unit-frequency profile given by Theorem \ref{thm:main}, and $(u_\lambda,v_\lambda)$ be given by \eqref{introfrequencyscaling}. Set
		\[
		a_*=\int UV\,,\qquad s_u=\frac pD=\frac2q\,,\qquad
		s_v=\frac qD=\frac2p\,,\qquad s_u+s_v=1\,.
		\]
		Near every $(\lambda_0,0,0)$ with $\lambda_0>0$, the first two
		equations of \eqref{TMsystem}, without the mass constraint, admit a
		unique local $C^1$ radial branch
		$(\lambda,\varepsilon_1,\varepsilon_2)\mapsto(u_{\lambda,\varepsilon},v_{\lambda,\varepsilon})$
		through $(u_{\lambda_0},v_{\lambda_0})$. Its mass satisfies
		\begin{equation}\label{TMexpansion}
			B(u_{\lambda,\varepsilon},v_{\lambda,\varepsilon})=a_*-\varepsilon_1D_{\vartheta_1,s_v}[V](\lambda)-\varepsilon_2D_{\vartheta_2,s_u}[U](\lambda)+o(|\varepsilon|)
		\end{equation}
		locally uniformly in $\lambda$, where
		\begin{equation}\label{TMcoefficient}
			D_{\vartheta,s}[W](\lambda)=\frac1{2\vartheta\lambda^2}\int\big[2s z(e^z-1)-\Psi(z)\big]\,,\qquad z=\vartheta\lambda^{2s}W^2.
		\end{equation}
		For nonnegative perturbation coefficients the solutions are positive,
		classical and exponentially decaying. If $p,q\le8$, both
		coefficients in \eqref{TMexpansion} are strictly positive for every
		$\lambda>0$. For $p>8$, the coefficient $D_{\vartheta_1,s_v}[V](\lambda)$ is
		negative at sufficiently small positive frequencies; the analogous
		statement holds for $D_{\vartheta_2,s_u}[U](\lambda)$ when $q>8$.
	\end{proposition}
	
	Observe that, if $p,q\leq8$ sufficiently small nonzero nonnegative perturbations give mass strictly below $a_*$, and no radial solution of mass $a_*$ lies in the local solution neighbourhood. Moreover, the expansion \eqref{TMexpansion} holds on the whole critical curve, whereas a mass decrease for all nonnegative perturbations does not.
\vskip0.2truecm

\paragraph{\textbf{Organisation of the paper.}}

After some preliminary results collected in Section \ref{Section_preliminary}, we start by discussing existence, decay, and uniqueness for the unit-frequency problem \eqref{system1} in Section \ref{Section_system1}; in particular, Theorem \ref{thm:main} will be proved here. The scaling connection between solutions of \eqref{system1} and normalised solutions of \eqref{system} will be investigated in Section \ref{Section_normalised}, leading to the classification result in Theorem \ref{introexistencescaling}. In Section \ref{Section_GN} we identify a new bilinear Gagliardo-Nirenberg inequality and prove Theorem \ref{introinterpolation}, whose variational consequences described in Theorem \ref{introenergytrichotomy} will be proved in Section \ref{secgeometry}. Section \ref{Section_attainment} is devoted to the study of the best constant $C_K^*$ (Theorem \ref{introattainment}), while in Section \ref{Section_branches_powers} we present the proof of the bifurcation results in Theorem \ref{intropowercontinuation} and Proposition \ref{introcriticalresponseNgeq3}. The remaining Sections~\ref{secTM}-\ref{secTMsolutions_ac} are devoted to the planar case. The exponential Gagliardo-Nirenberg inequality Theorem~\ref{introexponentialGN} is proved in Subsections~\ref{secTMscalar}--\ref{secTMtruncated}. Section \ref{secTMsolutions} we prove existence results for the planar normalised system with exponential nonlinearities \eqref{system_exp}. Firstly, Subsection~\ref{Section_9.4} contains some consequences of Theorem in this setting; the existence of solutions with prescribed bilinear mass $a\in(0,a_c)$ in Theorem \ref{introfullfibreexistence} is proved in Subsections \ref{sectildelinking} and \ref{sectildecompactness}, where the linking construction and the compactness criterion are provided, respectively. Finally, the bifurcation analysis near the mass threshold $a_c$ of Theorems \ref{intropureexponential}-\ref{introcriticalresponse} is contained in Section~\ref{secTMsolutions_ac}.
\vskip0.2truecm

\paragraph{\textbf{Notation.}}
We denote by $t_+=\max\{t,0\}$ and $t_-=\max\{-t,0\}$ the negative and positive parts, respectively, so both
are nonnegative. Norms $\|\cdot\|_r$ refer to $L^r(\R^N)$. We define $H^1_{\rm rad}=\{w\in H^1(\R^N):w\text{ is radial}\}$
and denote its dual by $H^{-1}_{\rm rad}$. Similarly, $C_{0,\mathrm{rad}}(\R^N)=\{w\in C_0(\R^N):w\text{ is radial}\}$, with the norm inherited from $C_0(\R^N)$. Hereafter, the letter $C$ will be used to denote positive constants which are independent of relevant quantities and whose value may change from line to line.

\section{Preliminaries}\label{Section_preliminary}

\subsection{A conjugate Sobolev setting}
For $r,s>1$ set
\[
 X=W^{1,r}(\R^N)\times W^{1,s}(\R^N)\,.
\]
\begin{lemma}\label{setting}
	Under \eqref{range}, there exist conjugate exponents $r,s>1$ with
	\begin{equation*}
		r<q\,,\quad s<p\,,\quad 0<N\left(\frac1r-\frac1q\right)<1\,,\quad0<N\left(\frac1s-\frac1p\right)<1\,.
	\end{equation*}
	For every such choice, the embeddings $W^{1,r}(\R^N)\hookrightarrow L^q(\R^N)$ and
	$W^{1,s}(\R^N)\hookrightarrow L^p(\R^N)$ are continuous. The maps $J,B:X\to\R$
	are $C^1$, and $\Sigma_a$ is a $C^1$ hypersurface for $a>0$.
	A critical point of $J|_{\Sigma_a}$ solves \eqref{system} weakly for
	some $\lambda\in\R$.
\end{lemma}

\begin{proof}
	Choose $r>1$ such that
	$$\max\left\{\frac1q,1-\frac1p-\frac1N\right\}<\frac1r<\min\left\{\frac1q+\frac1N,1-\frac1p\right\}$$
	and set $\frac1s=1-\frac1r$. The interval is nonempty by \eqref{range} and $p,q>2$.
	Set $\theta_1=N\big(\frac1r-\frac1q\big)$ and $\theta_2=N\big(\frac1s-\frac1p\big)$.
	The scalar Gagliardo--Nirenberg inequality \eqref{scalarGN} then gives
	\begin{equation}\label{GNrs}
	 \|u\|_q\le C_{N,r,q}\|\nabla u\|_r^{\theta_1}\|u\|_r^{1-\theta_1}\,,
	 \qquad
	 \|v\|_p\le C_{N,s,p}\|\nabla v\|_s^{\theta_2}\|v\|_s^{1-\theta_2}\,.
	\end{equation}
	H\"older's inequality controls both bilinear forms, and the continuous
	Lebesgue embeddings give the asserted differentiability. Since
	$B'(u,v)[u,0]=a$ on $\Sigma_a$, its derivative is nonzero.
	Writing the multiplier identity as $J'+\lambda B'=0$ and testing
	separately in the two components gives \eqref{system}.
\end{proof}

For $\lambda>0$, we define a weak solution of \eqref{system} a pair $(u,v)\in X$ whose nonlinearities are
locally integrable and which satisfies
\[
\int(\nabla u\cdot\nabla\varphi+\lambda u\varphi)
=\int|v|^{p-2}v\varphi\,,\qquad
\int(\nabla v\cdot\nabla\psi+\lambda v\psi)
=\int|u|^{q-2}u\psi
\]
for all $\varphi,\psi\in C_c^\infty(\R^N)$.

\vskip0.2truecm
Let us next justify the choice of our constraint and its connection with the equations in \eqref{system}, and compare our choice with the ones in the literature.

\begin{remark}
\label{quadraticconstraint}
	For a smooth compactly supported pair $(u,v)$, define
	\[
	 C(u,v):=\frac12\int(c_{11}u^2+2c_{12}uv+c_{22}v^2)\,.
	\]
	The identity $J'+\lambda C'=0$ then gives
	\[
	 \begin{cases}
	 -\Delta u+\lambda(c_{12}u+c_{22}v)=|v|^{p-2}v\,,\\
	 -\Delta v+\lambda(c_{11}u+c_{12}v)=|u|^{q-2}u\,.
	 \end{cases}
	\]
	To obtain a common nonzero diagonal frequency for arbitrary pairs, the
	coefficients must satisfy $c_{11}=c_{22}=0$ and $c_{12}\ne0$;
	absorbing $c_{12}$ into the multiplier gives precisely $C=B$ in \eqref{bilinear_mass}.
	Thus bilinear mass is forced within this class of quadratic constraints.
	The choice, instead, of a total mass \eqref{total_mass} as in \cite{MaiaMiyagaki} corresponds to $c_{11}=c_{22}=1$, $c_{12}=0$ and therefore gives
	\eqref{totalmassEL}. Separate component masses produce two crossed
	multiplier terms for this same energy, hence a
	different variational structure. Note that on $u=v$ the two single constraints
	coincide, so the scalar reduction can therefore hide the distinction.
\end{remark}

\subsection{Poho\v zaev-type identities}
\begin{proposition}\label{pohozaev}
	Let $(u,v)\in\cS$ be a positive solution for \eqref{system} with
	frequency $\lambda>0$. Write $A=A(u,v)$ and $B=B(u,v)$. Then
	\begin{equation}\label{testing}
		\int u^q=\int v^p=:H=A+\lambda B\,,
	\end{equation}
	\begin{equation}\label{poho}
		(N-2)A+N\lambda B=N\sigma H\,.
	\end{equation}
	Consequently,
	\begin{align}
		A&=\frac{N(1-\sigma)}2 H>0\,,\label{Aidentity}\\
		\lambda B&=\frac{N\sigma-(N-2)}2 H>0\,,\label{Bidentity}\\
		J(u,v)&=\frac{N-(N+2)\sigma}{2}H\,.\label{Jidentity}
	\end{align}
\end{proposition}
\begin{proof}
	The identity \eqref{testing} is obtained by testing the first equation with $v$ and the second with $u$. To get \ref{poho}, multiply the first equation in \eqref{system} by $x\cdot\nabla v$ and the second by $x\cdot\nabla u$. Integration by parts gives
	$$\int(-\Delta u)(x\cdot\nabla v)=\int\nabla u\cdot\nabla v+\int\nabla u\cdot(D^2v\,x)\,,$$
	and, analogously,
	$$\int(-\Delta v)(x\cdot\nabla u)=\int\nabla u\cdot\nabla v+\int\nabla v\cdot(D^2u\,x)\,.$$
	Hence
	$$I_1=2A+\int\big[\nabla u\cdot(D^2v\,x)+\nabla v\cdot(D^2u\,x)\big]=2A+\int x\cdot\nabla(\nabla u\cdot\nabla v)\,.$$
	Therefore, using
	$$\int x\cdot\nabla F=-N\int F\,,$$
	we find
	\begin{equation}\label{poho1}
		I_1=2A-NA=(2-N)A\,.
	\end{equation}
	Similarly,
	\begin{equation}\label{poho2}
		\int u(x\cdot\nabla v)+\int v(x\cdot\nabla u)=\int x\cdot\nabla(uv)=-N\int uv=-NB\,,
	\end{equation}
	and
	\begin{equation}\label{poho3}
		\int v^{p-1}(x\cdot\nabla v)+\int u^{q-1}(x\cdot\nabla u)
		=-N\int\left(\frac{v^p}p+\frac{u^q}q\right)=-N\sigma H\,.
	\end{equation}
	Combining \eqref{poho1}-\eqref{poho3} proves \eqref{poho}. We stress the fact that, although all above computations are formal, they can be made however rigorous by using with cutoff functions equal to $1$ on $B_R$ and supported in $B_{2R}$, since decay makes every cutoff error vanish as $R\to+\infty$. 
	Subtracting \eqref{poho} from $N$ times \eqref{testing} gives \eqref{Aidentity}, while \eqref{Bidentity} and \eqref{Jidentity} follow algebraically.
\end{proof}


\subsection{The tilde map and the geometry of bilinear mass}\label{sectildepreliminaries}
The tilde-map duality argument, introduced by de Figueiredo, do Ó and Ruf \cite{FigueiredoDoORuf2005}, provides a way to overcome the strong indefiniteness of the Hamiltonian functional by identifying suitable manifolds on which its quadratic part is sign-definite. Here, we recall the formulation developed by Cassani and Tarsi \cite[Sections 2.1--2.2]{CassaniTarsi2015}.
The tilde map is defined by maximising the Hamiltonian pairing with respect to the second component, while keeping the first component fixed. This construction provides a useful way to exploit the dual structure of the system. We describe its realisation in weighted Hilbert spaces, which is suited to the planar problem with exponential nonlinearities. For general power nonlinearities, we instead retain the conjugate Sobolev framework introduced in Lemma~\ref{setting}.

Let $\mathcal H=H^1(\R^N)$, or its radial subspace, and fix
$\omega,b>0$. Write
\[
 (f,g)_\omega=\int(\nabla f\cdot\nabla g+\omega fg)\,,\quad\ 
 \|f\|_\omega^2=(f,f)_\omega\,,\quad\ 
 \mathfrak q_\omega(u,v)=(u,v)_\omega=A(u,v)+\omega B(u,v)\,.
\]
The component norms are $\|u\|_+=b^{-1}\|u\|_\omega$ and
$\|v\|_-=b\|v\|_\omega$. The parameter $\omega$ fixes an equivalent
Hilbert norm; it is independent of the unknown normalising frequency. For the proof of the next Lemma we refer to \cite{CassaniTarsi2015}.

\begin{lemma}[The weighted Hilbert tilde map]\label{tildemaplemma}
	For every $u\in\mathcal H$ the problem
	\begin{equation}\label{tildevariationaldefinition}
	 \max_{\|v\|_-=\|u\|_+}\mathfrak q_\omega(u,v)
	       =\|u\|_+^2
	\end{equation}
	has the unique maximiser $\widetilde u=b^{-2}u$.
	Its inverse sends $v$ to $b^2v$. On the product space the map
	\begin{equation}\label{tildeinvolution}
	 \mathcal T(u,v)=(b^2v,b^{-2}u)
	\end{equation}
	is an isometric involution for the weighted norm
	$\|(u,v)\|_{b,\omega}^2=b^{-2}\|u\|_\omega^2+b^2\|v\|_\omega^2$.
	Its eigenspaces
	\[
	 E_b^+=\{(bh,b^{-1}h):h\in\mathcal H\}\,,\qquad
	 E_b^-=\{(bh,-b^{-1}h):h\in\mathcal H\}
	\]
	give an orthogonal direct sum $\mathcal H^2=E_b^+\oplus E_b^-$.
	Moreover,
	\begin{equation}\label{tildecoercivepairing}
	 \mathfrak q_\omega'(Z)[\mathcal T Z]=\|Z\|_{b,\omega}^2\,.
\end{equation}
\end{lemma}

Equivalently, introduce the positive and negative coordinates
\begin{equation}\label{tildecoordinates}
 u=\frac b{\sqrt2}(z+w)\,,\quad\ 
 v=\frac1{b\sqrt2}(z-w)\,,\quad\ 
 z=\frac{b^{-1}u+bv}{\sqrt2}\,,\quad\ 
 w=\frac{b^{-1}u-bv}{\sqrt2}\,.
\end{equation}
The tilde involution sends $(z,w)$ to $(z,-w)$, and expansion gives
\begin{equation}\label{tildequadraticdecomposition}
 \begin{aligned}
 A(u,v)&=\tfrac12(\|\nabla z\|_2^2-\|\nabla w\|_2^2)\,,&
 B(u,v)&=\tfrac12(\|z\|_2^2-\|w\|_2^2)\,,\\
 \mathfrak q_\omega(u,v)&=\tfrac12(\|z\|_\omega^2-\|w\|_\omega^2)\,,&
 \|(u,v)\|_{b,\omega}^2&=\|z\|_\omega^2+\|w\|_\omega^2\,.
 \end{aligned}
\end{equation}
Thus the quadratic energy has infinitely many positive and negative
directions, while its tilde test controls both component norms.
The case $b=1$ gives $\widetilde u=u$ and the familiar diagonal and
antidiagonal graphs.
\vskip0.2truecm

For a prescribed mass $a>0$, set
\begin{equation}\label{S_ra}
	\mathbb S_{\mathcal H}=\{\zeta\in\mathcal H:\|\zeta\|_2=1\}\quad\mbox{and}\quad r_a(w)=(2a+\|w\|_2^2)^{1/2}.
\end{equation}
The map
\begin{equation}\label{tildechart}
	\mathcal X_a(\zeta,w)=\left(\frac b{\sqrt2}[r_a(w)\zeta+w],\ \frac1{b\sqrt2}[r_a(w)\zeta-w]\right)
\end{equation}
is a smooth parametrisation of the whole Hilbert constraint $B=a$.
Indeed \eqref{tildequadraticdecomposition} gives $B=a$ directly;
conversely $B=a$ implies $\|z\|_2=r_a(w)>0$, so the inverse is
$(u,v)\mapsto(z/\|z\|_2,w)$. This inverse is smooth as well.
The factor $r_a(w)$ is the correction needed to add a negative
direction while keeping the mass fixed.

The tilde test itself has a corresponding tangent correction.
For $Z=(u,v)$ with $B(Z)=a$ define
$M_b(Z)=b^{-2}\|u\|_2^2+b^2\|v\|_2^2$. Since
$B'(Z)[\mathcal T Z]=M_b(Z)$ and $B'(Z)[Z]=2a$,
\begin{equation}\label{tildemasscorrection}
 Y_Z=\mathcal T Z-\frac{M_b(Z)}{2a}Z
                  \in T_Z\{B=a\}.
\end{equation}
If the negative coordinate is restricted to a linear subspace
$W_m\subset\mathcal H$, both $\mathcal T Z$ and $Y_Z$ remain in
the corresponding linear test space. Hence these tests are available
in finite negative approximations and in the full space. In
Section~\ref{secTM}, the choice $b=(\vartheta_1/\vartheta_2)^{1/4}$
balances the exponential coefficients; \eqref{tildechart} supplies
the mass-preserving linking geometry, and
\eqref{tildecoercivepairing}--\eqref{tildemasscorrection} supply the
tests used for uniform estimates. The nonlinear estimates and the
full-fibre existence argument are proved there.



\section{The unit-frequency problem \texorpdfstring{\eqref{system1}}{(system1)}: existence and uniqueness of positive solutions}
\label{Section_system1}

In Theorem \ref{introexistencescaling} essential will be the connection, via scaling, between the constrained problem \eqref{system} and its unit-frequency counterpart \eqref{system1}. In this section we deal with solutions of \eqref{system1}: we show that any weak $(H^1)^2$ is in fact classical and exponentially decaying (Lemma \ref{profileregularity}), we prove the existence of a least-action solution, and we identify the least-action constant (Lemma \ref{dualexistence}). Our main contribution is the uniqueness and nondegeneracy result of Theorem \ref{thm:main}, which will be proved at the end of this section.

Since we will work with decaying solution in the class $\cS$, we first collect their well-known properties in the following lemma.
We record the passage from the weak solution class to the strong one, because neither exponent is assumed to be Sobolev subcritical by itself.  


\begin{lemma}[Regularity, symmetry and decay]\label{profileregularity}
	Under \eqref{range}, the class of positive $H^1(\R^N)^2$ weak solutions of
	\eqref{system1} and the one of positive classical solutions decaying at infinity
	coincide. Every such pair is smooth, radial and strictly
	radially decreasing about a common centre, and, for every $0<c<1$, there exists a constant $C_c>0$ such that
	\begin{equation}\label{profiledecay}
		U+V+|\nabla U|+|\nabla V|\le C_c\,e^{-c|x|}\,.
	\end{equation}
\end{lemma}
\begin{proof}
	Let $T_k(t):=\min\{t,k\}$ for $t\ge0$, and let $\chi_R$ be a
	standard cutoff which equals one on $B_R$, vanishes outside $B_{2R}$,
	and satisfies $|\nabla\chi_R|\le C/R$. Using
	$\chi_R^2T_k(V)$ in the first equation, by truncation and density, gives
	\begin{align*}
		\int\chi_R^2V^{p-1}T_k(V)\dd x
		&=\int\left[\nabla U\cdot\nabla(\chi_R^2T_k(V))
		+u\chi_R^2T_k(V)\right]\dd x \le C\|U\|_{H^1}\|V\|_{H^1}\,,
	\end{align*}
	where $C$ is independent of $R\ge1$ and $k>0$; here
	$|\nabla T_k(V)|\le|\nabla V|$ and $|T_k(V)|\leq V$ were used.
	Fatou's lemma as $R\to+\infty$, followed by monotone convergence as
	$k\to+\infty$, proves $V\in L^p(\R^N)$. Interchanging the two
	equations, one gets $U\in L^q(\R^N)$. Hence
	$V^{p-1}\in L^{p/(p-1)}(\R^N)\,,\ u^{q-1}\in L^{q/(q-1)}(\R^N)\,.
	$
	Since an $H^1$ solution is tempered, inversion of $-\Delta+1$ and
	the global elliptic estimates give
	$
	U\in W^{2,p/(p-1)}(\R^N)\,,\ V\in W^{2,q/(q-1)}(\R^N)\,.
	$
	The pair is therefore in the strong-solution class of
	\cite[Proposition 3.1]{BMR}; the bootstrap in
	\cite[Theorem 1.7 and Appendix A.1]{BMR} gives bounded, smooth
	representatives vanishing at infinity.  By
	\cite[Theorem 2]{BuscaSirakov}, both components are radial about a
	common point and strictly decreasing in the radial variable.
	After translation to the origin,
	\[
	U'(r)<0\,,\qquad V '(r)<0\,, \qquad r>0\,.
	\]
	For every $0<c<1$, choose $\varepsilon>0$ with
	$c^2<1-\varepsilon$. Since $U,V\to0$, outside a sufficiently large
	ball one has $U^{q-2}, V^{p-2}\le\varepsilon$. Hence $U+V$ satisfies
	$
	-\Delta(U+V)+(1-\varepsilon)(U+V)\le0
	$
	there. Exterior comparison, followed by the interior gradient
	estimates in \cite[Chapters 3 and 9]{GilbargTrudinger}, yields \eqref{profiledecay}.
	The same argument applies to positive classical decaying solutions,
	which therefore belong to $H^1(\R^N)^2$. Thus the weak and classical
	decaying solution classes coincide. Moreover, boundedness of $U,V$
	and the chain rule give $U^{q-1},V^{p-1}\in H^1(\R^N)$.
\end{proof}
%

Concerning existence, we follow the classical dual variational approach \cite{BMT} applied on the whole space with the Bessel resolvent.
Recall the definition of the resolvent $\mathcal R=(-\Delta+1)^{-1}$, which acts boundedly from $L^t(\R^N)$ to $W^{2,t}(\R^N)$
for $1<t<\infty$, and as a contraction on $L^t$ for $1\le t\le\infty$.
The convolution kernel of $\mathcal R$ has the representation
\begin{equation}\label{eq:bessel-potential}
	\cR f(x)=\int G_N(x-y)f(y)\dd y\quad\mbox{with}\quad G_N(x)=\int_0^\infty e^{-t}(4\pi t)^{-N/2}e^{-|x|^2/(4t)}\dd t\,.	
\end{equation}
It is positive and radially decreasing, has unitary $L^1-$norm, and belongs to
$L^b$ for $1\le b<N/(N-2)$ if $N\ge3$, and for every finite $b\ge1$
if $N=2$. In particular $G_N\in L^{1/\sigma}$ under \eqref{range}. 

Moreover, $\mathfrak C_{p,q}$ defined in \eqref{Cpq} is finite and positive. Indeed, 
\begin{equation*}
	\int f\,\cR g=\int f(G_N\ast g)\leq\|f\|_{q'}\|G_N\ast g\|_q\leq\|G_N\|_{1/\sigma}\|g\|_{p'}=\|G_N\|_{1/\sigma}
\end{equation*}
by Young's inequality for convolutions. Ground-state existence, positivity and attainment of the equivalent best constant $\mathfrak{C}_{p,q}$ in \eqref{Cpq} are established in \cite[Theorems 1.8--1.9 and Lemma 3.5]{BMR}; see
also \cite[Section 6.2, Theorem 6.4 and (6.5)]{BMT}.
The exponents in \cite{BMR} are $p-1,q-1$ in our notation.
Their hypotheses include $N=2$ and all $p,q>2$. The proof below is included to make the
dual normalisation and subsequent mass formulas self-contained.

Recall the definition of $c_+$ in \eqref{c_+} for the least action among positive profiles. 
To compare directly with the best constant in the cited literature, set
\[
S_{p,q}=\inf\left\{\|(-\Delta+1)w\|_{p'}^{p'}:w\in W^{2,p'}(\R^N),\ \|w\|_q=1\right\}\,.
\]
Duality and $w=\mathcal Rg$ give
$\mathfrak C_{p,q}=\|\mathcal R\|_{L^{p'}\to L^q}
=S_{p,q}^{-1/p'}$. Maximising the scalar fibre of
$\frac1{p'}\|(-\Delta+1)w\|_{p'}^{p'}-\frac1q\|w\|_q^q$ gives
\[
c_+=\left(\frac1{p'}-\frac1q\right)
S_{p,q}^{q/(q-p')}
=(1-\sigma)\mathfrak C_{p,q}^{-1/(1-\sigma)}\,.
\]

\begin{lemma}[Attainment of dual constant and least action]\label{dualexistence}
	For every $N\ge2$ and $p,q$ in \eqref{range}, the supremum in
	\eqref{Cpq} is attained by nonnegative radially decreasing
	functions. It yields a positive radial solution of \eqref{system1} in $\mathcal S$, and
	\begin{equation}\label{dualaction}
		c_+=(1-\sigma)\mathfrak C_{p,q}^{-1/(1-\sigma)}\,.
	\end{equation}
	In particular for $N=2$ the least-action positive profile exists for arbitrary $p,q>2$.
\end{lemma}

\begin{proof}
	The positivity of $G_N$ permits replacing $f,g$ by their absolute values.
	The Riesz rearrangement inequality then allows to choose a maximising
	sequence $f_n,g_n$ which is nonnegative and radially decreasing about the
	origin, see \cite[Theorem 3.7]{LiebLoss}.
	Since $\|f_n\|_{q'}=\|g_n\|_{p'}=1$, we can pass to weak limits in $L^{q'}$ and $L^{p'}$,
	respectively, denoted by $f,g$. Write $w_n=\mathcal Rg_n$.
	The resolvent estimate bounds $(w_n)_n$ in $W^{2,p'}$, and
	$\|w_n\|_{p'}\le1$ because $\|G_N\|_1=1$ by Young's inequality for convolutions.
	Moreover $w_n$ is radially decreasing. Setting $\omega_N:=|B_1|$, we claim that
	\[
	w_n(\rho)\le\omega_N^{-1/p'}\rho^{-N/p'},\qquad
	\int_{|x|>R}w_n^q\le C R^{-N(q-p')/p'}.
	\]
	Indeed, since $w_n$ is radially decreasing, then
		$$1\geq\|w_n\|_{p'}^{p'}\geq\int_{B_\rho(0)}w_n^{p'}\geq\omega_N\rho^Nw_n(\rho)^{p'}\,.$$
		Moreover, using this estimate,
		\begin{equation*}
			\int_{|x|>R}w_n^q\leq N\omega_N\int_R^{+\infty}\omega_N^{-q/p'}\rho^{-Nq/p'+N-1}\dd\rho=C_{N,p,q}R^{-N(q-p')/p'}.
		\end{equation*}
	Here $q>p'$ follows from $\sigma<1$. The strict inequality
	$1/p'-1/q=1-\sigma<2/N$ gives compact local embedding of
	$W^{2,p'}$ into $L^q$, which, together with the boundedness of $(w_n)_n$ in $W^{2,p'}(\R^N)$, implies that there exists a nonnegative $w\in W^{2,p'}(\R^N)$ such that $w_n\rightharpoonup w$ in $W^{2,p'}(\R^N)$ and $w_n\to w$ a.e. in $\R^N$. Hence, by Fatou's Lemma, the tail bound proved above for $w_n$ holds for $w$ as well. Combining local compactness with the tail bounds
	shows, after extraction, $\mathcal Rg_n=w_n\to w$ strongly in $L^q(\R^N)$. The identification of the limit $w=\mathcal Rg$ follows from the resolvent equation in distributions. Consequently, by weak-strong convergence,
	\[
	\int f\mathcal Rg=\lim_{n\to+\infty}\int f_n\mathcal Rg_n=\mathfrak C_{p,q}\,.
	\]
	Weak lower semicontinuity gives $\|f\|_{q'},\|g\|_{p'}\le1$, whereas
	Young's inequality for convolutions applied to the right-hand side of \eqref{Cpq} gives
	$\mathfrak C_{p,q}\le\mathfrak C_{p,q}\|f\|_{q'}\|g\|_{p'}$.
	Both norms therefore equal one, proving attainment. Uniform convexity
	also gives strong convergence of $f_n,g_n$ in their respective spaces.
	
	The multiplier equations on the
	two unit spheres, tested against $f$ and $g$, are
	\[
	\mathcal Rg=\mathfrak C_{p,q} f^{1/(q-1)}\,,\qquad
	\mathcal Rf=\mathfrak C_{p,q} g^{1/(p-1)}\,.
	\]
	Then
	$$u_0=f^{1/(q-1)}\,,\qquad v_0=g^{1/(p-1)}\,,$$
	satisfy
		\begin{equation*}
			\begin{cases}
				-\Delta u_0+u_0=\mathfrak C_{p,q}^{-1}v_0^{p-1}\,,\\
				-\Delta v_0+v_0=\mathfrak C_{p,q}^{-1}u_0^{q-1}\,.
			\end{cases}
		\end{equation*}
	In order to get rid of the constant in the right-hand side, define
	\[
	U=\mathfrak C_{p,q}^{-p/D}u_0\,,\qquad V=\mathfrak C_{p,q}^{-q/D}v_0\,.
	\]
	Since $u_0=C^{-1}\mathcal Rg$ and $v_0=C^{-1}\mathcal Rf$, and since
		$$-1+(p-1)\frac qD=\frac pD\quad\mbox{and}\quad-1+(q-1)\frac pD=\frac qD\,,$$
		it is easy to verify that
	these functions solve \eqref{system1}.
	The kernel is strictly positive, so both components are positive. 
	A simple computation shows that
	\begin{equation}\label{dualH}
		\int U^q=\int V^p=\mathfrak C_{p,q}^{-pq/D}=\mathfrak C_{p,q}^{-1/(1-\sigma)}\,.
	\end{equation}
	Hence the constructed pair is a strong solution. 
	Therefore Lemma \ref{profileregularity} implies $(U,V)\in\mathcal S$ and proves exponential decay.
	
	It remains to prove the characterisation \eqref{dualaction} for the least action.
	For any $(u,v)\in\mathcal S$, testing the equations with $(u,v)$ gives
	$$H:=\int u^q=\int\nabla u\cdot\nabla v+\int uv=\int v^p\,.$$
	Taking $f=u^{q-1}$ and $g=v^{p-1}$ in the homogeneous form of \eqref{Cpq}, 
	and using $\mathcal Rg=u$, yields
		\begin{equation*}
			\mathfrak C_{p,q}\geq\frac{\int u^{q-1}\cR(v^{p-1})}{\|u^{q-1}\|_{q'}\|v^{p-1}\|_{p'}}=\frac{\|u\|_q^q}{\|u\|_q^{q-1}\|v\|_p^{p-1}}\,,
		\end{equation*}
		from which
	\[
	H\le\mathfrak C_{p,q} H^{1/q'+1/p'}=\mathfrak C_{p,q} H^{2-\sigma}\quad\mbox{and hence}\quad
	H\ge\mathfrak C_{p,q}^{-1/(1-\sigma)}.
	\]
	Note that the action functional on solutions is $I_1=(1-\sigma)H$, and \eqref{dualH} attains this lower
	bound. This proves \eqref{dualaction} and concludes the proof.
\end{proof}

\subsection{Global uniqueness and nondegeneracy: Proof of Theorem \ref{thm:main}}\label{secuniqueness}
We first establish nondegeneracy in the whole range \eqref{range} and then compactness of all positive profiles; uniqueness will be then a plain consequence. We stress the fact that our argument does not assume any least-action type assumption.
Throughout this section, we assume that $(U,V)\in\mathcal S$ is a centred profile, and we recall that Lemma~\ref{profileregularity} gives $U',V'<0$ on $(0,\infty)$ and exponential decay at $\infty$. In order to have a lighter notation, in all radial integrals we omit the common sphere-area factor.

For $a\ge0$ and $r>0$ we introduce the modified Bessel functions of the first and second kind given by
\begin{equation*}
	I_a(r)=\sum_{j=0}^\infty\frac{(r/2)^{2j+a}}{j!\,\Gamma(j+a+1)},\qquad
	K_a(r)=\int_0^\infty e^{-r\cosh t}\cosh(at)\dd t,
\end{equation*}
respectively, where $\Gamma$ is Euler's gamma function, see
\cite[Chapter 10]{NIST}. Moreover, for $\cR=(-\Delta+1)^{-1}$, we also recall its the representation for radial functions (identified with their profiles)
\begin{equation}\label{eq:radial-resolvent}
	\cR f(r)=\int_0^\infty\mathcal G(r,s)f(s)s^{N-1}\dd s,
\end{equation}
where
\begin{equation}\label{sp:Bessel-kernel}
	\mathcal G(r,s)=(rs)^{-(N-2)/2}I_{\frac{N-2}2}(\min\{r,s\})K_{\frac{N-2}2}(\max\{r,s\})\,,\qquad r,s>0\,.
\end{equation}
The normalization follows from the modified-Bessel Wronskian identity,
and the endpoint asymptotics select the regular solution at zero and
the decaying solution at infinity, see \cite[Chapter 10]{NIST}.

The radial counterpart of Proposition \ref{pohozaev} is
\begin{align*}
	&(N-2)\int_0^\infty U'V'r^{N-1}\dd r
	+N\int_0^\infty UVr^{N-1}\dd r =\frac Nq\int_0^\infty U^qr^{N-1}\dd r
	+\frac NP\int_0^\infty V^pr^{N-1}\dd r\,,
\end{align*}
which, together with \eqref{testing}, gives
\begin{equation}\label{eq:pohozaev}
	2\int_0^\infty UV r^{N-1}\dd r
	=\delta\int_0^\infty U^qr^{N-1}\dd r\,,
	\qquad
	\delta=\frac Nq+\frac Np-(N-2)>0\,.
\end{equation}
\begin{lemma}\label{cumulative}
	For every $r>0$,
	\[
	F_0(r)=\int_0^r\left(U^q-\frac2\delta UV\right)s^{N-1}\dd s>0\,,
	\qquad
	G_0(r)=\int_0^r\left(V^p-\frac2\delta UV\right)s^{N-1}\dd s>0\,.
	\]
\end{lemma}

\begin{proof}
	Both functions vanish at zero and tend to zero at infinity by
	\eqref{testing} and \eqref{eq:pohozaev}. The radial equations are
	\begin{equation}\label{radial_eqs}
		U''+\frac{N-1}{r}U'=U-V^{p-1}\,,
		\qquad
		V''+\frac{N-1}{r}V'=V-U^{q-1}\,.
	\end{equation}
	Consequently,
	\begin{align*}
		\frac\dd{\dd r}\left[r^{N-1}(U'V-UV')\right]
		&=r^{N-1}\left[
		\left(U''+\frac{N-1}{r}U'\right)V
		-U\left(V''+\frac{N-1}{r}V'\right)\right]\\
		&=r^{N-1}\bigl[(U-V^{p-1})V-U(V-U^{q-1})\bigr] =r^{N-1}(U^q-V^p)\,.
	\end{align*}
	Radial $C^2$ regularity gives $U'(r),V'(r)=O(r)$ as $r\to0$;
	hence
	$
	\lim_{r\to0}r^{N-1}(U'V-UV')=0\,.
	$
	Integration from $0$ to $r$ therefore gives
	\begin{equation}\label{eq:cumulative-wronskian}
		r^{N-1}\bigl(U'(r)V(r)-U(r)V'(r)\bigr)
		=\int_0^r(U^q-V^p)s^{N-1}\dd s
		=F_0(r)-G_0(r)\,.
	\end{equation}
	In the last equality, the two terms involving $UV$ in the definitions
	of $F_0$ and $G_0$ cancel.
	Suppose one of $F_0$, $G_0$ takes a negative value. Since both functions
	vanish at zero and tend to zero at infinity, $\min\{F_0,G_0\}$ attains
	a negative global minimum at some $r_0\in(0,\infty)$. If it is attained
	by $F_0$, then $F_0(r_0)\le G_0(r_0)$, $F_0'(r_0)=0$ and
	$F_0''(r_0)\ge0$. Since
	\[
	F_0'=r^{N-1}UV\left(\frac{U^{q-1}}V-\frac2\delta\right),
	\]
	the identity $F_0'(r_0)=0$ gives $U(r_0)^{q-1}/V(r_0)=\frac2\delta$, and hence
	\[
	F_0''(r_0)=r_0^{N-1}U(r_0)V(r_0)
	\left(\frac{U^{q-1}}V\right)'(r_0)\,.
	\]
	It follows that $({U^{q-1}}/V)'(r_0)\ge0$, or equivalently,
	\[
	\frac{V'(r_0)}{V(r_0)}
	\le q\frac{U'(r_0)}{U(r_0)}
	<\frac{U'(r_0)}{U(r_0)}\,.
	\]
	Equation~\eqref{eq:cumulative-wronskian} now gives
	$F_0(r_0)>G_0(r_0)$, a contradiction. If the minimum is attained by
	$G_0$, the same argument with the components interchanged gives
	$G_0(r_0)>F_0(r_0)$, again a contradiction. Thus $F_0,G_0\ge0$.
	At any interior zero of $F_0$, the same local-minimum argument
	would give $F_0>G_0$, forcing $G_0<0$; an interior zero of $G_0$ is excluded
	in the same way. This proves both strict inequalities.
\end{proof}

Set
\begin{equation}\label{eq:measures}
	\dd\mu= U^qr^{N-1}\dd r,\qquad
	\dd\nu=V^pr^{N-1}\dd r,\qquad
	\dd\rho=UVr^{N-1}\dd r\,.
\end{equation}
By \eqref{testing} and \eqref{eq:pohozaev}, they are finite, have positive densities, and satisfy
$\mu((0,\infty))=\nu((0,\infty))=\frac2\delta\rho((0,\infty))$.
By the characterization of the usual stochastic order in
\cite[Section 1.A.1]{ShakedShanthikumar}, we have, for nonconstant
nonincreasing functions $h$, $k$,
\begin{equation}\label{eq:measure-comparison}
	\int_0^\infty h\dd\mu>\frac2\delta\int_0^\infty h\dd\rho\,,
	\quad h\in L^1(\mu)\cap L^1(\rho)\,;
	\qquad
	\int_0^\infty k\dd\nu>\frac2\delta\int_0^\infty k\dd\rho\,,
	\quad k\in L^1(\nu)\cap L^1(\rho)\,.
\end{equation}
The inequalities are strict by Lemma \ref{cumulative} and since $h,k$ are nonconstant.



\subsubsection{Nondegeneracy}

Since
\begin{equation}\label{eq:nonlinear-resolvent}
	U=\mathcal R(V^{p-1})\quad\mbox{and}\quad V=\mathcal R(U^{q-1})\,,
\end{equation}
define
\begin{equation}\label{sp:def-PQ}
	Pk=\frac{\mathcal R(V^{p-1}k)}U\,,\qquad
	Qh=\frac{\mathcal R(U^{q-1}h)}V\,.
\end{equation}
These initially act on bounded radial functions.

\begin{lemma}\label{sp:weighted-operators}
	The operators $P:L^2(\nu)\to L^2(\mu)$ and
	$Q:L^2(\mu)\to L^2(\nu)$ extend to compact injective contractions.
	They satisfy $P^*=Q$ and $P1=Q1=1$. Thus $T=PQ$ is compact,
	injective, self-adjoint and nonnegative, with simple largest eigenvalue
	$1$ and eigenspace $\operatorname{span}\{1\}$. 
\end{lemma}
\begin{proof}
	By \eqref{eq:nonlinear-resolvent} and \eqref{eq:radial-resolvent},
	\begin{equation}\label{eq20}
		U(r)=\int_0^\infty\mathcal G(r,s)V(s)^{p-1}s^{N-1}\dd s\,,
		\qquad
		V(r)=\int_0^\infty\mathcal G(r,s)U(s)^{q-1}s^{N-1}\dd s\,.
	\end{equation}
	Thus the positive kernels defining $P$ and $Q$ have total mass one,
	and hence $P1=Q1=1\,.$
	
	For nonnegative $h\in L^1(\nu)$ and
	$k\in L^1(\mu)$, Tonelli's theorem, the symmetry of
	$\mathcal G$, and \eqref{eq:nonlinear-resolvent} give
	\begin{equation}\label{sp:mass-preservation}
		\begin{aligned}
			\int_0^\infty Ph\dd\mu
			&=\int_0^\infty U^{q-1}\cR(V^{p-1}h)r^{N-1}\dd r
			=\int_0^\infty V^{p-1}h\,\cR(U^{q-1})r^{N-1}\dd r
			=\int_0^\infty h\dd\nu\,,\\
			\int_0^\infty Qk\dd\nu
			&=\int_0^\infty V^{p-1}\cR(U^{q-1}k)r^{N-1}\dd r
			=\int_0^\infty U^{q-1}k\,\cR(V^{p-1})r^{N-1}\dd r
			=\int_0^\infty k\dd\mu\,.
		\end{aligned}
	\end{equation}
	Let $h\in L^2(\nu)$. By
	\eqref{sp:mass-preservation},
	$P(|h|^2)<\infty$ $\mu$-almost everywhere. For such $r$, H\"older's
	inequality and \eqref{eq:radial-resolvent} give
	\begin{equation}\label{eq22}
		|Ph(r)|^2\le\frac1{U(r)}\int_0^\infty\mathcal G(r,s)V(s)^{p-1}|h(s)|^2s^{N-1}\dd s=P(|h|^2)(r)\,.
	\end{equation}
	Integrating and using \eqref{sp:mass-preservation}, we obtain
	$
	\|Ph\|_{L^2(\mu)}^2
	\le\int_0^\infty P(|h|^2)\dd\mu
	=\int_0^\infty |h|^2\dd\nu
	=\|h\|_{L^2(\nu)}^2.
	$
	The same argument for $Q$ yields
	$
	\|Qk\|_{L^2(\nu)}\le\|k\|_{L^2(\mu)}\ \text{for all   $k\in L^2(\mu)$.}
	$
	
	Le us now address adjointness. Let $h\in L^2(\nu)$ and
	$k\in L^2(\mu)$. By the contraction estimates,
	\[
	\int_0^\infty P(|h|)|k|\dd\mu
	\le \|h\|_{L^2(\nu)}\|k\|_{L^2(\mu)}\,.
	\]
	Thus the double integral below is absolutely convergent, and Fubini's
	theorem and the symmetry of $\mathcal G$ yield
	\begin{align*}
		\langle Ph,k\rangle_{L^2(\mu)}
		&=\int_0^\infty\int_0^\infty
		\mathcal G(r,s)U(r)^{q-1}V(s)^ph(s)k(r)
		r^{N-1}s^{N-1}\dd s\dd r
		=\langle h,Qk\rangle_{L^2(\nu)}\,.
	\end{align*}
	Therefore $P^*=Q$.
	
	Let us now turn to the compactness issue. Let $(h_j)_j$ be bounded in $L^2(\nu)$ and set
	$f_j=V^{p/2}h_j$. Then $(f_j)_j$ is bounded in
	$L^2((0,\infty),r^{N-1}\dd r)$ and
	\begin{equation}\label{eq23}
		U^{q/2}Ph_j
		=U^{(q-2)/2}\cR\bigl(V^{(p-2)/2}f_j\bigr).
	\end{equation}
	By continuity of the operator $\cR$, the sequence
	$\cR(V^{(p-2)/2}f_j)$ is bounded in
	$H^2_{\mathrm{rad}}(\R^N)$ and is therefore relatively compact
	in $L^2$ on every bounded ball. Moreover,
	\[
	\left\|
	U^{(q-2)/2}\cR\bigl(V^{(p-2)/2}f_j\bigr)
	\right\|_{L^2((R,\infty),r^{N-1}\dd r)}
	\le
	\sup_{r\ge R}U(r)^{(q-2)/2}
	\left\|\cR\bigl(V^{(p-2)/2}f_j\bigr)\right\|_{L^2}.
	\]
	Since $U(r)\to0$ as $r\to+\infty$, the right-hand side tends to zero
	uniformly in $j$. Thus, after passing to a subsequence,
	$U^{(q-2)/2}\cR(V^{(p-2)/2}f_j)$ converges in
	$L^2((0,\infty),r^{N-1}\dd r)$; by \eqref{eq23},
	$(Ph_j)$ converges in $L^2(\mu)$. Hence $P$ is compact. Since $Q=P^*$, the operator $Q$ is
	also compact.

	Concerning injectivity, we first note that
	\[
	\int_0^\infty |V^{p-1}h|^2r^{N-1}\dd r
	=\int_0^\infty V^{p-2}|h|^2\dd\nu
	\le\|V\|_\infty^{p-2}\|h\|_{L^2(\nu)}^2\quad\text{for $h\in L^2(\nu)$}\,.
	\]
	If $Ph=0$, then $\cR(V^{p-1}h)=uPh=0$. The injectivity of $\cR$ and
	the positivity of $V$ imply $h=0$. Hence $P$ is injective, and the
	same argument proves that $Q$ is injective.
	
	To obtain the claimed properties of $T$, we exploit $P^*=Q$. Since $P=Q^*$, for every
	$f\in L^2(\mu)$,
	$T=PQ=Q^*Q,\ 
	\langle Tf,f\rangle_{L^2(\mu)}
	=\|Qf\|_{L^2(\nu)}^2\ge0.$
	Consequently, $T$ is compact, injective, self-adjoint, and
	nonnegative. Furthermore, $T1=1,\ \|T\|\le\|P\|\,\|Q\|\le1,$
	so $1$ is the largest eigenvalue of $T$.
	
	To show simplicity, suppose finally that $f\in L^2(\mu)$ satisfies
	$Tf=f$. Then
	$\|f\|_{L^2(\mu)}^2
	=\langle Tf,f\rangle_{L^2(\mu)}
	=\|Qf\|_{L^2(\nu)}^2.$
	Using this and the second equality in
	\eqref{sp:mass-preservation}, we obtain
	\begin{align*}
		0
		&=\|f\|_{L^2(\mu)}^2-\|Qf\|_{L^2(\nu)}^2 =\int_0^\infty Q(|f|^2)\dd\nu
		-\int_0^\infty |Qf|^2\dd\nu =\int_0^\infty
		\bigl[Q(|f|^2)-|Qf|^2\bigr]\dd\nu\,.
	\end{align*}
	Similarly to \eqref{eq22}, 
	$
	|Qf|^2\le Q(|f|^2).
	$
	So the integrand is nonnegative and therefore vanishes for
	almost every $r$. For every such $r$, the definition of $Q$ and
	\eqref{eq:nonlinear-resolvent} give
	\[
	\frac1{V(r)}\int_0^\infty
	\mathcal G(r,s)U(s)^{q-1}s^{N-1}\dd s=1\,,
	\qquad
	\frac1{V(r)}\int_0^\infty
	\mathcal G(r,s)U(s)^{q-1}f(s)s^{N-1}\dd s=Qf(r)\,.
	\]
	Consequently,
	\begin{align*}
		0=Q(|f|^2)(r)-|Qf(r)|^2 
		&=\frac1{V(r)}\int_0^\infty\mathcal G(r,s)U(s)^{q-1}\bigl(|f(s)|^2-2f(s)Qf(r)+|Qf(r)|^2\bigr)s^{N-1}\dd s\\
		&=\frac1{v(r)}\int_0^\infty\mathcal G(r,s)U(s)^{q-1}|f(s)-Qf(r)|^2s^{N-1}\dd s\,.
	\end{align*}
	Since $\mathcal G(r,s)U(s)^{q-1}>0$ for $r,s>0$, it follows that
	$f(s)=Qf(r)$ for almost every $s>0$. Thus $f$ is constant almost
	everywhere, and hence $\ker(T-I)=\operatorname{span}\{1\}$. We thus conclude that the eigenvalue $1$ is simple.
\end{proof}

\begin{lemma}\label{sp:monotonicity}
	If $h\in L^2(\nu)$ has a nonincreasing representative, then so does
	$Ph$. If $k\in L^2(\mu)$ has a nonincreasing representative, then so
	does $Qk$. Consequently, $T$ preserves the cone of nonincreasing
	functions in $L^2(\mu)$.
\end{lemma}

\begin{proof}
	We first establish the determinant inequality for the Green kernel.
	Fix $0<r_1<r_2$. By \eqref{sp:Bessel-kernel}, the function
	\[
	s\mapsto\frac{\mathcal G(r_2,s)}{\mathcal G(r_1,s)}
	\]
	is constant on $(0,r_1)$ and $(r_2,\infty)$, while on
	$(r_1,r_2)$ it is a positive constant times
	$I_{\frac{N-2}2}(s)/K_{\frac{N-2}2}(s)$. The modified-Bessel Wronskian
	identity gives
	\[
	\frac{\dd}{\dd s}
	\left(
	\frac{I_{\frac{N-2}2}(s)}{K_{\frac{N-2}2}(s)}
	\right)
	=\frac1{sK_{\frac{N-2}2}(s)^2}>0\,,
	\]
	see \cite[Chapter 10]{NIST}. Therefore,
	\begin{equation}\label{sp:TP2}
		\mathcal G(r_1,s)\mathcal G(r_2,t)
		-\mathcal G(r_1,t)\mathcal G(r_2,s)\ge0\,,
		\qquad 0<s<t\,.
	\end{equation}
	Let $h\in L^2(\nu)$ be bounded and nonincreasing. By
	\eqref{sp:def-PQ} and \eqref{eq:radial-resolvent},
	$$Ph(r)=\frac1{U(r)}\int_0^\infty\mathcal G(r,s)V(s)^{p-1}h(s)s^{N-1}\dd s\,.$$
	For $0<r_1<r_2$, by Fubini's theorem and \eqref{eq20}, 
	\begin{align*}
		&U(r_1)U(r_2)\bigl(Ph(r_2)-Ph(r_1)\bigr)\\
		&=U(r_1)\int_0^\infty
		\mathcal G(r_2,s)V(s)^{p-1}h(s)s^{N-1}\dd s -U(r_2)\int_0^\infty
		\mathcal G(r_1,s)V(s)^{p-1}h(s)s^{N-1}\dd s\\
		&=\int_0^\infty\int_0^\infty h(s)
		\bigl(\mathcal G(r_2,s)\mathcal G(r_1,t)
		-\mathcal G(r_1,s)\mathcal G(r_2,t)\bigr)
		V(s)^{p-1}V(t)^{p-1}(st)^{N-1}\dd s\dd t\\
		&=\int_{0<s<t<\infty}\bigl(h(s)-h(t)\bigr)
		\bigl(\mathcal G(r_2,s)\mathcal G(r_1,t)
		-\mathcal G(r_1,s)\mathcal G(r_2,t)\bigr)
		V(s)^{p-1}V(t)^{p-1}(st)^{N-1}\dd s\dd t\,.
	\end{align*}
	For $0<s<t$, the monotonicity of $h$ and \eqref{sp:TP2} give
	$h(s)\geq h(t)$ and $\mathcal G(r_2,s)\mathcal G(r_1,t)\leq\mathcal G(r_1,s)\mathcal G(r_2,t)$.
	Since $U(r_1)U(r_2)>0$, it follows that
	$Ph(r_2)\le Ph(r_1)$. Thus $Ph$ is nonincreasing.
	
	Now let $h\in L^2(\nu)$ have a nonincreasing representative and set
	$
	h_j=\max\{-j,\min\{h,j\}\}.
	$
	Then $h_j$ is bounded and nonincreasing and $h_j\to h$ in
	$L^2(\nu)$. Since $P:L^2(\nu)\to L^2(\mu)$ is a contraction,
	$Ph_j\to Ph$ in $L^2(\mu)$. Passing to an a.e. convergent subsequence,
	we conclude that $Ph$ has a nonincreasing representative.
	Interchanging the two components proves the corresponding assertion
	for $Q$. The conclusion for $T=PQ$ follows immediately.
\end{proof}

\begin{lemma}\label{sp:second-mode}
	The restriction of $T$ to the orthogonal complement of the constant
	functions in $L^2(\mu)$ has a largest eigenvalue $\theta_2$, which
	satisfies $0<\theta_2<1$. Its eigenspace contains a nonzero
	nonincreasing function.
\end{lemma}

\begin{proof}
	Since $T$ is self-adjoint and $T1=1$, the orthogonal complement of
	the constant functions is invariant under $T$. By
	Lemma~\ref{sp:weighted-operators}, its restriction to this subspace is
	a nonzero compact, self-adjoint, nonnegative operator. Hence the
	spectral theorem \cite[Theorem~6.11]{Brezis} shows that its norm is a
	positive eigenvalue $\theta_2$. Since $\theta_2\le\|T\|=1$ and
	$\ker(T-I)=\operatorname{span}\{1\}$, we have $0<\theta_2<1$. 
	For $R>0$, set
	\[
	h_R=\mathbf 1_{(0,R)}
	-\frac{\mu((0,R))}{\mu((0,\infty))}\,.
	\]
	Then $h_R$ is nonincreasing and $\dd\mu$-orthogonal to the constant functions.
	We claim that, for some $R>0$, its orthogonal projection onto the
	$\theta_2$-eigenspace is nonzero. Otherwise, any nonzero function $e$
	in this eigenspace would satisfy, for every $R>0$,
	\[
	0=\langle e,h_R\rangle_{L^2(\mu)}
	=\int_0^R e\dd\mu
	-\frac{\mu((0,R))}{\mu((0,\infty))}
	\int_0^\infty e\dd\mu
	=\int_0^R e\dd\mu\,,
	\]
	where the last equality follows because $e$ is orthogonal to the
	constant functions. Since $\dd\mu=U^qr^{N-1}\dd r$ has a strictly
	positive density on $(0,\infty)$, the absolute continuity of the
	indefinite integral would imply $e=0$ a.e., a
	contradiction.
	
	Fix such an $R$ and denote the projection of $h_R$ onto the
	$\theta_2$-eigenspace by $h$. Then $h\ne0$ and $Th=\theta_2h$.
	Since $h$ is the orthogonal projection of $h_R$ onto the
	$\theta_2$-eigenspace, the spectral theorem gives a number
	$0\le\rho<\theta_2$ such that
	\[
	\bigl\|\theta_2^{-j}T^jh_R-h\bigr\|_{L^2(\mu)}
	=\theta_2^{-j}\bigl\|T^j(h_R-h)\bigr\|_{L^2(\mu)}
	\le
	\left(\frac{\rho}{\theta_2}\right)^j
	\|h_R-h\|_{L^2(\mu)}
	\to 0\,.
	\]
	$\theta_2^{-j}T^jh_R$ is nonincreasing by
	Lemma~\ref{sp:monotonicity}. Passing to an almost-everywhere
	convergent subsequence, we conclude that $h$ is nonincreasing. Thus $h$ is the required eigenfunction.
\end{proof}

\begin{proposition}\label{sp:strict-gap}
	The eigenvalue $\theta_2$ in Lemma~\ref{sp:second-mode}
	satisfies $\theta_2<((p-1)(q-1))^{-1}$. Consequently, the radial kernel of \eqref{eq:linearized} is trivial
	in both $L^\infty(\R^N)^2$ and $L^2(\R^N)^2$.
\end{proposition}
\begin{proof}
	Choose a nonzero nonincreasing eigenfunction
	$h\in L^2(\mu)$ given by Lemma~\ref{sp:second-mode} and set
	$k=Qh\in L^2(\nu)$. Since $h$ belongs to the orthogonal complement
	of the constant functions,
	$
	\int_0^\infty h\dd\mu=0.
	$
	Moreover, by this and $P^*=Q$,
	$
	\int_0^\infty k\dd\nu
	=\langle Qh,1\rangle_{L^2(\nu)}
	=\langle h,P1\rangle_{L^2(\mu)}
	=\int_0^\infty h\dd\mu=0.
	$
	Thus, we have
	\begin{equation}\label{sp:eigenrelations}
		Qh=k,\qquad Pk=Th=\theta_2h,\qquad
		\int_0^\infty h\dd\mu=\int_0^\infty k\dd\nu=0\,.
	\end{equation}
	By Lemma~\ref{sp:monotonicity}, $k$ is nonincreasing. Since $Q$ is
	injective, $k\ne0$; hence both $h$ and $k$ are nonconstant. 
	We first verify the integrability needed to apply the measure
	comparison. Since
	$$\int_0^\infty |U^{q-1}h|^2r^{N-1}\dd r
	=\int_0^\infty U^{q-2}|h|^2\dd\mu
	\le\|U\|_\infty^{q-2}\|h\|_{L^2(\mu)}^2\,,
	$$ and similarly $$
	\int_0^\infty |V^{p-1}k|^2r^{N-1}\dd r
	\le\|V\|_\infty^{p-2}\|k\|_{L^2(\nu)}^2\,,$$
	the continuity of $\cR$ and \eqref{sp:eigenrelations} give
	$
	Vk=\cR(U^{q-1}h)\in H^2_{\mathrm{rad}}(\R^N),
	\ 
	Uh=\theta_2^{-1}\cR(V^{p-1}k)
	\in H^2_{\mathrm{rad}}(\R^N).
	$
	Since
	$Uh,Vk,U,V\in L^2((0,\infty),r^{N-1}\dd r)$, the
	Cauchy--Schwarz inequality yields
	\begin{align*}
		\int_0^\infty |h|\dd\rho
		=\int_0^\infty |U(r)h(r)|V(r)r^{N-1}\dd r 
		<\infty\,,\qquad
		\int_0^\infty |k|\dd\rho
		=\int_0^\infty |V(r)k(r)|U(r)r^{N-1}\dd r 
		<\infty\,.
	\end{align*}
	The finiteness of $\mu$ and $\nu$ also gives
	$h\in L^1(\mu)$ and $k\in L^1(\nu)$. Therefore,
	\eqref{eq:measure-comparison} and \eqref{sp:eigenrelations} imply
	\begin{equation}\label{sp:negative-rho}
		\int_0^\infty k\dd\rho<0\,,\qquad
		\int_0^\infty h\dd\rho<0\,.
	\end{equation}
	
	Set
	\[
	\xi=\frac Nq+\frac{rU'}U\,,\qquad
	\eta=\frac Np+\frac{rV'}V\,.
	\]
	By the definition of $\xi$, \eqref{eq:measures}, and
	Lemma \ref{profileregularity}, for every $R>0$,
	\begin{align}\label{eq27}
		\int_0^R\xi\dd\mu
		=\frac1q
		\left[r^NU(r)^q\right]_0^R>0\,,\qquad
		\int_R^\infty\xi\dd\mu
		=\frac1q
		\left[r^NU(r)^q\right]_R^\infty<0\,.
	\end{align}	
	By Lemma \ref{profileregularity},
	$\xi \in L^2(\mu)$ and, since $h\in L^2(\mu)$,
	$\xi h\in L^1(\mu)$. Moreover, $h\not\equiv0$ and
	\eqref{sp:eigenrelations} implies that $h$ changes sign. Since $h$ is nonincreasing, $\{h>t\}$ is a
	proper initial interval for $0<t<\sup h_+$, whereas $\{h<-t\}$ is a
	proper terminal interval for $0<t<\sup h_-$. Hence \eqref{eq27} gives
	\[
	\int_{\{h>t\}}\xi\dd\mu>0
	\quad(0<t<\sup h_+)\,,\qquad
	\int_{\{h<-t\}}\xi\dd\mu<0
	\quad(0<t<\sup h_-)\,.
	\]
	Since
	\[
	h_+(r)=\int_0^\infty\mathbf 1_{\{h(r)>t\}}\dd t\,,
	\qquad
	h_-(r)=\int_0^\infty\mathbf 1_{\{h(r)<-t\}}\dd t\,,
	\]
	Fubini's theorem gives
	\begin{equation}\label{sp:xi-positive}
		\begin{aligned}
			\langle\xi,h\rangle_{L^2(\mu)}
			&=\int_0^\infty\xi h_+\dd\mu
			-\int_0^\infty\xi h_-\dd\mu\\
			&=\int_0^\infty
			\left(\int_{\{h(r)>t\}}\xi\dd\mu\right)\dd t
			-\int_0^\infty
			\left(\int_{\{h(r)<-t\}}\xi\dd\mu\right)\dd t>0\,.
		\end{aligned}
	\end{equation}
	
	Using
	$
	(-\Delta+1)(x\cdot\nabla w)
	=x\cdot\nabla\bigl((-\Delta+1)w\bigr)-2\Delta w,
	$
	the definitions of $\xi,\eta$ and \eqref{system1} give
	\begin{align*}
		(-\Delta+1)(U\xi)&-(p-1)V^p(V\eta) =(-\Delta+1)
		\left(\frac NqU+x\cdot\nabla U\right)
		-(p-1)V^p(V\eta)\\
		&=\frac NqV^{p-1}
		+(p-1)V^px\cdot\nabla V-2\Delta U
		-(p-1)V^p(V\eta)\\
		&=\frac NqV^{p-1}
		+(p-1)V^{p-2}\left(V\eta-\frac NpV\right)
		+2(V^{p-1}-U)-(p-1)V^{p-2}(V\eta)\\
		&=\left(\frac Nq
		-\frac{(p-1)N}p+2\right)V^{p-1}-2U
		\overset{\eqref{eq:pohozaev}}{=}\delta V^{p-1}-2U\,.
	\end{align*}
	The same calculation with the two components interchanged gives
	\[
	(-\Delta+1)(V\eta)-(q-1)U^{q-2}(U\xi)=\delta U^{q-1}-2V\,.
	\]	
	By Lemma \ref{profileregularity}, we also have
	$
	\eta\in L^2(\nu),\ U\xi,\ V\eta,\ V^{p-1}\eta,\ U^{q-1}\xi\in L^2(\R^N).
	$
	Applying $\cR$ to
	$$
	(-\Delta+1)(U\xi)=(p-1)V^{p-1}\eta+\delta V^{p-1}-2U,
	\ 
	(-\Delta+1)(V\eta)=(q-1)U^{q-1}\xi+\delta U^{q-1}-2V,
	$$
	and using $\cR(V^{p-1})=U$ and $\cR(U^{q-1})=V$, we obtain
	$$
	U\xi-(p-1)\cR(V^{p-1}\eta)
	=\delta U-2\cR U,\ 
	V\eta-(q-1)\cR(U^{q-1}\xi)
	=\delta V-2\cR V.
	$$
	Dividing by $U$ and $V$, respectively, and using the definitions of
	$P$ and $Q$ gives
	\begin{equation}\label{sp:dilation-weighted}
		\xi-(p-1)P\eta=\delta-2\frac{\cR U}{U}\,,
		\qquad
		\eta-(q-1)Q\xi=\delta-2\frac{\cR V}{V}\,.
	\end{equation}
	Since $\xi\in L^2(\mu)$, $\eta\in L^2(\nu)$, and $\mu,\nu$ are
	finite, Lemma~\ref{sp:weighted-operators} gives
	$P\eta\in L^2(\mu)$ and $Q\xi\in L^2(\nu)$. Hence
	\eqref{sp:dilation-weighted} implies
	$
	\frac{\cR U}{U}\in L^2(\mu),
	\ 
	\frac{\cR V}{V}\in L^2(\nu).
	$	
	By this and the self-adjointness of $\cR$ and
	\eqref{sp:eigenrelations},
	\begin{align*}
		\left\langle\frac{\cR U}{U},h\right\rangle_{L^2(\mu)}
		&=\int_0^\infty(\cR U)U^{q-1}h r^{N-1}\dd r
		=\int_0^\infty U\,\cR(U^{q-1}h)r^{N-1}\dd r
		=\int_0^\infty k\dd\rho\,,\\
		\left\langle\frac{\cR V}{V},k\right\rangle_{L^2(\nu)}
		&=\int_0^\infty(\cR V)V^{p-1}k r^{N-1}\dd r
		=\int_0^\infty V\,\cR(V^{p-1}k)r^{N-1}\dd r
		=\theta_2\int_0^\infty h\dd\rho\,.
	\end{align*}
	Using $Qh=k$, $Pk=\theta_2h$, and the zero-mean identities in
	\eqref{sp:eigenrelations}, together with $P^*=Q$, the first identity
	in \eqref{sp:dilation-weighted} and \eqref{sp:negative-rho} give
	\begin{align*}
		\langle\xi,h\rangle_{L^2(\mu)}
		-(p-1)\langle\eta,k\rangle_{L^2(\nu)}
		&=\langle\xi,h\rangle_{L^2(\mu)}
		-(p-1)\langle\eta,Qh\rangle_{L^2(\nu)} =\langle\xi-(p-1)P\eta,h\rangle_{L^2(\mu)}\\
		&=\delta\int_0^\infty h\dd\mu
		-2\left\langle\frac{\cR U}{U},h\right\rangle_{L^2(\mu)} =-2\int_0^\infty k\dd\rho>0\,.
	\end{align*}
	Similarly, the second identity in \eqref{sp:dilation-weighted} and \eqref{sp:negative-rho} give 
	\begin{align*}
		\langle\eta,k\rangle_{L^2(\nu)}
		-(q-1)\theta_2\langle\xi,h\rangle_{L^2(\mu)}
		&=\langle\eta,k\rangle_{L^2(\nu)}
		-(q-1)\langle\xi,Pk\rangle_{L^2(\mu)} =\langle\eta-(q-1)Q\xi,k\rangle_{L^2(\nu)}\\
		&=\delta\int_0^\infty k\dd\nu
		-2\left\langle\frac{\cR V}{V},k\right\rangle_{L^2(\nu)} =-2\theta_2\int_0^\infty h\dd\rho>0\,,
	\end{align*}
	Consequently,
	$
	\langle\xi,h\rangle_{L^2(\mu)}
	>(p-1)\langle\eta,k\rangle_{L^2(\nu)}
	>(p-1)(q-1)\theta_2\langle\xi,h\rangle_{L^2(\mu)}.
	$
	Since \eqref{sp:xi-positive} gives
	$\langle\xi,h\rangle_{L^2(\mu)}>0$, we conclude that
	$\theta_2<1/((p-1)(q-1)).$
	
	Finally, let $(\phi,\psi)$ be a radial solution of
	\eqref{eq:linearized} belonging either to
	$L^\infty(\R^N)^2$ or to $L^2(\R^N)^2$. Then
	\[
	\left\|\frac\phi U\right\|_{L^2(\mu)}^2
	=\int_0^\infty|\phi|^2U^{q-2}r^{N-1}\dd r\,,
	\qquad
	\left\|\frac\psi V\right\|_{L^2(\nu)}^2
	=\int_0^\infty|\psi|^2V^{p-2}r^{N-1}\dd r\,.
	\]
	In the bounded case these integrals are finite by
	Lemma \ref{profileregularity}; in the $L^2$ case they are finite because
	$U^{q-2}$ and $V^{p-2}$ are bounded. The same argument gives
	$V^{p-2}\psi,\ U^{q-2}\phi\in L^2(\R^N).$
	Applying $\cR$ to \eqref{eq:linearized} gives
	\[
	\phi/U=(p-1)P(\psi/V)\,,
	\qquad
	\psi/V=(q-1)Q(\phi/U)\,,
	\qquad
	T(\phi/U)=\frac{\phi/U}{(p-1)(q-1)}\,.
	\]
	If $\phi/U\ne0$, then the self-adjointness of $T$ and $T1=1$ give
	\[
	\frac{\langle\phi/U,1\rangle_{L^2(\mu)}}{(p-1)(q-1)}
	=\langle T(\phi/U),1\rangle_{L^2(\mu)}
	=\langle\phi/U,T1\rangle_{L^2(\mu)}
	=\langle\phi/U,1\rangle_{L^2(\mu)}\,.
	\]
	Since $(p-1)(q-1)>1$, we have $\phi/U\perp1$ in $L^2(\mu)$. Thus $((p-1)(q-1))^{-1}$ is an
	eigenvalue of $T$ on the orthogonal complement of the constants, so
	$((p-1)(q-1))^{-1}\le\theta_2$, contradicting $\theta_2<((p-1)(q-1))^{-1}$.
	Therefore $\phi=0$, and the second identity gives $\psi=0$.
\end{proof}

We are now in the position to prove the nondegeneracy part of the proof of Theorem \ref{thm:main}.

\begin{proof}[Proof of Theorem \ref{thm:main}: Nondegeneracy]
	Let $(\phi,\psi)$ solve \eqref{eq:linearized} in either
	$L^\infty(\R^N)^2$ or $L^2(\R^N)^2$.
	By Lemma \ref{profileregularity},
	$U^{q-2},V^{p-2}\in L^2(\R^N)\cap L^\infty(\R^N)$,
	so $V^{p-2}\psi,U^{q-2}\phi\in L^2(\R^N)$ in both cases.
	Inverting $-\Delta+1$ on tempered distributions gives
	\[
	\phi=(p-1)\cR(V^{p-2}\psi)\,,\qquad
	\psi=(q-1)\cR(U^{q-2}\phi)\,.
	\]
	Thus $(\phi,\psi)\in H^2(\R^N)^2$ by
	Lemma \ref{profileregularity}. Differentiating \eqref{system1} shows that
	$(\partial_{x_j}U,\partial_{x_j}V)$ solves \eqref{eq:linearized}.
	These translation pairs are bounded by Lemma \ref{profileregularity}, and the
	same argument places them in $H^2(\R^N)^2$.
	It therefore suffices to classify the $L^2$ solutions.
	
	We use the spherical-harmonic decomposition as in
	\cite[Section~2.1]{FrankKimPistoia}.
	Write $x=r\omega$, and choose a real orthonormal basis
	$\{Y_{\ell,m}\}$ of spherical harmonics on $\mathbb S^{N-1}$, where
	$-\Delta_{\mathbb S^{N-1}}Y_{\ell,m}
	=\ell(\ell+N-2)Y_{\ell,m}$ and $m$ indexes a basis of degree $\ell$.
	Define the coefficients by
	\[
	a_{\ell,m}(r)=\int_{\mathbb S^{N-1}}
	\phi(r\omega)Y_{\ell,m}(\omega)\dd\omega\,,\qquad
	b_{\ell,m}(r)=\int_{\mathbb S^{N-1}}
	\psi(r\omega)Y_{\ell,m}(\omega)\dd\omega\,.
	\]
	Parseval's identity on $\mathbb S^{N-1}$, integrated with respect
	to $r^{N-1}\dd r$,\\ gives
	$a_{\ell,m},b_{\ell,m}\in L^2((0,\infty),r^{N-1}\dd r)$ and
	\begin{equation}\label{fn:harmonic-expansion}
		\phi(r\omega)=\sum_{\ell,m}a_{\ell,m}(r)Y_{\ell,m}(\omega)\,,
		\qquad
		\psi(r\omega)=\sum_{\ell,m}b_{\ell,m}(r)Y_{\ell,m}(\omega)\,,
	\end{equation}
	with convergence in $L^2(\R^N)$.
	By the spherical-harmonic projection identities in
	\cite[Section~2.1]{FrankKimPistoia}, angular projection commutes
	with $\Delta$. It also commutes with multiplication by
	$U^{q-2}$ and $V^{p-2}$, since these functions are radial.
	Applying the projection to \eqref{eq:linearized}, we therefore
	find that
	$(a_{\ell,m}Y_{\ell,m},b_{\ell,m}Y_{\ell,m})$
	is an $H^2(\R^N)^2$ distributional solution of the same
	system. Using
	\[
	\Delta=\partial_r^2+\frac{N-1}{r}\partial_r
	+\frac1{r^2}\Delta_{\mathbb S^{N-1}}\,,
	\qquad
	-\Delta_{\mathbb S^{N-1}}Y_{\ell,m}
	=\ell(\ell+N-2)Y_{\ell,m}\,,
	\]
	we obtain the coefficient equations
	\begin{equation}\label{fn:mode-system}
		\begin{cases}
			-a_{\ell,m}''-\dfrac{N-1}{r}a_{\ell,m}'
			+\left(1+\dfrac{\ell(\ell+N-2)}{r^2}\right)a_{\ell,m}
			=(p-1)V^{p-2}b_{\ell,m}\,,\\[2mm]
			-b_{\ell,m}''-\dfrac{N-1}{r}b_{\ell,m}'
			+\left(1+\dfrac{\ell(\ell+N-2)}{r^2}\right)b_{\ell,m}
			=(q-1)U^{q-2}a_{\ell,m}\,,
		\end{cases}
		\qquad r>0\,.
	\end{equation}
	
	For $\ell=0$, $Y_{0,m}$ is a nonzero constant. Hence
	$(a_{0,m}(|x|),b_{0,m}(|x|))$ is a radial
	$H^2(\R^N)^2$ solution of \eqref{eq:linearized}.
	Proposition~\ref{sp:strict-gap} therefore gives
	$a_{0,m}=b_{0,m}=0$.
	
	For each $\ell\ge1$, define
	$R_\ell:L^2((0,\infty),r^{N-1}\dd r)\to L^2((0,\infty),r^{N-1}\dd r)$ by
	\begin{equation}\label{fn:green}
		(R_\ell f)(r)=\int_0^\infty\mathcal G_\ell(r,s)f(s)s^{N-1}\dd s\,,
	\end{equation}
	with
	\begin{equation}\label{fn:green2}
		\mathcal G_\ell(r,s)=(rs)^{-(N-2)/2}I_{\ell+\frac{N-2}2}(\min\{r,s\})K_{\ell+\frac{N-2}2}(\max\{r,s\})>0\,,\qquad r,s>0\,.
	\end{equation}
	The Bessel resolvent formula
	\cite[\S\S~2.2--2.3, \S~3, and Theorem~4.1]{DerezinskiRichard},
	with order $\ell+(N-2)/2\ge1$ and resolvent parameter $1$,
	applied under the change $g(r)\mapsto r^{(N-1)/2}g(r)$,
	shows that $R_\ell$ is bounded and that, for
	$f,g\in L^2((0,\infty),r^{N-1}\dd r)$,
	\begin{equation}\label{fn:radial-inverse}
		g=R_\ell f\quad\Longleftrightarrow\quad
		-g''-\frac{N-1}{r}g'+\left(1+\frac{\ell(\ell+N-2)}{r^2}\right)g=f
		\quad\text{in }(0,\infty)\,.
	\end{equation}
	The condition $\ell+(N-2)/2\ge1$ ensures uniqueness within this
	$L^2$ space, without an additional boundary condition at zero.
	The symmetry $\mathcal G_\ell(r,s)=\mathcal G_\ell(s,r)$ gives
	\begin{equation}\label{fn:radial-adjoint}
		\int_0^\infty (R_\ell f)g\,r^{N-1}\dd r
		=\int_0^\infty f(R_\ell g)r^{N-1}\dd r,
		\qquad f,g\in L^2((0,\infty),r^{N-1}\dd r)\,.
	\end{equation}
	Moreover, the order monotonicity of the modified-Bessel product
	\cite[proof of Theorem~1, item~(d)]{BariczPonnusamy} gives
	\begin{equation}\label{fn:green-difference}
		0<\mathcal G_\ell(r,s)<\mathcal G_1(r,s),\qquad \ell\ge2,\quad r,s>0\,.
	\end{equation}
	
	Since $a_{\ell,m},b_{\ell,m}\in
	L^2((0,\infty),r^{N-1}\dd r)$ and $U,V$ are bounded,
	the right-hand sides of \eqref{fn:mode-system} also belong to
	$L^2((0,\infty),r^{N-1}\dd r)$.
	Applying \eqref{fn:radial-inverse} to the two equations in
	\eqref{fn:mode-system} yields
	\begin{equation}\label{fn:projected}
		\begin{aligned}
			a_{\ell,m}=(p-1)R_\ell(V^{p-2}b_{\ell,m})\,,\qquad
			b_{\ell,m}=(q-1)R_\ell(U^{q-2}a_{\ell,m})\,.
		\end{aligned}
	\end{equation}
	Differentiating \eqref{radial_eqs} gives
	\begin{equation}\label{fn:translation-equations}
		\begin{cases}
			-U'''-\dfrac{N-1}{r}U''
			+\left(1+\dfrac{N-1}{r^2}\right)U'=(p-1)V^{p-2}V'\,,\\[2mm]
			-V'''-\dfrac{N-1}{r}V''
			+\left(1+\dfrac{N-1}{r^2}\right)V'=(q-1)U^{q-2}U'\,.
		\end{cases}
	\end{equation}
	By Lemma \ref{profileregularity},
	$U',V',V^{p-2}V',U^{q-2}U'\in L^2((0,\infty),r^{N-1}\dd r)$.
	For $\ell=1$, \eqref{fn:translation-equations} is precisely the
	radial equation in \eqref{fn:radial-inverse}, first with
	$(g,f)=(U',(p-1)V^{p-2}V')$ and then with $(g,f)=(V',(q-1)U^{q-2}U')$. Consequently,
	\begin{equation}\label{fn:translation-relations}
		\begin{aligned}
			-U'=(p-1)R_1\bigl(V^{p-2}(-V')\bigr)\,,\qquad
			-V'=(q-1)R_1\bigl(U^{q-2}(-U')\bigr)\,.
		\end{aligned}
	\end{equation}
	Both $-U'$ and $-V'$ are strictly positive on $(0,\infty)$
	by Lemma \ref{profileregularity}. Positivity of the kernel in \eqref{fn:green}-\eqref{fn:green2} and \eqref{fn:projected}
	imply that
	\begin{equation}\label{fn:absolute-comparison}
		|a_{\ell,m}|\le(p-1)R_\ell(V^{p-2}|b_{\ell,m}|)\,,
		\qquad
		|b_{\ell,m}|\le(q-1)R_\ell(U^{q-2}|a_{\ell,m}|)\,.
	\end{equation}
	All integrals below are finite by Cauchy--Schwarz, since
	$a_{\ell,m},b_{\ell,m},U',V'\in L^2((0,\infty),r^{N-1}\dd r)$, the weights
	$U^{q-2},V^{p-2}$ are bounded, and $R_\ell$ is bounded on $L^2((0,\infty),r^{N-1}\dd r)$.
	Since $U^{q-2}(-U')>0$, the kernel ordering
	\eqref{fn:green-difference}, together with
	\eqref{fn:translation-relations}, gives
	$R_\ell(U^{q-2}(-U'))\le R_1(U^{q-2}(-U'))=(-V')/(q-1)$
	for $\ell\ge1$, with equality when $\ell=1$.
	Using this inequality, \eqref{fn:absolute-comparison},
	and \eqref{fn:radial-adjoint}, we obtain
	\begin{equation}\label{fn:comparison}
		\begin{aligned}
			\int_0^\infty U^{q-2}(-U')|a_{\ell,m}|r^{N-1}\dd r
			&\le(p-1)\int_0^\infty U^{q-2}(-U')R_\ell(V^{p-2}|b_{\ell,m}|)r^{N-1}\dd r\\
			&=(p-1)\int_0^\infty V^{p-2}|b_{\ell,m}|R_\ell\bigl(U^{q-2}(-U')\bigr)r^{N-1}\dd r\\
			&\le\frac{p-1}{q-1}\int_0^\infty V^{p-2}(-V')|b_{\ell,m}|r^{N-1}\dd r\,.
		\end{aligned}
	\end{equation}
	The same computation with the two components interchanged gives
	\[
	\int_0^\infty V^{p-2}(-V')|b_{\ell,m}|r^{N-1}\dd r
	\le\frac{q-1}{p-1}\int_0^\infty U^{q-2}(-U')|a_{\ell,m}|r^{N-1}\dd r\,.
	\]
	Combining these inequalities shows that equality holds
	at every step of \eqref{fn:comparison}.
	
	For $\ell=1$, equality in the first step of
	\eqref{fn:comparison} and $U^{q-2}(-U')>0$ imply
	\[
	R_1(V^{p-2}|b_{1,m}|)
	=\frac{|a_{1,m}|}{p-1}
	=\bigl|R_1(V^{p-2}b_{1,m})\bigr|
	\quad\text{a.e.}\,.
	\]
	If the positive and negative parts of $b_{1,m}$ were both
	nonzero, their images under $R_1(V^{p-2}\,\cdot)$ would both
	be strictly positive by \eqref{fn:green}-\eqref{fn:green2}. The integral triangle
	inequality would then be strict, contradicting this equality.
	Thus $b_{1,m}$ has one sign, and \eqref{fn:projected} shows that
	every nonzero solution of the degree-one integral system has
	both components strictly of the same sign almost everywhere.
	Choose $c\in\R$ such that
	\[
	\int_0^\infty (a_{1,m}-cU')U'r^{N-1}\dd r=0\,;
	\]
	this is possible because
	$0<\int_0^\infty |U'|^2r^{N-1}\dd r<\infty$.
	By linearity and \eqref{fn:translation-relations},
	$(a_{1,m}-cU',b_{1,m}-cV')$ also solves \eqref{fn:projected}
	with $\ell=1$. If this pair were nonzero, its first component
	would have one strict sign, contradicting the displayed
	integral since $U'<0$. Hence
	$(a_{1,m},b_{1,m})=c(U',V')$.
	
	For $\ell\ge2$, positivity of $U^{q-2}(-U')$ and
	\eqref{fn:green-difference} give
	\[
	R_\ell\bigl(U^{q-2}(-U')\bigr)<R_1\bigl(U^{q-2}(-U')\bigr)=\frac{-V'}{q-1}\,.
	\]
	If $b_{\ell,m}\ne0$, the second inequality in
	\eqref{fn:comparison} would therefore be strict, and not an equality as shown before, a contradiction.
	Thus $b_{\ell,m}=0$, and \eqref{fn:projected} gives
	$a_{\ell,m}=0$. We have proved
	\begin{equation}\label{fn:sector-classification}
		\begin{cases}
			(a_{\ell,m},b_{\ell,m})
			\in\operatorname{span}\{(U',V')\}\,,&\ell=1\,,\\
			(a_{\ell,m},b_{\ell,m})=(0,0)\,,
			&\ell=0\text{ or }\ell\ge2\,.
		\end{cases}
	\end{equation}
	Since the degree-one harmonic space is
	$\operatorname{span}\{\omega_1,\ldots,\omega_N\}$,
	\eqref{fn:harmonic-expansion} and
	\eqref{fn:sector-classification} give
	$c_1,\ldots,c_N\in\R$ such that
	\[
	(\phi(r\omega),\psi(r\omega))
	=(U'(r),V'(r))\sum_{j=1}^N c_j\omega_j
	=\sum_{j=1}^N c_j(\partial_{x_j}U,\partial_{x_j}V)(r\omega)\,.
	\]
	The translation pairs are linearly independent because $U'<0$ and
	$\omega_1,\ldots,\omega_N$ are linearly independent.
	Since these pairs belong to $H^2(\R^N)^2$ and
	$H^2(\R^N)\subset H^1(\R^N)\subset L^2(\R^N)$,
	the same kernel formula holds in $H^1(\R^N)^2$ and
	$H^2(\R^N)^2$.
\end{proof}

\subsubsection{Compactness and continuation throughout the hyperbola}
\begin{lemma}[Radial Liouville theorem]\label{lem:cc-liouville}
	Under \eqref{range}, there is no positive radial $C^2(\R^N)^2$
	solution of
	\begin{equation}\label{eq:cc-massless}
		-\Delta u=v^{p-1},\quad -\Delta v=u^{q-1}\qquad\text{in }\R^N.
	\end{equation}
\end{lemma}
\begin{proof}
	For $N\ge3$, apply \cite[Corollary 3.1]{Mitidieri} with zero
	weights and exponents $p-1,q-1$. Its condition is
	$N[(p-1)(q-1)-1]<2pq$, equivalent to \eqref{range}.
	For $N=2$, $-ru'(r)=\int_0^r sv(s)^{p-1}\dd s\ge c>0$ for
	$r\ge1$, hence integration from $1$ to $R$ gives $u(R)\le u(1)-\left(\int_0^1 sv(s)^{p-1}\dd s\right)\log R$, 
	contradicting $U(R)>0$ for large $R$.
\end{proof}

Define the parameters range
\begin{equation*}
	\D:=\left\{(p,q)\in(2,\infty)^2:\frac1p+\frac1q>\frac{N-2}N\right\}\,.
\end{equation*}

\begin{proposition}[Compactness of centred profiles]\label{prop:cc-compactness}
	Let $(p_j,q_j)\to(p,q)$ in $\cD_N$ and $(U_j,V_j)$ be centred positive solutions of \eqref{system1} at $(p_j,q_j)$. Then there exists a subsequence which converges to a centred positive solution of \eqref{system1} at $(p,q)$ in $C_0(\R^N)^2$, $C^2_{\rm loc}(\R^N)^2$ and $H^1(\R^N)^2$.
\end{proposition}
\begin{proof}
	We first prove that
	$\sup_j\bigl(U_j(0)+V_j(0)\bigr)<\infty$.
	By the radial monotonicity we already known that
	$\|U_j\|_\infty=U_j(0)$ and $\|V_j\|_\infty=V_j(0)$.
	Suppose by contradiction that $\max\{U_j(0),V_j(0)\}\to+\infty$ along a subsequence. Set
	\[
	\alpha_j=\frac{2(p_j+1)}{p_jq_j-1}\,,\qquad
	\beta_j=\frac{2(q_j+1)}{p_jq_j-1}\,,\qquad
	t_j=\max\big\{U_j(0)^{1/\alpha_j},V_j(0)^{1/\beta_j}\big\}\,,
	\]
	and define
	\[
	\tU_j(x)=t_j^{-\alpha_j}U_j(x/t_j)\,,\qquad
	\tV_j(x)=t_j^{-\beta_j}V_j(x/t_j)\,.
	\]
	Since $(p_j,q_j)\to(p,q)$ satisfying \eqref{range}, then $\alpha_j,\beta_j$
	converge to positive finite limits, and hence $t_j\to+\infty$.
	The identities $p_j\beta_j=\alpha_j+2$ and
	$q_j\alpha_j=\beta_j+2$ give
	\[
	-\Delta\tU_j+t_j^{-2}\tU_j=\tV_j^{p_j-1}\,,\qquad
	-\Delta\tV_j+t_j^{-2}\tV_j=\tU_j^{q_j-1}\,,\qquad 
	0<\tU_j,\tV_j\le1\,.
	\]
	Interior elliptic estimates \cite[Chapter~9]{GilbargTrudinger} give uniform
	local $W^{2,N+1}$ bounds for $\tU_j,\tV_j$.
	After passage to a subsequence, $(U_j,V_j)$ converges weakly
	in $W^{2,N+1}_{\mathrm{loc}}(\R^N)^2$ and locally uniformly
	to $(\tU,\tV)$.
	Taking limits in the weak formulation of the rescaled equations shows that
	$(\tU,\tV)$ is a nonnegative radial weak solution of
	\eqref{eq:cc-massless}, with $\max\{\tU(0),\tV(0)\}=1$.
	Elliptic regularity gives $\tU,\tV\in C^2(\R^N)$.
	Neither component can vanish identically, since the equations
	would then force both to vanish. The strong maximum principle
	gives $\tU,\tV>0$, contradicting Lemma~\ref{lem:cc-liouville}.
	Hence $(U_j)_j$ and $(V_j)j$ are uniformly bounded.
		
	Interior elliptic estimates therefore give, after passage to a subsequence,
	$
	(U_j,V_j)\to(U,V)$
	in $C^2_{\mathrm{loc}}(\R^N)^2,
	$
	where $(U,V)$ is a nonnegative radial nonincreasing solution of
	\eqref{system1} at $(p,q)$.
	Then, by
	$U_j(0)	\le-\Delta U_j(0)+U_j(0)
	=V_j(0)^{p_j-1},\ 
	V_j(0)
	\le-\Delta V_j(0)+V_j(0)
	=U_j(0)^{q_j-1},$ we have
	$$U_j(0)\leq U_j(0)^{(p_j-1)(q_j-1)}\,,\quad\mbox{and}\quad V_j(0)\leq V_j(0)^{(p_j-1)(q_j-1)}\,.$$
	Since $(p_j-1)(q_j-1)>1$, these inequalities give
	$U_j(0),V_j(0)\ge1$, and hence $U(0),V(0)\ge1$.
	The strong maximum principle then gives $U,V>0$.
	
	We next prove that $U,V$ vanish at infinity.
	Positivity and radial monotonicity give finite limits
	$U_\infty,V_\infty\ge0$. For any sequence $|x_k|\to+\infty$
	and any fixed $R>0$,
	$
	U(|x_k|+R)\le U(x_k+x)\le U(|x_k|-R),
	\  |x|\le R,
	$
	for all sufficiently large $k$, and similarly for $V$.
	Hence $U(x_k+\cdot)\to U_\infty$ and
	$V(x_k+\cdot)\to V_\infty$ locally uniformly.
	Choose $\zeta\in C_c^\infty(\R^N)$.
	By \eqref{system1}, integration by parts, and local uniform
	convergence,
	\begin{align*}
		\int\bigl(U(x_k+x)&-V(x_k+x)^{p-1}\bigr)\zeta(x)\,\dd x
		=\int\Delta U(x_k+x)\zeta(x)\,\dd x\\
		&=\int U(x_k+x)\Delta\zeta(x)\,\dd x
		\to	U_\infty\int\Delta\zeta(x)\,\dd x=0\,.
	\end{align*}
	So we have $U_\infty=V_\infty^{p-1}$.
	Applying the same argument to the second equation also gives 
	$V_\infty=U_\infty^{q-1}$. It follows that
	$(U_\infty,V_\infty)$ is either $(0,0)$ or $(1,1)$.
	
	To exclude $(1,1)$, we use the radial equations \eqref{radial_eqs} for $U_j,V_j$:
	\begin{equation*}
		\begin{split}
			\frac{\dd}{\dd r}\left(
			U_j'V_j'-U_jV_j+\frac{U_j^{q_j}}{q_j}
			+\frac{V_j^{p_j}}{p_j}\right)
			&=(U_j''-U_j+V_j^{p_j-1})V_j'
			+(V_j''-V_j+U_j^{q_j-1})U_j'\\
			&=-\frac{2(N-1)}r U_j'V_j'\le0\,.
		\end{split}
	\end{equation*}
	The expression in parentheses tends to zero as $r\to+\infty$
	by Lemma \ref{profileregularity}, and is therefore nonnegative.
	Letting $j\to+\infty$ and using local $C^2$ convergence, we obtain
	\begin{equation}\label{eq:cc-radial-energy}
		U'V'-UV+\frac{U^q}q+\frac{V^p}p\ge0\,,\qquad r>0\,.
	\end{equation}
	Suppose that $U_\infty=V_\infty=1$. Applying interior $W^{2,N+1}$ estimates and Sobolev embedding to
	$(-\Delta+1)(U-1)=V^{p-1}-1$ gives
	$$|\nabla U(x_k)|\le C\left(\|U-1\|_{L^\infty(B_2(x_k))}+\|V^{p-1}-1\|_{L^\infty(B_2(x_k))}\right)\to 0\,,$$
	where $C$ is independent of $k$. The same argument applied to
	$(-\Delta+1)(V-1)=U^{q-1}-1$ gives $\nabla V(x_k)\to0$.
	Since $|\nabla U(x_k)|=|U'(|x_k|)|$ and the sequence
	$|x_k|\to+\infty$ was arbitrary, we have $U'(r),V'(r)\to0$ as $r\to+\infty$.
	Letting $r\to+\infty$ in \eqref{eq:cc-radial-energy} therefore yields
	\[
	0\le-1+\frac1q+\frac1p<0\,,
	\]
	a contradiction, since $p,q>2$. Thus $U_\infty=V_\infty=0$, and
	Lemma \ref{profileregularity} shows that $(U,V)$ is a centred
	positive decaying solution.
	
	Finally, let $\varepsilon>0$ and choose $R>0$ such that
	$U(R)+V(R)<\varepsilon$.
	Local uniform convergence gives
	$U_j(R)+V_j(R)<2\varepsilon$ for all sufficiently large $j$.
	Radial monotonicity then yields
	\[\sup_{r\ge R}\bigl(|U_j(r)-U(r)|+|V_j(r)-V(r)|\bigr)\le U_j(R)+V_j(R)+U(R)+V(R)<3\varepsilon\,.\]
	Together with uniform convergence on $\overline B_R$, this proves
	$\|U_j-U\|_\infty+\|V_j-V\|_\infty\to0$.
\end{proof}

\begin{lemma}\label{lem:cc-local-continuation}
	Let $(U_0,V_0)$ be a centred positive solution of
	\eqref{system1} at $(p_0,q_0)\in\D$.
	There exist $\tdelta,\varepsilon>0$ such that, for every
	$(p,q)\in\D$ with $|p-p_0|+|q-q_0|<\tdelta$,
	\eqref{system1} has exactly one centred positive solution
	$(U_{p,q},V_{p,q})$ satisfying
	\[
	\|U_{p,q}-U_0\|_\infty+\|V_{p,q}-V_0\|_\infty<\varepsilon\,.
	\]
	Moreover, the map
	$(p,q)\mapsto(U_{p,q},V_{p,q})$ is $C^1$ on this parameter
	neighbourhood, with values in $C_{0,\mathrm{rad}}(\R^N)^2$.
\end{lemma}
\begin{proof}
	We define the map $\mathcal F:\D\times C_{0,\mathrm{rad}}(\R^N)^2\to C_{0,\mathrm{rad}}(\R^N)^2$ such that
	\begin{equation}\label{eq:cc-fixed-point}
		\mathcal F(p,q,U,V)=\bigl(U-\cR(V_+^{p-1}),\ v-\cR(U_+^{q-1})\bigr)\,.
	\end{equation}
	Recall that $\cR$ is a bounded operator on $C_0(\R^N)$; moreover, it is clear from
	the representation \eqref{eq:radial-resolvent} that	$\cR f$ is radial whenever $f$ is radial.
	The scalar derivatives
	$
	\partial_t(t_+^{p-1})=(p-1)t_+^{p-2},\ 
	\partial_p(t_+^{p-1})=t_+^{p-1}\log t_+,
	$
	where the second expression is defined as zero for $t\le0$,
	are uniformly continuous when $p$ ranges over a compact
	subinterval of $(1,\infty)$ and $t$ over a bounded interval.
	The mean-value formula therefore shows that
	$(p,w)\mapsto w_+^{p-1}$ is $C^1$ in the uniform norm.
	Consequently, $\mathcal F$ is a $C^1$ map.
	
	For every $a\in C_0(\R^N)$, the operator: 
	\begin{equation}\label{eq:cc-compact-resolvent}
		C_0(\R^N)\to C_0(\R^N),\quad h\mapsto\cR(ah)\quad\text{is compact.}
	\end{equation}
	Indeed, for $\|h\|_\infty\le1$,
	\eqref{eq:bessel-potential} and $\|G_N\|_{L^1(\R^N)}=1$ give
	\[
	\|\cR(ah)\|_\infty\le\|a\|_\infty,\qquad
	|\cR(ah)(x+z)-\cR(ah)(x)|
	\le\|a\|_\infty
	\|G_N(\cdot+z)-G_N\|_{L^1(\R^N)}\,.
	\]
	Splitting the convolution integral into $|y|\le R$ and $|y|>R$
	also gives
	\[
	|\cR(ah)(x)|
	\le \|a\|_\infty
	\int_{\{|y|\ge |x|-R\}}G_N(y)\dd y
	+\sup_{|y|>R}|a(y)|\,,\qquad |x|>R\,.
	\]
	Translation continuity in $L^1$ gives equicontinuity.
	Since $a\in C_0$ and $G_N\in L^1$, first choosing $R$ large
	and then $|x|$ large gives uniform decay at infinity.
	Ascoli--Arzel\`a theorem therefore provides compactness in $C_0(\R^N)$.
	If $a$ is radial, the same conclusion holds on $C_{0,\mathrm{rad}}(\R^N)$.
	
	At a centred positive solution $(U,V)$, one has
	$$D_{(U,V)}\mathcal F(p,q,U,V)[\phi,\psi]=\bigl(\phi-(p-1)\cR(V^{p-2}\psi),\ \psi-(q-1)\cR(U^{q-2}\phi)\bigr)$$
	for $(\phi,\psi)\in C_{0,\mathrm{rad}}(\R^N)^2$. Because $U^{q-2},V^{p-2}\in C_{0,\mathrm{rad}}(\R^N)$,
	\eqref{eq:cc-compact-resolvent} shows that this derivative
	is the identity minus a compact operator.
	If $(\phi,\psi)$ belongs to its kernel, then
	\[
	\phi=(p-1)\cR(V^{p-2}\psi)\,,\qquad
	\psi=(q-1)\cR(U^{q-2}\phi)\,.
	\]
	Applying $-\Delta+1$ shows that $(\phi,\psi)$ is a bounded
	radial solution of \eqref{eq:linearized}.
	Proposition~\ref{sp:strict-gap} gives then $\phi=\psi=0$.
	The Fredholm alternative \cite[Theorem~6.6(c)]{Brezis} thus implies that $D_{(U,V)}\cF$ is also surjective.
	It is therefore bijective, and its inverse is bounded by the
	bounded inverse theorem.
	
	For $\varepsilon>0$ the implicit function theorem gives $\tdelta>0$
	such that, for every $(\widetilde p,\widetilde q)\in\D$ with
	$|\widetilde p-p|+|\widetilde q-q|<\delta$, there exists exactly
	one pair $(\widetilde U,\widetilde V)
	\in C_{0,\mathrm{rad}}(\R^N)^2$ satisfying
	\[
	\mathcal F(\widetilde p,\widetilde q,\widetilde U,\widetilde V)=0\,,
	\qquad
	\|\widetilde U-U\|_\infty+\|\widetilde V-V\|_\infty<\varepsilon\,.
	\]
	Moreover, the map
	$(\widetilde p,\widetilde q)\mapsto(\widetilde U,\widetilde V)$
	is $C^1$ on this parameter neighbourhood.
	
	Choose $\varepsilon<\|U\|_\infty+\|V\|_\infty$, decreasing
	$\tdelta$ if necessary, so that $(\widetilde U,\widetilde V)\ne(0,0)$.
	By \eqref{eq:cc-fixed-point} and $G_N>0$, both components are
	strictly positive: a zero of either component would force both
	to vanish identically.
	Elliptic regularity and Lemma \ref{profileregularity} then show that
	$(\widetilde U,\widetilde V)$ is a centred positive decaying solution.
	Conversely, every nearby centred positive solution satisfies
	$\mathcal F=0$ by \eqref{eq:nonlinear-resolvent}, so the uniqueness
	already proved applies.
\end{proof}

We are finally in the position to complete the proof of Theorem \ref{thm:main} by proving that, for each $(p,q)\in\D$, there is exactly one
centred positive solution.

\begin{proof}[Proof of Theorem~\ref{thm:main}: Uniqueness]
	Fix $(p_0,q_0)\in\D$. The set of centred positive solutions
	at $(p_0,q_0)$ is compact by Proposition~\ref{prop:cc-compactness}, and each solution is
	isolated by Lemma~\ref{lem:cc-local-continuation}.
	Hence there are only finitely many such solutions. Indeed, otherwise, compactness would give a sequence of pairwise distinct
	solutions converging, after passage to a subsequence, to a solution
	$(U,V)$. Local uniqueness in Lemma~\ref{lem:cc-local-continuation} gives a neighbourhood of $(U,V)$
	containing no other solution at $(p_0,q_0)$, which is a contradiction.
	
	We now claim that the number of centred positive solutions
	is constant for parameters sufficiently close to $(p_0,q_0)$.
	Suppose that there exist $m\ge1$ solutions at $(p_0,q_0)$, and choose pairwise
	disjoint neighbourhoods of these solutions in
	$C_{0,\mathrm{rad}}(\R^N)^2$.
	By Lemma~\ref{lem:cc-local-continuation}, the neighbourhoods
	can be chosen so that every parameter sufficiently close to
	$(p_0,q_0)$ has exactly one solution in each of them.
	There can be no additional solutions outside their union.
	Indeed, such solutions along a sequence
	$(p_j,q_j)\to(p_0,q_0)$ would, by
	Proposition~\ref{prop:cc-compactness}, converge along a
	subsequence to one of the $m$ solutions at $(p_0,q_0)$.
	They would therefore eventually belong to its neighbourhood,
	a contradiction.
	If there is no solution at $(p_0,q_0)$, the same compactness
	proposition excludes solutions along any sequence of parameters
	converging to $(p_0,q_0)$.
	Thus the number of centred positive solutions is locally
	constant throughout $\D$.
	
	Now fix $(p,q)\in\D$ and set $s=\min\{p,q\}$.
	The segment
	\[
	(p_t,q_t)=(s+t(p-s),s+t(q-s)),\qquad 0\le t\le1,
	\]
	lies in $\D$, since
	\[
	\frac1{p_t}+\frac1{q_t}
	\ge\frac1p+\frac1q>\frac{N-2}N\,.
	\]
	At $(s,s)$, existence of a centred positive solution follows
	from \cite[Theorems~1.7--1.9]{BMR} and Lemma \ref{profileregularity}.
	For any solution at $(s,s)$, subtracting the two equations
	and testing with $U-V\in H^1(\R^N)$ gives
	\[
	\int\bigl(|\nabla(U-V)|^2+|U-V|^2\bigr)\dd x=-\int(U^s-V^s)(U-V)\dd x\le0\,.
	\]
	Therefore $U=V$. Since $(s,s)\in\D$, the exponent $s$ is
	scalar subcritical, and Kwong's theorem~\cite{Kwong} gives
	uniqueness of the centred positive solution of
	$-\Delta U+U=U^s$.
	Consequently, there is exactly one centred positive solution
	at $(s,s)$.
	The number of centred positive solutions is locally constant
	along the segment and hence constant on $[0,1]$: it therefore equals $1$ at $(p,q)$.
	The symmetry result in Lemma \ref{profileregularity} gives in turn uniqueness
	up to common translations.
	Uniqueness and Lemma~\ref{lem:cc-local-continuation} also imply
	that the centred solution depends $C^1$ on $(p,q)$ with values
	in $C_{0,\mathrm{rad}}(\R^N)^2$.
\end{proof}

\section{Normalised solutions for \texorpdfstring{\eqref{system}}{(system)}: Existence, Classification, and further results}\label{Section_normalised}

In light of the existence and uniqueness result for system \eqref{system1} in Theorem \ref{thm:main}, obtained in the previous section, we are now ready to present the proof of Theorem \ref{introexistencescaling} concerning normalised solutions of system \eqref{system}. 
\begin{proof}[Proof of Theorem~\ref{introexistencescaling}]
	First, we deal with the scaling connection between \eqref{system} and \eqref{system1}.
	In order for the ansatz 
	$$u_\lambda(x)=\lambda^\xi U(\sqrt\lambda x)\,,\quad v_\lambda(x)=\lambda^\eta V(\sqrt\lambda x)\,,$$
	to be a solution of the equations in \eqref{system}, necessarily $(\xi,\eta)$ must satisfy $\xi+1=(p-1)\eta$ and $\eta+1=(q-1)\xi$.
	The unique solution is $\xi=p/D$, $\eta=q/D$ with $D$ defined in \eqref{parameters}; a simple change of variables gives \eqref{masslaw}.
	The inverse transformation is
	$$U(y)=\lambda^{-p/D}u(y/\sqrt\lambda)\,,\quad V(y)=\lambda^{-q/D}v(y/\sqrt\lambda)\,.$$
	Such transformation clearly preserve positivity and the decay class. 
	Lemma \ref{dualexistence} shows the existence a profile in $\mathcal S$, while the scaling law \eqref{masslaw} then selects the frequency \eqref{frequencychoice}. 
	Furthermore, Theorem~\ref{thm:main} identifies $\mathcal S$, modulo translations,
	with one profile, therefore necessarily the least-action one constructed in
	Theorem~\ref{introexistencescaling}. For $\gamma\ne0$, the mass law chooses the
	one frequency given in \eqref{frequencychoice}. Note that the Poho\v zaev identities in Proposition \ref{pohozaev} imply
	$B(U,V)=(\gamma+1)I_1(U,V)$, which for $\gamma=0$ reduces to $B(U,V)=I_1(U,V)$. Hence, \eqref{dualaction} gives the unique mass $a_*=c_+$ in
	\eqref{introcriticalmass}, and \eqref{masslaw} preserves it for all $\lambda>0$: every positive frequency produces a solution of that mass. 
\end{proof}

\begin{remark}
	Observe that, by the identity \eqref{Jidentity} in Proposition \ref{poho}, positive solutions have negative, zero or positive energy according
	as $\sigma$ is larger than, equal to or smaller than $N/(N+2)$, namely in dependence of the sign of $\gamma_N(p,q)$.
	Also, \eqref{Bidentity} rules out such positive finite-energy solutions with
	$\lambda>0$ at or above the Sobolev-critical hyperbola. 
\end{remark}


%


\begin{remark}[The diagonal case]
	If $p=q=m$ and $(u,v)$ is a positive decaying solution of
	\eqref{system} with $\lambda>0$, then $u=v$. Indeed, subtraction and testing with $u-v$ yield
	\[
	\int|\nabla(u-v)|^2+\lambda\int (u-v)^2=\int(v^{m-1}-u^{m-1})(u-v)\le0\,,
	\]
	thus $u-v=0$. Therefore, applying Theorem \eqref{introexistencescaling} gives that, if $m\neq2+\frac4N$ the normalised positive
	solution is unique up to translation, while at $m=2+\frac4N$ its only mas is $\|Q_m\|_2^2$ and its frequencies vary in the whole $(0,\infty)$. Hence, we retrieve the well-known scalar results in \cite{Kwong}.
\end{remark}

\begin{corollary}[Uniqueness of the nonnegative dual optimiser]
	The nonnegative normalised maximising pair in \eqref{Cpq}
	is unique up to a common translation.
\end{corollary}
\begin{proof}
	Every nonnegative maximiser satisfies the multiplier equations in the
	proof of Theorem~\ref{introexistencescaling} and produces a member of
	$\mathcal S$ by the same fixed amplitude factors. Theorem~\ref{thm:main}
	identifies these profiles. Reversing the amplitude factors and taking
	the powers $q-1,p-1$ identifies the original normalised maximisers.
\end{proof}

\subsection{Frequency and energy of normalised solutions}
We finally investigate the dependence of the frequency and of the energy of the (unique) normalised solution on the mass $a$. Recall the definition of $\gamma$ in \eqref{parameters}, and note that $\gamma>-1$ by \eqref{range}. 


\begin{proposition}[Energy and frequency of positive normalised solutions]\label{introleastenergy}
	Assume \eqref{range}, $\gamma_N(p,q)\ne0$, and let $a>0$. Let $(\lambda,u,v)$ be the positive normalised solution triple, which is unique up to common translation by Theorem~\ref{introexistencescaling}. Then
	\begin{equation}\label{introgroundfrequency}
		\lambda=\left(\frac{a}{(\gamma+1)c_+}\right)^\frac1\gamma.
	\end{equation}
	Moreover, setting $E_+(a):=J(u,v)$, then
	$$E_+(a)=-\gamma c_+ \left(\frac{a}{(\gamma+1)c_+}\right)^{\frac{\gamma+1}\gamma}.$$
	In particular, $E_+'(a)=-\lambda$.
\end{proposition}

\begin{proof}
	Let $(U,V)\in\cS$ be the unit-frequency positive profile of action $I_1(U,V)=c_+$ given by Theorem \ref{thm:main}, then by Theorem \ref{introexistencescaling} then $(u_\lambda,v_\lambda)$ solve \eqref{system}. The Pohozaev identities in Proposition~\ref{pohozaev} yield
	\[
	B(U_+,V_+)=(\gamma+1)c_+\,,
	\qquad
	J(U_+,V_+)=-\gamma c_+\,,
	\]
	hence, the mass scaling law \eqref{masslaw} implies
	\[
	a=\lambda^\gamma B(U_+,V_+)=\lambda^\gamma(\gamma+1)c_+\,.
	\]
	Since $\gamma_N(p,q)\ne0$, we obtain \eqref{introgroundfrequency}. Moreover, the scaling law for $J$ gives
	\begin{equation*}
		E_+(a)=J(u_\lambda,v_\lambda)=\lambda^{\gamma+1}J(U_+,V_+)=-\gamma c_+\left(\frac a{(\gamma+1)c_+}\right)^{\frac{\gamma+1}{\gamma}}.
	\end{equation*}
	Differentiating with respect to $a$, yields $E_+'(a)=-\lambda$.
\end{proof}

\section{A new Gagliardo-Nirenberg type inequality}\label{Section_GN}
The question underlying Theorem~\ref{introinterpolation} is whether
the bilinear mass critical curve \eqref{criticalcurve} detected by the scaling also governs
the competition between the bilinear Dirichlet energy $A$ and the nonlinear potential $F$ on the bilinear mass sphere $\Sigma_a$ \eqref{Sigma_a}. An estimate of $F$ in terms of powers of $A$ and $B$ would make this competition explicit, exactly as in the scalar theory. A delicate part of the problem is the domain in which such an inequality can hold: the crossed forms $A$ and $B$ can change sign, and even their positivity does not control
the two components. We first identify this obstruction, then derive
the only scaling-compatible exponents and finally construct a class on which
the desired bound holds. 
\vskip0.2truecm
Set
$$\mathcal P=\{(u,v)\in X:A(u,v)>0,\ B(u,v)>0\}\,.$$

\begin{proposition}[Failure on the full positive set]\label{obstruction}
	No finite constant can bound $F(u,v)$ solely in terms of fixed positive
	values of $A(u,v)$ and $B(u,v)$ on $\mathcal P$. Moreover, for every $a>0$,
	\[
	 \inf_{\Sigma_a}J=-\infty\,,
	\]
	even if the infimum is restricted to strictly positive smooth pairs with
	$A>0$.
\end{proposition}
\begin{proof}
	Choose a strictly positive Schwartz function $\phi$ with $\int\phi^2=a$ and set $u_t=t\phi$, $v_t=t^{-1}\phi$. For every $t>0$, one has
	$$B(u_t,v_t)=a\,,\qquad A(u_t,v_t)=\int|\nabla\phi|^2\,,\qquad F(u_t,v_t)=\frac{t^q}{q}\int\phi^q+\frac{t^{-p}}p\int\phi^p\,.$$
	The last expression tends to infinity as $t\to+\infty$ or $t\to0$.
\end{proof}

In particular, positivity, radial symmetry, proportionality of the
components, or simultaneous equality in the two Cauchy--Schwarz
inequalities cannot by themselves prove the desired estimate \eqref{introGN}.
The missing ingredient is a quantitative control of the relative sizes
of the components. To compare with the usual vector-valued estimates,
let $\eta,\theta>0$ and $2<\ell=\eta+\theta<2^*$. H\"older's inequality and the
scalar GN inequality give, for $u,v\in H^1(\R^N)$,
\begin{equation*}
	\int |u|^\eta|v|^\theta\le C\bigl(\|\nabla u\|_2^2+\|\nabla v\|_2^2\bigr)^{N(\ell-2)/4}\bigl(\|u\|_2^2+\|v\|_2^2\bigr)^{\ell/2-N(\ell-2)/4}.
\end{equation*}
Sharp estimates of this type for coupled gradient systems are studied in \cite{MZ}. However, the right-hand side involves sums of squared norms, whereas the desired estimate for $F$ must be expressed in terms of the cross terms $A$ and $B$. Cauchy--Schwarz provides bounds on these cross terms in terms of the sums, but not conversely. As shown by the family constructed in Proposition~\ref{obstruction}, no such reverse control holds in general.
\vskip0.2truecm

\subsection{The admissible sets and the proof of Theorem~\texorpdfstring{\ref{introinterpolation}}{(introinterpolation)}}
In order to first detect the only possible positive exponents $\alpha,\beta$ in \eqref{introGN}, for $\mu,t>0$ let us introduce the two-parameter transformation $T_{\mu,t}:\cP\to\cP$ defined by
\begin{equation*}
	T_{\mu,t}(u,v)=\bigl(\mu^{1/q}u(t\,\cdot),\mu^{1/p}v(t\,\cdot)\bigr)\,.
\end{equation*}
Direct computation gives
\begin{equation}\label{Tlaws}
 \begin{aligned}
 A(T_{\mu,t}(u,v))&=\mu^\sigma t^{2-N}A(u,v),\\
 B(T_{\mu,t}(u,v))&=\mu^\sigma t^{-N}B(u,v),\\
 F(T_{\mu,t}(u,v))&=\mu t^{-N}F(u,v).
 \end{aligned}
\end{equation}
An estimate of the kind \eqref{introGN} is homogeneous under both
parameters $\mu$ and $t$ precisely when
$\sigma(\alpha+\beta)=1$ and $(2-N)\alpha-N\beta=-N$.
Therefore the only possible exponents $\alpha$ and $\beta$ are the ones in \eqref{alphabeta}. 
Both are positive under \eqref{range}; $\alpha=1$ is corresponds to the
mass-critical case $\gamma_N(p,q)=0$. For $p=q=m$ and $u=v$, these are the exponents
of \eqref{scalarGN} written in terms of squared norms.

\begin{remark}
	The same critical balance can be seen directly in the conjugate
	Sobolev spaces. If $\gamma_N(p,q)=0$, choose
	\begin{equation*}
		r=\frac{Nq}{N+2}\,,\qquad s=\frac{Np}{N+2}\,,\quad\mbox{with}\quad\frac1r+\frac1s=1\,.
	\end{equation*}
	These exponents satisfy Lemma~\ref{setting}, since $p,q>2$, and
	equivalently $q=r\big(1+\frac2N\big)$ and $p=s\big(1+\frac2N\big)$. Taking the powers $q$ and $p$ of the two
	inequalities in \eqref{GNrs} gives
	\[
	\|u\|_q^q\le C\|\nabla u\|_r^2\|u\|_r^{q-2}\,,\qquad
	\|v\|_p^p\le C\|\nabla v\|_s^2\|v\|_s^{p-2}\,.
	\]
	Thus both nonlinear terms become quadratic in their component
	gradient norms at exactly the bilinear-mass critical curve. 
\end{remark}
\vskip0.2truecm
For $(u,v)\in\mathcal P$ put
\begin{equation}\label{normalisation}
 \widehat t=\sqrt{B(u,v)/A(u,v)}\,,\quad\ 
 \widehat\mu=\left(\frac{B(u,v)^{\frac N2-1}}{A(u,v)^{\frac N2}}\right)^\frac1\sigma\,,
 \quad\ (\widehat u,\widehat v)=T_{\widehat\mu,\widehat t}(u,v)\,.
\end{equation}
Then \eqref{Tlaws} implies
\begin{equation}\label{normalised}
 A(\widehat u,\widehat v)=B(\widehat u,\widehat v)=1\,,\qquad
 F(\widehat u,\widehat v)=\frac{F(u,v)}{A(u,v)^\alpha B(u,v)^\beta}\,.
\end{equation}
The transformation sending $(A,B)$ to $(1,1)$ is unique, and it is not difficult to check that
\begin{equation}\label{hatinvariance}
 \widehat{T_{\mu,t}(u,v)}=(\widehat u,\widehat v)\,.
\end{equation}
The scaling transformations satisfy \(T_{a,b}T_{\mu,t}=T_{a\mu,bt}\). Consequently, the normalisation parameters for \(T_{\mu,t}(u,v)\) are \(\widehat\mu/\mu\) and \(\widehat t/t\), so that normalising the transformed pair yields exactly the same normalised pair as normalising \((u,v)\) directly. 

This is the analogue of normalising the mass and gradient norm in
the scalar variational characterisation of the sharp GN constant
\cite{Weinstein}. Here the associated quotient is
$$\mathcal Q_{\rm GN}(u,v)=\frac{A(u,v)^\alpha B(u,v)^\beta}{F(u,v)}\,,\qquad (u,v)\in\cP\,.$$
A positive lower bound for this quotient is equivalent to the
desired inequality. However, imposing $A=B=1$ does not bound a
sequence in $X$: one component may grow while the other decreases.
Indeed Proposition~\ref{obstruction} implies
$\inf_{\mathcal P}\mathcal Q_{\rm GN}=0$. Thus bilinear
normalisation alone gives neither a positive lower bound for the quotient nor the
compactness required for attainment. We impose component bounds
after normalisation to obtain a uniform estimate while retaining
both scaling parameters.
\begin{remark}
The tilde map of Subsection~\ref{sectildepreliminaries} explains the
positive and negative geometry of the action, but its positive
graph need not be preserved by these transformations. Already in
the weighted Hilbert realisation, if $v=b^{-2}u$ and
$(u_{\mu,t},v_{\mu,t})=T_{\mu,t}(u,v)$, then
\[
 v_{\mu,t}=b^{-2}\mu^{1/p-1/q}u_{\mu,t}.
\]
Thus, for $p\ne q$ the graph with fixed $b$ is not invariant
under arbitrary $\mu$. Positivity of the action pairing on a tilde
graph and invariance of a domain for the homogeneous GN quotient
are therefore distinct requirements. The following construction enforces the
latter explicitly.
\end{remark}
Fix conjugate exponents $r,s>1$ as in Lemma~\ref{setting}. For $K>0$, recall that the set $\cM_K$ is defined in \eqref{MK}.

\begin{proof}[Proof of Theorem~\ref{introinterpolation}]
	Choose a nonzero nonnegative smooth function $\phi\in C^\infty_c(\R^N)$ and apply
	\eqref{normalisation} to $(\phi,\phi)\in\cP$. The resulting pair has
	$A=B=1$ and $(\phi,\phi)\in\cM_K$ with
	$$K:=\max\{\|\widehat\phi\|_r\,,\,\|\nabla\widehat\phi\|_r\,,\,\|\widehat\phi\|_s\,,\,\|\nabla\widehat\phi\|_s\}<\infty$$
	since $\widehat\phi\in C^\infty_c(\R^N)$. Invariance of $\cM_K$ under $T_{\mu,t}$ follows from
	\eqref{hatinvariance}, while translation invariance is clear because norms enjoy this property.
	Fix now $a>0$ and consider $(u,v)\in\cM_K$ for some $K>0$. By \eqref{normalised} $(\hu,\hv)\in\cM_K$ with $B(\hu,\hv)=1$, we claim that
	$T_{a^{1/\sigma},1}(\widehat u,\widehat v)\in\cM_K$ has bilinear mass $a$. Indeed,
	\begin{equation*}
		B\big(T_{a^{1/\sigma},1}(\hu,\hv)\big)=B\left(a^\frac1{q\sigma}\hu,a^\frac1{p\sigma}\hv\right)=a^{\frac1\sigma\left(\frac1p+\frac1q\right)}B(\hu,\hv)=a
	\end{equation*}
	by virtue of \eqref{range} and \eqref{normalised}. Finally apply \eqref{GNrs} to $(\widehat u,\widehat v)$. By definition of $\cM_K$ it is clear that
	$$F(\widehat u,\widehat v)\le C_{N,r,q}^q\frac{K^q}q+C_{N,s,p}^p\frac{K^p}p\,.$$
	Use eventually \eqref{normalised} to undo the normalisation and obtain \eqref{introGN}.
\end{proof}

To stress the role of the bounds in $\cM_K$, observe that uniform control of the normalised component norms is exactly
what the bilinear normalisations alone fail to provide.
For example, the reciprocal-amplitude family in
Proposition~\ref{obstruction} has fixed $A,B$ and hence fixed
normalisation parameters, but unbounded normalised component norms, hence it eventually leaves each fixed $\mathcal M_K$.

\begin{remark}[A Hilbert-space specialisation]
	In the individually subcritical range $2<p,q<2^*$, one can instead
	choose $r=s=2$. For $0<\delta<1$ let
	\begin{equation}\label{Mdelta}
	 \mathcal H_\delta=\big\{(u,v)\in H^1\times H^1:A,B>0,
	 \ \|\widehat u\|_2^2,\|\nabla\widehat u\|_2^2,
	 \|\widehat v\|_2^2,\|\nabla\widehat v\|_2^2\le\delta^{-1}\big\}\,.
	\end{equation}
	Since $A(\widehat u,\widehat v)=B(\widehat u,\widehat v)=1$, Cauchy--Schwarz also bounds
	all four squared norms from below by $\delta$.
	Thus \eqref{Mdelta} is the closed version of imposing the four reciprocal
	ratios in a normalised Hilbert-space construction between $\delta$ and
	$\delta^{-1}$.
	
	This class is nonempty for every $0<\delta<1$. Indeed choose a smooth
	$w$ with $\int w^2=\int|\nabla w|^2=1$. 
	Then $(w,w)$ is already normalised and all four squared norms are one.
	Consequently Theorem~\ref{introinterpolation} holds on $\mathcal H_\delta$ with
	$K=\delta^{-1/2}$. 

\end{remark}

\subsection{An equivalent inhomogeneous estimate}
An alternative approach is to consider the quotient $(A+B)/F^\sigma$,
whose numerator is the quadratic part of the action at frequency one $I_1$.
Minimising with respect to the spatial scaling yields the homogeneous GN estimate, thereby linking the fixed-frequency and prescribed-mass problems at the level of inequalities. 
To establish the reverse implication, the domain must be invariant under the scaling transformations used in the argument, and in fact every $\mathcal M_K$ has this property.

\begin{proposition}\label{equivalent}
	Let $\mathcal E\subset\mathcal P$ be invariant under every $T_{\mu,t}$.
	The inequality \eqref{introGN} on $\mathcal E$ is equivalent to the existence of $c>0$ such that
	\begin{equation}\label{inhomogeneous}
		A(u,v)+B(u,v)\ge c F(u,v)^\sigma\quad\hbox{on }\mathcal E\,.
	\end{equation}
\end{proposition}
\begin{proof}
	Set $h=N(1-\sigma)\in(0,2)$ by \eqref{range}. 
	Suppose first \eqref{introGN} holds. By \eqref{alphabeta} one has
	$$F(u,v)^\sigma\le C^\sigma A(u,v)^{\alpha\sigma} B(u,v)^{1-\alpha\sigma},$$
	and the weighted arithmetic--geometric mean gives
	$$A(u,v)^{\alpha\sigma}B(u,v)^{1-\alpha\sigma}\leq\alpha\sigma A(u,v)+(1-\alpha\sigma)B(u,v)\le A(u,v)+B(u,v)\,,$$
	since $\alpha\sigma=\frac h2\in(0,1)$. Conversely, assume now \eqref{inhomogeneous} holds in $\cE\subset\cP$. Apply \eqref{inhomogeneous} to $T_{t^N\!,t}(u,v)$ and a simple computation shows that
		$$A(T_{t^N\!,t}(u,v))=t^{2-h}A(u,v)\,,\ \ B(T_{t^N\!,t}(u,v))=t^{-h}B(u,v)\,,\ \ F(T_{t^N\!,t}(u,v))=F(u,v)\,.$$
	Thus
	\[
	cF(u,v)^\sigma\le At^{2-h}+Bt^{-h}\quad\mbox{for all}\ \,t>0\,.
	\]
	Minimisation in $t$ of the right-hand side gives
	$$cF(u,v)^\sigma\le k_h A(u,v)^\frac h2 B(u,v)^{1-\frac h2}=k_hA(u,v)^{\alpha\sigma} B(u,v)^{1-\alpha\sigma}$$
	with
	$$k_h=\left(\frac h{2-h}\right)^{1-\frac h2}+\left(\frac h{2-h}\right)^{-\frac h2}\!.$$
	By \eqref{alphabeta} taking the $1/\sigma$ power gives the desired exponents in \eqref{introGN}.
\end{proof}
%

\section{The constrained energy geometry}\label{secgeometry}
Fix $K$ for which $\mathcal M_K\ne\emptyset$. Recall that the mass slice is defined by $\cM_K^a=\cM_K\cap\Sigma_a$, and the best constant $C_K^*$ in \eqref{introGN} is given in \eqref{C_K*}. Observe that such supremum has the same value on every mass slice, because $T_{\mu,1}$ changes the mass arbitrarily and preserves the ratio. 

We next show that the energy functional $J$ on its dilation fibres sets \eqref{fiber_set} has the same behaviour appearing in the scalar theory of \cite{Weinstein,Jeanjean,BS,BScorrection}, according to the position of the exponents $(p,q)$ with respect to the critical curve \eqref{critical_curve}, now on the explicit invariant sets required by the bilinear structure. Let us define first the mass-preserving dilation by
\begin{equation*}
	S_t(u,v)=\left(t^{N/(q\sigma)}u(t\,\cdot),t^{N/(p\sigma)}v(t\,\cdot)\right)=T_{t^{N/\sigma},t}(u,v)\,.
\end{equation*}
Since $B(S_t(u,v))=B(u,v)$, it preserves $\cM_K^a$ and, by \eqref{Tlaws} and \eqref{alphabeta}, satisfies
\begin{equation}\label{fiber}
	A(S_t(u,v))=t^2A(u,v),\quad F(S_t(u,v))=t^{2\alpha}F(u,v)\,,\quad
	J(S_t(u,v))=t^2A(u,v)-t^{2\alpha}F(u,v)\,.
\end{equation}
\begin{proof}[Proof of Theorem~\ref{introenergytrichotomy}]
	The bilinear Gagliardo-Nirenberg inequality \eqref{introGN} implies that,
	\begin{equation}\label{lowerbound}
	 J(u,v)\ge A(u,v)-C_K^*a^\beta A(u,v)^\alpha.
	\end{equation}
	hence
	\begin{equation*}
		J(S_t(u,v))\geq A(u,v)\left(t^2-t^{2\alpha}C_K^*A(u,v)^{\alpha-1} a^\beta\right).
	\end{equation*}
	When $\gamma_N(p,q)>0$, i.e. $\alpha\in(0,1)$, it is then clear that $J(u,v)\to+\infty$ along any sequence in $\cM_K^a$
	with $A(u,v)\to+\infty$. Moreover, the minimum of the map $t-ct^\alpha$ over $t>0$ is
	$-(1-\alpha)\alpha^{\alpha/(1-\alpha)}c^{1/(1-\alpha)}$ and $J(S_t(u,v))\geq J(S_{t_\alpha}(u,v))$
	gives the lower bound in \eqref{subinf}.
	
	In case $\gamma_N(p,q)=0$, i.e. $\alpha=1$, \eqref{lowerbound} reads
	$$J(u,v)\ge(1-C_K^*a^{2/N})A(u,v)\,.$$
	For $a\leq a_K$ one has $J(u,v)\geq0$ and by \eqref{fiber} $J(S_t(u,v))$ tends to zero as $t\to0$ for any fixed $(u,v)\in\cM_K^a$, proving that $\inf_{\cM_K^a}J=0$.
	If $a<a_K$, positivity of $J(S_t(u,v))$ is strict. If instead $a>a_K$, by using the definition of $C_K^*$, there exists $(u,v)\in\cM_K^a$ such that $F(u,v)>A(u,v)$. Indeed, fixing $\varepsilon>0$ there exists $(u,v)\in\cM_K^a$ such that
	\begin{equation*}
		C_K^*-\varepsilon\leq\frac{F(u,v)}{A(u,v)^\alpha a^\frac2N}<C_K^*\frac{F(u,v)}{A(u,v)}\,,
	\end{equation*}
	which clearly implies the claim. On its fiber set \eqref{fiber_set}, the energy tends to $-\infty$ as $t\to+\infty$, since $J(S_t(u,v))=t^2(A(u,v)-F(u,v))$. 
	In case $a=a_K$, similar computations show that $(u,v)\in\cM_K^a$ is such that $A(u,v)=F(u,v)$, equivalently $J(u,v)=0$, if and only if $(u,v)$ is an optimiser of the bilinear Gagliardo-Nirenberg inequality \eqref{introGN}. A best-constant pair can be rescaled to mass $a_K$.
	
	If $\gamma_N(p,q)<0$, i.e. $\alpha>1$, \eqref{fiber} shows that $J$ tends to $-\infty$ as $t\to+\infty$ on any fiber set.
	Moreover, by \eqref{introGN} for $(u,v)\in\cM_K^a$ one finds
	\begin{equation*}
		J(u,v)\geq A(u,v)-C_K^*A(u,v)^\alpha B(u,v)^\beta=A(u,v)\left(1-C_K^*a^\beta A(u,v)^{\alpha-1}\right),
	\end{equation*}
	so taking $0<A(u,v)\leq\rho$ with $\rho=(2C_K^*a^\beta)^{-1/(\alpha-1)}$ gives
	\eqref{barrier}. Finally, differentiation of \eqref{fiber} gives
	$2tA-2\alpha t^{2\alpha-1}F$, which changes sign exactly once,
	at the displayed $t_*$. At that point the second derivative is
	$4(1-\alpha)A<0$.
\end{proof}


\begin{remark}
	A derivative in $t$ of $J$ composed with the fiber map yields
	$$P(u,v)=\left.\frac{\dd}{\dd t}J(S_t(u,v))\right|_{t=1}=2A(u,v)-2\alpha F(u,v)\,.$$
	For solutions, $F=\sigma H$ and \eqref{Aidentity} imply $P=0$. Thus the fibre identity and the PDE Poho\v zaev identity agree.
\end{remark}


\begin{remark}[The role of the estimate beyond homogeneous powers]
	The role of the bilinear mass Gagliardo-Nirenberg inequality established in Theorem \ref{introinterpolation} remains central when exact scaling of solutions
	is unavailable, unlike the case of the homogeneous nonlinearities considered in \eqref{system}. For example, for $f,g$ continuous, define
	$\mathfrak F(t)=\int_0^t f(\tau)\dd\tau$ and
	$\mathfrak G(t)=\int_0^t g(\tau)\dd\tau$, and suppose that, for some
	$c_0>0$ and every $t\in\R$,
	\[
	0\le\mathfrak F(t)\le\frac{c_0}{p}|t|^p\,,\qquad
	0\le\mathfrak G(t)\le\frac{c_0}{q}|t|^q\,.
	\]
	The corresponding energy functional is
	$J_{f,g}(u,v)=A(u,v)-\int[\mathfrak G(u)+\mathfrak F(v)]$.
	On $\mathcal M_K^a$, Theorem~\ref{introinterpolation} immediately gives
	\begin{equation*}
		J_{f,g}(u,v)\ge A(u,v)-c_0C_Ka^\beta A(u,v)^\alpha.
	\end{equation*}
	This bound supplies the same lower control for $\alpha<1$, positivity at sufficiently small
	mass for $\alpha=1$, and a barrier at small positive $A$ for
	$\alpha>1$. Negative energy at large scales requires additional
	lower growth assumptions on $f$ and $g$.
\end{remark}

\section{On the attainment of the best constant in the Gagliardo-Nirenberg inequality: the proof of Theorem~\texorpdfstring{\ref{introattainment}}{(introattainment)}}\label{Section_attainment}

\subsection{Proof of Theorem~\texorpdfstring{\ref{introattainment}}{(introattainment)}}
We split the proof of the theorem in several steps.
\begin{proof}[Proof of Theorem ~\ref{introattainment}, part 1] We show here that every optimiser of $C_K^*$ has at least one normalised component norm in \eqref{MK} equal to $K$.
	
	Consider the reciprocal-amplitude path
	$R_\tau(u,v)=(e^\tau u,e^{-\tau}v)$, which clearly preserves the bilinear forms $A$ and $B$.
	Normalisation commutes with $R_\tau$, so a point $(u,v)\in\cM_K^a$ satisfying strictly all four conditions in \eqref{MK} stays in $\cM_K^a$ for small $|\tau|$.
	Since
	$$\frac{\dd^2}{\dd\tau^2}F(R_\tau(u,v))=q\,e^{q\tau}\int|u|^q+p\,e^{-p\tau}\int|v|^p>0\,,$$
	the interpolation ratio \eqref{normalised} is strictly convex along the path $\tau\mapsto R_\tau(u,v)$, and the
	energy $J$ is strictly concave since 
	$$\frac{\dd^2}{\dd\tau^2}J(R_\tau(u,v))=-\frac{\dd^2}{\dd\tau^2}F(R_\tau(u,v))<0\,.$$
	This shows that neither can have the indicated extremum at
	an interior point of the path. On $\Sigma_a$, $B(R_\tau(u,v))=a$ for every real $\tau$,
	so no component restriction prevents this variation at any point.
\end{proof}


\begin{proof}[Proof of Theorem~\ref{introattainment}, part 2] Let us now prove \eqref{CK_bounds}.
	 
	The upper bound follows from \eqref{introGN}. Choose a smooth
	nonnegative function $w$ with $\int w^2=\int|\nabla w|^2=1$, and define
	$L_r=\max\{\|w\|_r,\|\nabla w\|_r\}$ and $L_s=\max\{\|w\|_s,\|\nabla w\|_s\}$.
	The pair $(tw,t^{-1}w)$ is already normalised, since $B(tw,t^{-1}w)=1$, and coincides with $(\widehat{tw},\widehat{t^{-1}w})$. 
	Clearly $\|tw\|_r\,,\,\|\nabla(tw)\|_r\leq tL_r$ and similarly $\|t^{-1}w\|_s\,,\,\|\nabla(t^{-1}w)\|_s\leq tL_s$. In order for $(tw,t^{-1}w)$ to belong to $\cM_K$, by \eqref{MK} one needs to impose the bounds $\frac{L_s}K\leq t\leq\frac K{L_r}$, an interval which is nonempty provided $K^2\ge L_rL_s$. Taking the extremes values, gives both
	$$C_K^*\ge\frac{\|w\|_q^q}{qL_r^q}K^q\quad\mbox{ and }\quad C_K^*\ge\frac{\|w\|_p^p}{pL_s^p}K^p\,.$$
	The assertion about $a_K$ follows from \eqref{threshold}.
\end{proof}

For the rest of the section, assume the individual bounds 
\begin{equation}\label{22*}
	2<p,q<2^*
\end{equation} and hence use $r=s=2$.
The Cauchy-Schwarz inequality makes $K=1$ the smallest possible component bound since $B(\hu,\hv)=1$ by \eqref{normalised}.
Write $\mathcal M_1$ for \eqref{MK} with $K=1$, and define
\[
\mathcal W=\{w\in H^1(\R^N):\|w\|_2=\|\nabla w\|_2=1\}\,.
\]

\begin{proof}[Proof of Theorem~\ref{introattainment}, part 3] Here we discuss the attainment on the minimal Hilbert class of the best constant $C_K^*$.

	We first reduce the problem to a scalar one. For $(u,v)\in\mathcal M_1$, the identity
	$\int\widehat u\widehat v=1$ and the bounds
	$\|\widehat u\|_2,\|\widehat v\|_2\le1$ force
	$\widehat u=\widehat v=w$ for some $w\in H^1(\R^N)$ with $\|w\|_2=1$.
	The gradient normalisation gives $\|\nabla w\|_2=1$, hence $w\in\cW$.
	Conversely every $(w,w)$ with $w\in\mathcal W$ is in $\mathcal M_1$.
	The second identity in \eqref{normalised} gives \eqref{C1formula} with a supremum.
	
	We next prove compactness and attainment. 
	Let $(w_n)_n\subset\cW$ be a maximising sequence for \eqref{C1formula}, namely such that
	\begin{equation}\label{C1seq}
		\frac1q\int|w_n|^q+\frac1p\int|w_n|^p>C_1^*-\frac1n\,.
	\end{equation}
	We claim that we can suppose that $(w_n)_n$ is nonnegative, radial, and radially decreasing. Indeed, for all $n\in\N$ let $w_n^*$ be the nonnegative
	symmetric decreasing rearrangement of $w_n$. Since rearrangements preserve Lebesgue norms, $w_n^*$ fulfills \eqref{C1seq} as well, however it may happen that $w_n^*\notin\cW$ since in general $\|\nabla w_n^*\|_2\leq1$. Defining $\tw_n:=\tmu w_n^*(\tnu\,\cdot)$, it is easy to see that 
	$$\tmu=\|\nabla w_n^*\|_2^{-\frac N2}\quad\ \mbox{and}\ \quad \tnu=\|\nabla w_n^*\|_2^{-1}$$
	is the unique choice of the parameters such that $\tw_n\in\cW$. It remains to verify that also $(\tw_n)_n$ is a maximising sequence in the sense of \eqref{C1seq}. In fact, one computes
	\begin{equation*}
		\frac1q\int|\tw_n|^p+\frac1p\int|\tw_n|^p=\|\nabla w_n^*\|_2^{\frac N2(2-p)}\int|w_n|^p+\|\nabla w_n^*\|_2^{\frac N2(2-q)}\int|w_n|^q\,.
	\end{equation*}
	Recalling \eqref{22*}, the exponents are negative and therefore
	\begin{equation*}
		\frac1q\int|\tw_n|^q+\frac1p\int|\tw_n|^p
		\geq\frac1q\int|w_n|^q+\frac1p\int|w_n|^p>C_1^*-\frac1n\,.
	\end{equation*}
	Since $(\tw_n)_n$ is bounded in $H^1(\R^N)$, after extraction, has a weak limit $w\in H^1(\R^N)$. Moreover, since $\tw_n$ is radially decreasing by construction, the estimate $\tw_n(\rho)\le\omega_N^{-1/2}\rho^{-N/2}$ holds for any $\rho>0$, since
	$$1=\|\tw_n\|_2^2\geq\int_{B_\rho(0)}|\tw_n|^2\geq\tw_n(\rho)^2\omega_N\rho^N\,.$$
	This, together with local Rellich compactness, implies strong convergence in
	both $L^p(\R^N)$ and $L^q(\R^N)$. In particular $w\ne0$ and
	\begin{equation}\label{C1w}
		F(w,w)=C_1^*>0\,.
	\end{equation}
	Let us now show that $w\in\cW$. 
	Suppose by contradiction that $w\notin\cW$, that is either $\|w\|_2<1$ or $\|\nabla w\|_2<1$ holds. Similarly as before, define $\phi:=\mu w(\nu\,\cdot)$. There is a unique choice of $\mu$, $\nu$ such that $\|\phi\|_2=\|\nabla\phi\|_2=1$, namely
	\begin{equation*}
		\mu=\|w\|_2^{-1+\frac N2}\|\nabla w\|_2^{-\frac N2}\quad\ \mbox{and}\ \quad\nu=\|\nabla w\|_2^{-1}\|w\|_2\,.
	\end{equation*}
	With such a choice, for $t\in\{p,q\}$, we compute
	\begin{equation}\label{at_bt}
		\int|\phi|^t=\mu^t\nu^{-N}\int|w|^t=\|\nabla w\|_2^{-a_t}\|w\|_2^{-b_t}\,,\quad\mbox{where}\quad a_t:=\frac{N(t-2)}2\,,\ \ b_t:=t-a_t\,.
	\end{equation}
	 Since either $\|w\|_2<1$ or $\|\nabla w\|_2<1$, both nonlinear terms strictly increase,
	contradicting \eqref{C1w}. Thus $\|w\|_2=\|\nabla w\|_2=1$, the supremum in \eqref{C1formula} is attained,
	and the chosen sequence converges strongly in $H^1(\R^N)$. 
	
	Let us next show that $w$ solves \eqref{mixedoptimizer} for suitable $\eta,\nu>0$.
	For any nonzero $h\in H^1(\R^N)$ define
	$$\mathcal Q(h)=\sum_{t\in\{p,q\}}\frac{\|h\|_t^t}{t\|\nabla h\|_2^{a_t}\|h\|_2^{b_t}}\,.$$
	The sum is interpreted with both terms counted, also when $p=q$. 
	The above computations show that $\mathcal Q(h)\le C_1^*$ with equality at
	$w$. For $t\in\{p,q\}$, differentiating this $C^1$ quotient at $w$ gives
	\begin{equation*}
		\frac\dd{\dd\varepsilon}_{\left|_{\varepsilon=0}\right.}\,\frac{\int|w+\varepsilon\varphi|^t}{t\left(\int|\nabla(w+\varepsilon\varphi|^2\right)^{\frac{a_t}2}\left(\int|w+\varepsilon\varphi|^2\right)^{\frac{b_t}2}}=\int|w|^{t-2}w\varphi-\frac1t\int|w|^t\left(a_t\int\nabla w\nabla\varphi+b_t\int w\varphi\right),
	\end{equation*}
	that is \eqref{mixedoptimizer} with
	\[
	\eta=\sum_{t\in\{p,q\}}\frac{a_t}t\|w\|_t^t>0\,,
	\quad\ \mbox{and}\ \quad
	\nu=\sum_{t\in\{p,q\}}\frac{b_t}t\|w\|_t^t>0\,,
	\]
	by summing the two expressions for $t\in\{p,q\}$. 
	Elliptic regularity and the strong maximum principle make $w$ classical
	and strictly positive. The attainment conclusions for $J$ follow by
	choosing the mass and then the optimal fibre in Theorem~\ref{introenergytrichotomy}. 
	In fact, choose a sequence on the mass slice whose interpolation ratios tend to $C_K^*$. On each of its fibres, $t^2A$ ranges over all positive numbers, and its ratio $Q=F/(A^\alpha a^\beta)$ is constant. Minimising on this fibre therefore gives the same expression with $Q$ in place of $C_K^*$. Taking the limit proves equality.
	
	Let us next discuss the connection between the optimisation problem \eqref{C1formula} and the equation \eqref{mixedoptimizer}. 
	Note that every pair in $\mathcal M_1$ has proportional components. If $u=cv>0$
	solves \eqref{system}, the two equations imply
	$v^{p-1}=c^qv^{q-1}$. For $p\ne q$ this would make $v$ a positive
	constant, contrary to integrability. When $p=q=m\in(2,2^*)$,
	$C_1^*=\frac2mG_{N,m}$, where $G_{N,m}$ is the sharp scalar constant in
	\eqref{scalarGN}. At $m=2+\frac4N$, its value
	$G_{N,m}=\frac m2\|Q_m\|_2^{-4/N}$ \cite{Weinstein} gives
	$a_1=(C_1^*)^{-N/2}=\|Q_m\|_2^2$.
\end{proof}
The idea behind the last step of the proof is to replace the vector problem with a scalar once, incurring an error that tends to zero as \(K\searrow1\).
\begin{proof}[Proof of Theorem~\ref{introattainment}, part 4]
	
	Write $u=\widehat u$, $v=\widehat v$ for a normalised admissible pair,
	so $A(u,v)=B(u,v)=1$ and each of its four component norms in \eqref{MK} is at most $K$ (recall that now $r=s=2$). Expanding the squares gives
	$$\|u-v\|_{H^1}^2=\|u\|_2^2+\|v\|_2^2-2B(u,v)+\|\nabla u\|_2^2+\|\nabla v\|_2^2-2A(u,v)\le4(K^2-1)\,.$$
	Similarly, for $w=\frac12(u+v)$ one has 
	$$1\leq\|w\|_2^2,\|\nabla w\|_2^2\le\frac12(K^2+1)=:s_K\,,$$ where the lower bound follows from the Cauchy-Schwarz inequality.
	Hence, $H^1$ norms of $u,\,v,\,w$ are uniformly bounded for $K\searrow1$.
	The mean-value inequality $\big||a|^t-|b|^t\big|\leq t\big(|a|^{t-1}+|b|^{t-1}\big)|a-b|$ for $t\in\{p,q\}$ and the Sobolev embeddings therefore give
	\begin{equation*}
		\begin{split}
			|F(u,v)-F(w,w)|&\le\int\left(|u|^{q-1}+|w|^{q-1}\right)|u-w|+\int\left(|v|^{p-1}+|w|^{p-1}\right)|v-w|\\
			&\leq\left(\|u\|_q^{q-1}+\|w\|_q^{q-1}\right)\|u-w\|_q+\left(\|v\|_p^{p-1}+\|w\|_p^{p-1}\right)\|v-w\|_p\\
			&\leq C\left(\|u-w\|_{H^1}+\|v-w\|_{H^1}\right)=C\|u-v\|_{H^1}\leq C\sqrt{K^2-1}\,.
		\end{split}
	\end{equation*}
	We normalise $w$ by scalar amplitude and dilation to obtain
	$\widetilde w\in\mathcal W$. With $a_t,b_t>0$ as in \eqref{at_bt}, 
	$$F(w,w)=\sum_{t\in\{p,q\}}\frac{\|\nabla w\|_2^{a_t}\|w\|_2^{b_t}}t\int|\widetilde w|^t\le s_K^{\frac{\max\{p,q\}}2}C_1^*\,.$$
	Since $s_K^{\max\{p,q\}/2}-1\le C(K^2-1)$ for $K\in[1,2]$, taking
	the supremum proves the upper bound in \eqref{nearminimalestimate}.
	The lower bound follows from the inclusion $\mathcal M_1\subset\mathcal M_K$.
\end{proof}

\subsection{Further results}
The scalar reduction described above extends to every $K\ge1$ if one restricts the
normalised components to be proportional. This is a subclass of
$\cM_K$, hence the following attainment result does not hold for the whole class.

\begin{proposition}[Attainment on the proportional subclass]\label{proportionalattainment}
	Under \eqref{22*}, for $K\ge1$ define
	$$C_{K,\mathrm{prop}}^*:=\sup_{\substack{w\,\in\,\mathcal W\\K^{-1}\le\, c\,\le K}}\left(\frac{c^q}q\int|w|^q+\frac{c^{-p}}p\int|w|^p\right).$$
	This constant is attained, satisfies $C_{K,\mathrm{prop}}^*\le C_K^*$,
	and is the larger of the two attained maxima
	\begin{equation*}
	 \max_{w\in\mathcal W}
	 \left(\frac{K^q}{q}\int|w|^q+\frac{K^{-p}}p\int|w|^p\right),
	 \qquad
	 \max_{w\in\mathcal W}
	 \left(\frac{K^{-q}}q\int|w|^q+\frac{K^p}p\int|w|^p\right).
	\end{equation*}
	A maximiser can be chosen positive, radial and nonincreasing. For
	the selected endpoint $c\in\{K,K^{-1}\}$ it satisfies
	\begin{equation}\label{eq_displayed}
		-\eta\Delta w+\nu w=c^q w^{q-1}+c^{-p}w^{p-1},	
	\end{equation}
	for some $\eta,\nu>0$. In particular, if $p=q=m$, then
	\begin{equation*}
	 C_{K,\mathrm{prop}}^*=
	       \frac{G_{N,m}}m\left(K^m+K^{-m}\right),
	\end{equation*}
	where $G_{N,m}$ is the sharp scalar constant in \eqref{scalarGN}.
\end{proposition}

\begin{proof}
	For $(u,v)=(cw,c^{-1}w)$ with $w\in\mathcal W$, one has $A=B=1$.
	The four component bounds in \eqref{MK} are equivalent to $K^{-1}\le c\le K$.
	Conversely every positively proportional normalised pair has this
	representation. 
	For each fixed $w\in\cW$, the map
	$$c\longmapsto\Phi(c):=\frac{c^q}{q}\int|w|^q+\frac{c^{-p}}p\int|w|^p$$
	is strictly convex in $\log c$. Indeed, defining $\psi(s):=\Phi(\log c)$, one has
	$$\psi(s)''=qe^{qs}\int|w|^q+p e^{-ps}\int|w|^p>0\,.$$
	Hence the maximum of $\Phi$ over $[K^{-1},K]$ is attained at one of the endpoints. Consequently, the supremum defining $C_{K,\mathrm{prop}}^*$ can be restricted to these two values of $c$.
	For either endpoint both coefficients in $\Phi(c)$ are
	strictly positive. The rearrangement, radial compactness and
	renormalisation proof of Proposition~\ref{introattainment} applies
	with these coefficients retained. More explicitly, a radial
	maximising sequence converges strongly in $L^p\cap L^q$ to a
	nonzero weak $H^1$ limit. If its squared mass or gradient norm is
	less than one, scalar renormalisation multiplies each power integral
	by $\|\nabla w\|_2^{-a_t}\|w\|_2^{-b_t}>1$, contradicting maximality. Both norms
	therefore equal one. Differentiating the weighted scalar quotient
	gives \eqref{eq_displayed}, with positive coefficients $\eta,\nu$; positivity is then a consequence of the maximum principle. 
\end{proof}

\begin{remark}[A loss of compactness in the class $\cM_K$]
	For any $K>1$, choose nonnegative smooth compactly supported functions
	$w,z\in\mathcal W$, set $s=\sqrt{K^2-1}$ and let
	$z_n=z(\,\cdot-y_n)$, with $|y_n|\to+\infty$ and disjoint supports
	from $w$. Then
	\[
	 u_n=w+sz_n\,,\qquad v_n=w\,,\qquad A(u_n,v_n)=B(u_n,v_n)=1\,,
	\]
	while both norms of $u_n$ equal $K$ and both norms of $v_n$ equal
	$1$. No common translation makes this sequence strongly compact in
	$H^1\times H^1$: indeed, keeping the nonzero $v$ component tight leaves the
	$z_n$ component escaping. 
	This example identifies a defect that an attainment proof for all of $\mathcal M_K$
	must rule out.
\end{remark}


\section{Local continuation for nonhomogeneous systems}\label{Section_branches_powers}
In this section, by means of an implicit function argument, we
obtain positive normalised solutions for the perturbed system \eqref{perturbedsystem}, branching from solutions of \eqref{system} which do not belong to
the mass-critical curve \eqref{critical_curve}. To this aim, we first prove that the global nondegeneracy obtained in Theorem \ref{thm:main} for the unit-frequency system \eqref{system1} transfer to our system \eqref{system} by means of the scaling \eqref{introfrequencyscaling}.


\begin{lemma}[The fixed-mass linearisation]\label{linearisation}
	Let $(p,q)$ in the range \eqref{range} and $a=B(u_\lambda,v_\lambda)=\int u_\lambda v_\lambda>0$.
	The bounded operator $\cA(H^1_{\rm rad})^2\times\R\to(H^{-1}_{\rm rad})^2\times\R$ defined by
	\begin{equation}\label{augmentedoperator}
		\cA_\lambda(h,k,\tau)=
		\begin{pmatrix}
			(-\Delta+\lambda)h-(p-1)v_\lambda^{p-2}k+\tau u_\lambda\\
			(-\Delta+\lambda)k-(q-1)u_\lambda^{q-2}h+\tau v_\lambda\\
			\int(v_\lambda h+u_\lambda k)
		\end{pmatrix}
	\end{equation}
	is an isomorphism if and only if $\gamma_N(p,q)\ne0$. If $\gamma_N(p,q)=0$, its kernel is precisely
	${\rm span}\{(\partial_\lambda u_\lambda,\partial_\lambda v_\lambda,1)\}$.
\end{lemma}
\begin{proof}
	\textit{Step 1: Invertibility of the radial linearisation without
		the mass constraint.}
	
	Define
	\begin{equation}\label{L_lambda}
		\mathcal L_\lambda
		\begin{pmatrix}h\\k\end{pmatrix}
		=
		\begin{pmatrix}
			(-\Delta+\lambda)h-(p-1)v_\lambda^{p-2}k\\
			(-\Delta+\lambda)k-(q-1)u_\lambda^{q-2}h
		\end{pmatrix}\,.
	\end{equation}
	%
	Relying on the nondegeneracy result in Theorem \ref{thm:main}, we first prove that $\mathcal L_\lambda:(H^1_{\mathrm{rad}}(\R^N))^2\to(H^{-1}_{\mathrm{rad}}(\R^N))^2$ is an isomorphism. Indeed, $\ker\bigl(\mathcal L_1\bigr)=\{0\}$, and the scaling \eqref{introfrequencyscaling} conjugates the linearised equations at frequency $\lambda$ to those at frequency $1$, up to multiplication by nonzero constants. Since it preserves radiality, one infers $\ker\bigl(\mathcal L_\lambda\bigr)=\{0\}$ for every $\lambda>0$. Moreover, the coefficients $u_\lambda^{q-2}$ and
	$v_\lambda^{p-2}$ in \eqref{L_lambda} are bounded and tend to zero at infinity by Lemma \ref{profileregularity}. 
	Combining this decay with the compact
	embedding $H^1(B_R(0))\hookrightarrow L^2(B_R(0))$ shows that
	the corresponding multiplication operators from $H^1$ to
	$H^{-1}$ are compact. Thus $\mathcal L_\lambda$ is a
	compact perturbation of the isomorphism $(-\Delta+\lambda)I:(H^1_{\mathrm{rad}}(\R^N))^2\to(H^{-1}_{\mathrm{rad}}(\R^N))^2$.
	It is therefore Fredholm of index zero. The Fredholm alternative now implies that $\mathcal L_\lambda$
	is an isomorphism.
	
	\medskip
	\noindent
	\textit{Step 2: The augmented linearisation at fixed bilinear mass.}
	
	For a fixed $\lambda>0$, set
	\[
	\ell_\lambda(h,k)
	=\int(v_\lambda h+u_\lambda k)
	\]
	and define the augmented operator
	\[
	\cA_\lambda(h,k,\tau)
	=
	\begin{pmatrix}
		\mathcal L_\lambda(h,k)
		+\tau(u_\lambda,v_\lambda)\\
		\ell_\lambda(h,k)
	\end{pmatrix}.
	\]
	We aim at proving that $\cA_\lambda:(H^1_{\mathrm{rad}}(\R^N))^2\times\R
	\to(H^{-1}_{\mathrm{rad}}(\R^N))^2\times\R$ is an isomorphism if and only if
	$\gamma_N(p,q)\ne0$, and identify its kernel when
	$\gamma_N(p,q)=0$. Recall that the null sets of $\gamma_N(p,q)$ and $\gamma$ defined in \eqref{parameters} coincide. The map $\lambda\mapsto(u_\lambda,v_\lambda)$ is differentiable, and differentiating the two equations in \eqref{system} with respect to $\lambda$ gives
	\[
	\mathcal L_\lambda
	\begin{pmatrix}
		\partial_\lambda u_\lambda\\
		\partial_\lambda v_\lambda
	\end{pmatrix}
	=
	-\begin{pmatrix}u_\lambda\\v_\lambda\end{pmatrix}.
	\]
	At a frequency for which $B(u_\lambda,v_\lambda)=a$, differentiation
	of the scaling identity for $B$ also gives
	\[
	\ell_\lambda(\partial_\lambda u_\lambda,
	\partial_\lambda v_\lambda)
	=\frac\dd{\dd\lambda}B(u_\lambda,v_\lambda)
	=\frac{\gamma a}\lambda\,.
	\]
	Consequently,
	\[
	\cA_\lambda
	(\partial_\lambda u_\lambda,
	\partial_\lambda v_\lambda,1)
	=
	\left(0,0,\frac{\gamma a}{\lambda}\right).
	\]
	
	To establish the full isomorphism assertion, let
	$(f,c)\in(H^{-1}_{\mathrm{rad}}(\R^N))^2\times\R$ be arbitrary with
	$f=(f_1,f_2)$, and set $(z_1,z_2)=\mathcal L_\lambda^{-1}f$, which is well defined since $\cL$ is invertible. The first two components of the equation
	\[
	\cA_\lambda(h,k,\tau)=(f,c)
	\]
	then imply
	\[
	(h,k)=(z_1,z_2)
	+\tau(\partial_\lambda u_\lambda,
	\partial_\lambda v_\lambda)\,.
	\]
	The last component is therefore equivalent to
	\[
	\ell_\lambda(z_1,z_2)
	+\tau\frac{\gamma a}{\lambda}=c\,.
	\]
	If $\gamma\ne0$, this determines a unique $\tau$, and hence a
	unique $(h,k,\tau)$. The boundedness of the inverses involved
	also gives a bounded inverse for $\cA_\lambda$. If instead $\gamma=0$, the same identities show that
	$\ker\cA_\lambda
	=
	\operatorname{span}
	\left\{
	(\partial_\lambda u_\lambda,
	\partial_\lambda v_\lambda,1)
	\right\}.$
	
\end{proof}

\begin{proof}[Proof of Theorem~\ref{intropowercontinuation}]

	Define $\Phi:(2,2^*)^2\times\R^2\times(H^1_{\rm rad}(\R^N))^2\times\R\to(H^{-1}_{\rm rad}(\R^N))^2\times\R$ by
	\[
	\begin{aligned}
		\Phi(p,q,\varepsilon_1,\varepsilon_2,u,v,\lambda)
		=\begin{pmatrix}
			(-\Delta+\lambda)u
			-v_+^{p-1}-\varepsilon_1v_+^{\ell_1-1}\\
			(-\Delta+\lambda)v
			-u_+^{q-1}-\varepsilon_2u_+^{\ell_2-1}\\
			\displaystyle\int uv-a
		\end{pmatrix}.
	\end{aligned}
	\]
	By construction, $\Phi(p_0,q_0,0,0,u_0,v_0,\lambda_0)=0$. Moreover, arguing as in the proof of Lemma \ref{lem:cc-local-continuation}, by \eqref{22*} the map $\Phi$ is of class $C^1$ from an open neighbourhood of $(p_0,q_0,0,0,u_0,v_0,\lambda_0)$ into $(H^{-1}_{\mathrm{rad}}(\R^N))^2\times\R$. Indeed, on any compact exponent interval inside $(2,2^*)$, the Nemytskii map $(b,z)\mapsto z_+^{b-1}$ is $C^1$ from
	$\R\times H^1$ to $H^{-1}$.
	The derivative of $\Phi$ with respect to $(u,v,\lambda)$ at the
	base point is
	\[
	D_{(u,v,\lambda)}\Phi
	\begin{pmatrix}h\\k\\\tau\end{pmatrix}
	=
	\begin{pmatrix}
		(-\Delta+\lambda_0)h
		-(p_0-1)v_0^{p_0-2}k+\tau u_0\\
		(-\Delta+\lambda_0)k
		-(q_0-1)u_0^{q_0-2}h+\tau v_0\\
		\displaystyle\int_{\R^N}(v_0h+u_0k)
	\end{pmatrix}.
	\]
	This is precisely the augmented operator $\cA_\lambda$ defined in Lemma \ref{linearisation}, hence an isomorphism
	since $\gamma_N(p_0,q_0)\ne0$. The Banach-space implicit function theorem therefore yields
	neighbourhoods of $(p_0,q_0,0,0)$ and
	$(u_0,v_0,\lambda_0)$, and a unique $C^1$ map $(p,q,\varepsilon_1,\varepsilon_2)\to (u,v,\lambda)$ 
	whose values solve $\Phi=0$. In particular, this gives the
	claimed local branch and its local uniqueness in $(H^1_{\mathrm{rad}})^2\times\R$.
	Shrinking the parameter neighbourhood if necessary, continuity
	ensures that $\lambda>0$ and that neither component vanishes identically.
	
%
	For $\varepsilon_1,\varepsilon_2\ge0$, to show that $u,v\geq0$ it is sufficient to test the first equation with $u_-$, obtaining
	$$-\|\nabla u_-\|_2^2-\lambda\|u_-\|_2^2 =\int\bigl(v_+^{p-1}+\varepsilon_1v_+^{\ell_1-1}\bigr)u_-\ge0\,,$$
	since $\lambda>0$. This shows $u_-=0$ and therefore $u\ge0$. The same argument applied to the second equation yields $v\ge0$. 
	Since both components are nontrivial, the strong maximum principle gives the strict signs.
	
	Although the implicit function theorem was applied in the radial space, the equations in fact hold in the usual weak sense against arbitrary test functions. Indeed, the residuals are radial. If $\varphi\in H^1(\R^N)$ and $\varphi_{\rm rad}$ denotes its spherical average, then $\langle T,\varphi\rangle=\langle T,\varphi_{\rm rad}\rangle$ for every radial distribution $T$. Thus testing with radial functions is sufficient to recover the full weak formulation.

%
%
	Since all exponents in the nonlinearities are strictly
	subcritical, the standard elliptic bootstrap for the coupled
	system yields bounded classical solutions. In particular,
	$u,v$ are uniformly continuous. Their $L^2$ integrability then
	implies that $u(x)+v(x)\to0$ as $|x|\to+\infty$ and in turn the following tail estimate:
	for every sufficiently small $\eta>0$, there is	$R>0$ such that
	$$v^{p-1}+\varepsilon_1v^{\ell_1-1}\le\eta v\,,
	\quad\ 
	u^{q-1}+\varepsilon_2u^{\ell_2-1}\le\eta u\,,
	\quad\ \text{for }|x|\ge R\,.$$
	This permits to follow the argument used to deduce exponential decay in the proof of Lemma \ref{profileregularity}.
\end{proof}

\subsection{Critical mass response to additional powers}\label{Section_massresponse}
At criticality the mass constraint destroys invertibility, but the
fixed-frequency PDE linearisation remains invertible in radial spaces.
This allows for an explicit description of the initial mass change.

\begin{proof}[Proof of Proposition~\ref{introcriticalresponseNgeq3}]
	The existence of the unique local branch of solutions follows from the fact that $\cL_\lambda:(H^1_{\rm rad}(\R^N))^2\to(H^{-1}_{\rm rad(\R^N)})^2$ defined in \eqref{L_lambda} is an isomorphism, as in the proof of Lemma \ref{linearisation}. 
	In the bilinear mass critical case, the scaling law gives $B(u_\lambda,v_\lambda)=\lambda^\gamma a_*=a_*$, so the unperturbed mass is independent of $\lambda$. To compute the derivatives with respect to the perturbation
	parameters, introduce the operator
	\[
	\cH_\lambda(h,k)=
	\begin{pmatrix}
		(-\Delta+\lambda)k-(q-1)u_\lambda^{q-2}h\\
		(-\Delta+\lambda)h-(p-1)v_\lambda^{p-2}k
	\end{pmatrix},
	\]
	which is symmetric with respect to the natural $H^{-1}$--$H^1$ duality pairing. Differentiating the unperturbed equations in \eqref{system} with respect to $\lambda$ gives
	\[
	\cH_\lambda
	\begin{pmatrix}
		\partial_\lambda u_\lambda\\
		\partial_\lambda v_\lambda
	\end{pmatrix}
	=-\begin{pmatrix}v_\lambda\\u_\lambda\end{pmatrix}.
	\]
	For $i=1,2$, write
	\[
	(h_i,k_i)=
	\left.\partial_{\varepsilon_i}
	(u_{\lambda,\varepsilon},v_{\lambda,\varepsilon})
	\right|_{\varepsilon=0}.
	\]
	Differentiating \eqref{perturbedsystem} with respect to
	$\varepsilon_1$ and $\varepsilon_2$, respectively, yields
	\[
	\cH_\lambda(h_1,k_1)
	=\begin{pmatrix}0\\v_\lambda^{\ell_1-1}\end{pmatrix},
	\qquad
	\cH_\lambda(h_2,k_2)
	=\begin{pmatrix}u_\lambda^{\ell_2-1}\\0\end{pmatrix}.
	\]
	Hence, by the symmetry of $\cH_\lambda$ we obtain
	\begin{equation*}
		\left.\partial_{\varepsilon_1}
		B(u_{\lambda,\varepsilon},v_{\lambda,\varepsilon})
		\right|_{\varepsilon=0}=\int(v_\lambda h_1+u_\lambda k_1)=-\int
		v_\lambda^{\ell_1-1}\partial_\lambda v_\lambda=-\frac1{\ell_1}\frac\dd{\dd\lambda}
		\int v_\lambda^{\ell_1}\,,
	\end{equation*}
	and similarly
	$$\left.\partial_{\varepsilon_2}B(u_{\lambda,\varepsilon},v_{\lambda,\varepsilon})\right|_{\varepsilon=0}=-\frac1{\ell_2}\frac\dd{\dd\lambda}\int u_\lambda^{\ell_2}\,.$$
	Finally, the scaling \eqref{introfrequencyscaling} implies
	\[
	\int v_\lambda^\ell
	=\lambda^{\ell q/D-N/2}\int V^\ell\,,
	\qquad
	\int u_\lambda^\ell
	=\lambda^{\ell p/D-N/2}\int U^\ell\,,
	\]
	therefore,
	\[
	\frac\dd{\dd\lambda}\int v_\lambda^\ell
	=\frac{e_\ell^{(v)}}{\lambda}
	\int v_\lambda^\ell\,,
	\qquad
	\frac\dd{\dd\lambda}\int u_\lambda^\ell
	=\frac{e_\ell^{(u)}}{\lambda}
	\int u_\lambda^\ell\,.
	\]
	These identities give the two coefficients in
	\eqref{criticalexpansion}. The continuity of the derivatives
	of the local branch with respect to $\varepsilon$ yields the
	remainder $o(|\varepsilon|)$, locally uniformly for $\lambda$
	near $\lambda_0$. This concludes the proof.
\end{proof}

\section{Exponential Gagliardo--Nirenberg inequalities in the plane}\label{secTM}
Throughout this section $N=2$. 
Since now \eqref{alphabeta} gives $\beta=1$ and $\alpha=\sigma^{-1}-1$, in the planar case the Gagliardo-Nirenberg inequality by Theorem \ref{introinterpolation} reads as $F\le C_K A^{\sigma^{-1}-1}B$ on the admissible class $\cM_K$. However, the allowance of exponential growth by the Trudinger-Moser inequality raises a further question: which limiting estimate
keeps the dependence on mass that makes Gagliardo--Nirenberg inequalities
useful in a normalised problem? We use the mass-weighted
Adachi--Tanaka formulation and then subtract the quadratic term to
obtain an estimate for the nonlinear potential.

\begin{remark}[Growth beyond the scalar Trudinger--Moser threshold]
\label{supercriticalLorentzscope}
For $\rho>2$ and $\rho'=\rho/(\rho-1)\in(1,2)$, the complementary
growth scales at infinity
\[
 e^{\vartheta_1|u|^\rho}\quad\mbox{and}\quad e^{\vartheta_2|v|^{\rho'}}
\]
for $\vartheta_1,\,\vartheta_2>0$ fit the Lorentz--Sobolev framework of \cite{CassaniTarsi2015}, with
the additional hypotheses of its existence theorems. 
The corresponding product setting is
\[
 W^1L^{2,\rho'}(\R^2)\times W^1L^{2,\rho}(\R^2),
\]
where $W^1L^{2,s}$ requires the function and its first derivatives
to belong to the Lorentz space $L^{2,s}$. Lorentz H\"older duality
gives
\[
 |A(u,v)|+|B(u,v)|\le C\bigl(
 \|\nabla u\|_{L^{2,\rho'}}\|\nabla v\|_{L^{2,\rho}}
 +\|u\|_{L^{2,\rho'}}\|v\|_{L^{2,\rho}}\bigr),
\]
thus the crossed gradient and bilinear mass remain meaningful in
this setting, suggesting a natural extension of the normalised
problem.

In the problem \eqref{TMpuresystem} that we consider here, the exponential powers are
both two, for every choice of $\vartheta_1,\vartheta_2>0$. 
Therefore, the present bilinear exponential GN estimate and constrained
compactness proofs are established in the space $H^1\times H^1$. 
An intriguing challenge is to extend these results to the complementary growth scales above, which would require working in the corresponding Lorentz norms.
\end{remark}

\subsection{The scalar mass-weighted formulation}\label{secTMscalar}
For a nonzero
$w\in H^1(\R^2)$ write
$$G=\|\nabla w\|_2^2\quad \mbox{and} \quad M=\|w\|_2^2\,.$$
For $\vartheta>0$, the full-norm estimate implies
\begin{equation}\label{scalarExpGN}
 \int_{\R^2}(e^{\vartheta w^2}-1)
 \le\frac{C_{\rm R}\vartheta M}{4\pi-\vartheta G}\,,
 \quad\mbox{when}\quad\vartheta G<4\pi\,,
\end{equation}
where $C_{\rm R}$ is defined in \eqref{Rufconstant}. 
Indeed, set $s^2=\vartheta M/(4\pi-\vartheta G)$ and
$z(x)=\sqrt{\vartheta/(4\pi)}\,w(sx)$. Then
$\|z\|_{H^1}^2=1$ and the change of variables in
\eqref{Rufconstant} gives \eqref{scalarExpGN}. This is the rescaling
underlying the equivalence with the improved Adachi--Tanaka inequality
proved in \cite{CST}; the scalar result is recalled here as the starting
point of the bilinear formulation.

Equivalently, for $0<\theta<4\pi$,
\begin{equation}\label{scalarExpGNnormalised}
 \int_{\R^2}\left(e^{\theta w^2/G}-1\right)
 \le\frac{C_{\rm R}\theta}{4\pi-\theta}\frac MG\,.
\end{equation}
This form makes the analogy with the standard Gagliardo-Nirenberg inequality
$\|w\|_r^r\le C_rG^{r/2-1}M$, $2<r<\infty$, explicit.
Under the scaling $w(x)\mapsto cw(tx)$, both sides of
\eqref{scalarExpGNnormalised} scale as $t^{-2}$, independently of
$c>0$. The exponential replaces the finite power, while the ratio
$M/G$ retains the mass and the correct spatial scaling.

\subsection{A bilinear exponential estimate and its sharp threshold}\label{secTMbilinear}
For $K\ge1$ define
\begin{equation}\label{TMclass}
 \mathcal E_K=\left\{(u,v)\in H^1(\R^2)^2:
 \begin{array}{l}
 A=\int\nabla u\cdot\nabla v>0,\quad\ \ \,B=\int uv>0,\\
 \|\nabla u\|_2^2,\|\nabla v\|_2^2\le K^2A,\ \, \|u\|_2^2,\|v\|_2^2\le K^2B
 \end{array}\right\}\,.
\end{equation}
These classes are invariant under common amplitude changes and spatial
dilations $(u,v)\mapsto(cu(t\,\cdot),cv(t\,\cdot))$, with $c,t>0$. Unlike the class $\mathcal M_K$ used to prove the power-type bilinear-mass Gagliardo-Nirenberg inequality in Theorem \ref{introinterpolation}, $\cE_K$ is defined directly by quadratic bounds.

\begin{proof}[Proof of Theorem~\ref{introexponentialGN}: the bilinear estimate and its threshold]
	Set
	\[
	 \theta=\frac{\vartheta K^2A}{4\pi}<1\,,\qquad
	 s^2=\frac{\vartheta K^2B}{4\pi-\vartheta K^2A}\,,\qquad
	 z(x)=\sqrt{\frac\vartheta{4\pi}}\,u(sx)\,.
	\]
	The planar scaling laws give
	$\|\nabla z\|_2^2\le\theta$ and
	$\|z\|_2^2\le1-\theta$. Applying \eqref{Rufconstant} and undoing
	the change of variables yields
	$\int(e^{\vartheta u^2}-1)\le C_{\rm R}s^2$.
	The same argument for $v$ proves \eqref{introexponentialbound}.
	This is the scalar 
	rescaling behind \cite{CST}, applied with the component bounds in \eqref{TMclass}. To prove sharpness, for $n\ge2$ take the radial Moser functions
	\[
	 \phi_n(x)=\frac1{\sqrt{2\pi}}
	 \begin{cases}
	 \sqrt{\log n}\,,&|x|\le\frac1n\,,\\
	 \displaystyle\frac{\log(1/|x|)}{\sqrt{\log n}}\,,
	       &\frac1n<|x|<1\,,\\
	 0\,,&|x|\ge1\,.
	 \end{cases}
	\]
	Direct integration gives $\|\nabla\phi_n\|_2^2=1$ and
	$b_n:=\|\phi_n\|_2^2\to0$. On the central disk,
	\[
	 \int_{|x|\le\frac1n}(e^{4\pi\phi_n^2}-1)=\pi\left(1-\frac1{n^2}\right).
	\]
	Put $t_n^2=A_0b_n/B_0$ and
	$w_n(x)=\sqrt{A_0}\phi_n(t_nx)$. Then
	$\|\nabla w_n\|_2^2=A_0$ and $\|w_n\|_2^2=B_0$.
	The pair $(u_n,v_n)=(Kw_n,K^{-1}w_n)$ belongs to
	$\mathcal E_K$ and has exactly $A=A_0$, $B=B_0$.
	If $\vartheta K^2A_0\ge4\pi$, its first exponential integral satisfies
	\[
	 \int(e^{\vartheta u_n^2}-1)\ge t_n^{-2}\pi\left(1-\frac1{n^2}\right)\to+\infty\,.
	\]
	This proves failure of uniform boundedness at and above the threshold.
\end{proof}

Taking $\vartheta=\theta/A$ in \eqref{introexponentialbound} yields the normalised
exponential Gagliardo--Nirenberg form
\begin{equation}\label{bilinearExpGNnormalised}
 \int_{\R^2}\left[(e^{\theta u^2/A}-1)+(e^{\theta v^2/A}-1)\right]
 \le\frac{2C_{\rm R}\theta K^2}{4\pi-\theta K^2}\frac BA\,,
 \quad\mbox{for}\quad 0<\theta<\frac{4\pi}{K^2}\,.
\end{equation}
For common amplitude changes and spatial dilations, $A$ becomes
$c^2A$ and $B$ becomes $c^2t^{-2}B$, so both sides again scale as
$t^{-2}$. The appearance of $B/A$ is thus the bilinear counterpart
of $M/G$ in \eqref{scalarExpGNnormalised}.

The component bounds remain essential. For a nonzero smooth compactly
supported $w$, reciprocal amplitudes $(cw,c^{-1}w)$ keep both $A$ and
$B$ fixed, while $\int(e^{\vartheta c^2w^2}-1)$ diverges as $c\to+\infty$.
At the full-norm endpoint, the different hypothesis
$\vartheta K^2(A+B)\le4\pi$ gives a bound $2C_{\rm R}$ directly from
\eqref{Rufconstant}. Thus a critical full-norm bound is compatible
with the failure at $\vartheta K^2A=4\pi$ when positive mass is fixed
separately.

\subsection{The truncated exponential and the quartic mass scale}\label{secTMtruncated}
The nonlinear potential associated with $t(e^{\vartheta t^2}-1)$
contains the truncated exponential
$\Psi(s)=e^s-1-s$. Removing the quadratic term in $t$ changes the
right-hand side from a mass bound to a product of mass and gradient
pairing, provided a fixed gap from the concentration threshold is kept.

\begin{proof}[Proof of Theorem~\ref{introexponentialGN}: the truncated estimate]

For $0\le a\le1$ and $s\ge0$, one can write
\[
 \Psi(as)=e^s-1-s=\sum_{j=2}^{\infty}\frac{a^js^j}{j!}
 \le a^2(e^s-1)\,.
\]
Set $\eta=(1+\rho)/2$ and $\vartheta'=4\pi\eta/(K^2A)$.
Then $0<\vartheta/\vartheta'\le\rho/\eta<1$. Applying the preceding
inequality to both components and using \eqref{introexponentialbound} at
$\vartheta'$ gives
\begin{align*}
 \int\big[\Psi(\vartheta u^2)+\Psi(\vartheta v^2)\big]
 &\le\left(\frac\vartheta{\vartheta'}\right)^2
       \frac{2C_{\rm R}\vartheta'K^2B}{4\pi-\vartheta'K^2A}=\frac{C_{\rm R}}{8\pi^2\eta(1-\eta)}\vartheta^2K^4AB\,.
\end{align*}
Since $4\eta(1-\eta)=1-\rho^2$, this is
\eqref{introtruncatedbound}.
\end{proof}

Subtracting the linear term in $\Psi$ is necessary for this scaling: the omitted term is
$\vartheta\int(u^2+v^2)$, which stays fixed under the mass-preserving
dilation below, whereas $AB$ tends to zero. It does not remove the
concentration obstruction. In the Moser sequence used above,
$\vartheta\int(u_n^2+v_n^2)=\vartheta(K^2+K^{-2})B_0$ is fixed.
The truncated exponential therefore still has unbounded supremum
at $\vartheta K^2A_0=4\pi$; no finite endpoint constant extends
\eqref{introtruncatedbound} to this threshold.

\begin{proof}[Proof of Corollary~\ref{scalartruncatedrefinement}]
	For $w\ne0$, $(w,w)\in\mathcal E_1$, so dividing
	\eqref{introtruncatedbound} by two proves
	\eqref{introscalartruncated}. For the sharper estimate put
	$R_3(s)=e^s-1-s-s^2/2$. The exponential series gives
	$R_3(\varepsilon s)\le\varepsilon^3(e^s-1)$ for
	$0\le\varepsilon\le1$, $s\ge0$. The function
	$f(x)=G^{-1/2}w(\sqrt{M/G}\,x)$ has squared mass and gradient
	norm both equal to one. Ruf's estimate for $f/\sqrt2$ therefore gives
	\[
	 \int R_3(\vartheta w^2)
	 \le\frac MG\left(\frac{\vartheta G}{2\pi}\right)^3C_{\rm R}
	 =\frac{C_{\rm R}\vartheta^3}{8\pi^3}MG^2.
	\]
	Adding $(\vartheta^2/2)\int w^4$ and using the sharp cubic GN
	inequality proves \eqref{introscalarquartic}. The zero function is
	immediate. To check the leading coefficient, choose a scalar GN
	optimiser $w$ with mass $M$ and spread it by $w_t(x)=tw(tx)$.
	Its gradient norm squared is $t^2G$, its quartic integral is
	$t^2\int w^4$, and the displayed remainder is $O(t^4)$.
	Dividing by $\vartheta^2M\|\nabla w_t\|_2^2$ and letting
	$t\to0$ gives the limit $1/\|Q_4\|_2^2$.
\end{proof}

\section{Normalised planar systems with exponential nonlinearities}\label{secTMsolutions}

In the fixed-frequency framework for the planar exponential system \eqref{system_exp}, compactness, positivity and decay results are proved by Cassani and Zhang in
\cite{CassaniZhang2017}. Qin, Tang and Zhang \cite{QTZ} also study critical
exponential growth with a prescribed potential. For gradient systems,
Cassani, Tavares and Zhang \cite{CassaniTavaresZhang2020} treat
Bose-fluid interactions on bounded planar domains with fixed linear
coefficients. Chen, Xie and Zhang \cite{CXZ} obtain normalised exponential ground
states for a gradient system with two separate $L^2$ constraints.
Here we retain the crossed equations, a single common frequency and
the bilinear constraint $B=a$. The scalar theorem of
Chang, Liu and Yan \cite{ChangLiuYan2023} assumes, in particular,
$f(t)=o(|t|^3)$ at zero and a sufficiently large lower-growth
coefficient. Our nonlinearity satisfies $h_\vartheta(t)/t^3\to\vartheta>0$ as $t\to0$, so the results therein do not directly supply the present mass interval.

\vskip0.2truecm
Note first that every nontrivial solution $(u,v)\in H^1$ of \eqref{system_exp} is positive. Indeed, testing the first equation with $u_-$ gives $-\|\nabla u_-\|_2^2-\lambda\|u_-\|_2^2=\int h_{\vartheta_1}(v)u_-\ge0$. Hence $u_-=0$, and likewise $v_-=0$. The strong maximum principle gives then the strict positivity.

We start by showing that a comparison between the parameters $\vartheta_1$ and $\vartheta_2$ induces an ordering on the components, which forces both synchronisation in equality case $\vartheta_1=\vartheta_2$, and $L^2$-mass bounds of the components in terms of the bilinear mass $a$.

\begin{proposition}[Global comparison, mass bounds, and synchronisation]\label{exponentialcomparison}
	Let $(u,v)$ be a positive bounded $H^1$ solution of the pure
	exponential system \eqref{system_exp} with $\lambda>0$. If
	$\vartheta_1\ge\vartheta_2$ and
	$c=(\vartheta_1/\vartheta_2)^{1/2}$, then
	\begin{equation}\label{exponentialordering}
		v\le u\le\sqrt{\frac{\vartheta_1}{\vartheta_2}}\,v\,,\qquad
		a\le\|u\|_2^2\le a\,\sqrt{\frac{\vartheta_1}{\vartheta_2}}\,,\qquad
		a\,\sqrt{\frac{\vartheta_2}{\vartheta_1}}\le\|v\|_2^2\le a\,.
	\end{equation}
	For $\vartheta_1>\vartheta_2$ both pointwise inequalities are
	strict. For $\vartheta_1=\vartheta_2=\vartheta$, every such pair
	is synchronised, namely $u=v=w$, where
	\begin{equation}\label{synchro}
		-\Delta w+\lambda w=w(e^{\vartheta w^2}-1)\,,\qquad \|w\|_2^2=a\,.
	\end{equation}
	The antisymmetric part of the linearisation at $(w,w)$ is coercive:
	\begin{equation}\label{antisymmetricexponential}
		L_{\rm a}=-\Delta+\lambda+h_\vartheta'(w)\,,\qquad
		\langle L_{\rm a}\phi,\phi\rangle
		\ge\|\nabla\phi\|_2^2+\lambda\|\phi\|_2^2\,.
	\end{equation}
	Thus every kernel element of the full fixed-frequency linearisation
	is symmetric and corresponds to a kernel element of
	$-\Delta+\lambda-h_\vartheta'(w)$.
\end{proposition}

\begin{proof}
	Subtracting the two equations in \eqref{system_exp} and setting $W=u-v$, we obtain
	$$(-\Delta+\lambda)W=h_{\vartheta_1}(v)-h_{\vartheta_2}(u)\,.$$
	Since $h_{\vartheta_2}$ is increasing, we can write
	$$h_{\vartheta_2}(u)-h_{\vartheta_2}(v)=k(u-v)\,,$$
	where $k=(h_{\vartheta_2}(u)-h_{\vartheta_2}(v))/(u-v)$, continuously extended at points where $u=v$, is nonnegative and bounded. Consequently,
	$$(-\Delta+\lambda+k)W=h_{\vartheta_1}(v)-h_{\vartheta_2}(v)\ge0\,,$$
	since $\vartheta_1\ge\vartheta_2$ and $v>0$. Testing this inequality with $W_-$ gives $W_-=0$, and hence $v\le u$.
	
	We now prove the upper bound. Let $c=(\vartheta_1/\vartheta_2)^{1/2}\ge1$ and set $e=u-cv$. By the definition of $c$, we have $h_{\vartheta_2}(cv)=c h_{\vartheta_1}(v)$. 
	As before, write $h_{\vartheta_2}(u)-h_{\vartheta_2}(cv)=\widetilde k(u-cv)$, where $\widetilde k\geq0$ is bounded. Subtracting $c$ times the second equation in \eqref{system_exp} from the first one, we thus obtain
	$$(-\Delta+\lambda+c\widetilde k)e=(1-c^2)h_{\vartheta_1}(v)\le0\,.$$
	Testing this inequality with $e_+$ and arguing as before, $e_+=0$, thus $u\le cv$.
	
	If $\vartheta_1>\vartheta_2$, then $c>1$ and, 
	the strong maximum principle, applied to the equations for $W$ and $-e$, gives $W>0$ and $-e>0$ in $\R^2$. Hence $v<u<cv$ in $\R^2$. 
	and
	$$v^2\le uv\le u^2,\qquad u^2\le cuv,\qquad uv\le cv^2,$$
	which, integrated, prove the mass bounds in \eqref{exponentialordering}. If $\vartheta_1=\vartheta_2=\vartheta$, we have $c=1$, and the pointwise inequalities immediately imply $u=v=:w$ solves \eqref{synchro}. 
	
	Finally, we consider the linearisation at $(w,w)$. The linearised system at fixed $\lambda$ is
	$$\begin{cases}
		(-\Delta+\lambda)\phi-h_\vartheta'(w)\psi=0\,,\\
		(-\Delta+\lambda)\psi-h_\vartheta'(w)\phi=0\,.
	\end{cases}$$
	Adding and subtracting the two equations, we see that the symmetric component $\phi+\psi$ and the antisymmetric component $\phi-\psi$ satisfy, respectively,
	$$(-\Delta+\lambda-h_\vartheta'(w))(\phi+\psi)=0\,,\mbox{and}\quad(-\Delta+\lambda+h_\vartheta'(w))(\phi-\psi)=0\,.$$
	Thus, the operator acting on the antisymmetric component is $L_{\rm a}=-\Delta+\lambda+h_\vartheta'(w)$. Since $h_\vartheta'(w)\ge0$, \eqref{antisymmetricexponential} follows. 
	This implies that $L_{\rm a}$ has trivial kernel and concludes the proof. 
\end{proof}

\subsection{Energy barriers and radial compactness below the threshold}\label{Section_9.4}
Define
\begin{equation}\label{TMpotential}
	\mathcal P_\vartheta(u,v)=\frac1{2\vartheta}
	\int\big[\Psi(\vartheta u^2)+\Psi(\vartheta v^2)\big]\,,\qquad
	J^{\rm exp}_\varepsilon(u,v)=A-\varepsilon\mathcal P_\vartheta(u,v)
\end{equation}
and consider the common mass-preserving dilation
\begin{equation}\label{dilation}
	\mathcal D_t(u,v)(x)=(tu(tx),tv(tx))\,,\qquad t>0\,.
\end{equation}
This transformation preserves both $\mathcal E_K$ and $B$, whereas $A(\mathcal D_t(u,v))=t^2A(u,v)$.
To identify the leading contribution of the exponential potential, we use the expansion $\Psi(s)=\sum_{j=2}^{\infty}s^j/j!\,$. Since all terms are nonnegative, monotone convergence gives
\begin{equation*}
	\mathcal P_\vartheta(\mathcal D_t(u,v))=\sum_{j=2}^{\infty}\frac{\vartheta^{j-1}}{2j!}\,t^{2j-2}\int\big(u^{2j}+v^{2j}\big)\,,
\end{equation*}
with the equality understood in $[0,\infty]$.
For each fixed $(u,v)\in\mathcal E_K$, Theorem~\ref{introexponentialGN} ensures that this series is finite for some sufficiently small $t_0>0$. For $0<t\le t_0$ its terms with $j\ge3$ are bounded by
$(t/t_0)^4$ times their sum at $t_0$. Hence
\begin{equation*}
	J^{\rm exp}_\varepsilon(\mathcal D_t(u,v))=t^2\left[A-\frac{\varepsilon\vartheta}4\int(u^4+v^4)\right]+O(t^4)\quad\mbox{as}\ \,t\to0\,.
\end{equation*}
Thus, the quartic term is the leading nonlinear contribution: it scales exactly like $A$, while all higher-order terms are of order $t^4$ or smaller as $t\to0$. This shows that, at sufficiently small scales, the exponential functional is governed to leading order by its quartic truncation, which is the analogue of the quartic $AB$ inequality below the concentration threshold.
\vskip0.2truecm
A direct consequence of Theorem \ref{introexponentialGN} is that one can identify a mountain-pass geometry for the energy functional $J^{\rm exp}_\varepsilon$ on the constraint $\Sigma_a$. 
\begin{lemma}\label{MP-geometry_Je}
	Let $K\ge1$, $\vartheta>0$ and $(u,v)\in\mathcal E_K$ with $\vartheta K^2A(u,v)<4\pi$ and $B(u,v)=a$. For $\varepsilon>0$,
	\begin{equation}\label{TMenergybound}
		J^{\rm exp}_\varepsilon(u,v)\ge\left(1-\frac{C_\rho}2\varepsilon\vartheta K^4a\right)A(u,v)\,.
	\end{equation}
	For $\varepsilon\in\left(0,2(C_\rho\vartheta K^4a)^{-1}\right)$, this is a strictly positive barrier at $A(u,v)=4\pi\rho/(\vartheta K^2)$. Moreover, every nontrivial smooth compactly supported pair in $\cE_K$ has negative energy after sufficiently large common mass-preserving dilation.
\end{lemma}

\begin{proof}[Proof of Lemma~\ref{MP-geometry_Je}]
	Clearly \eqref{TMenergybound} follows from \eqref{introtruncatedbound}, using $B=a$ in \eqref{TMpotential}. Set now $\delta_\rho=1-\frac12C_\rho\varepsilon\vartheta K^4a>0$ and $R_\rho=4\pi\rho/(\vartheta K^2)$. The energy bound \eqref{TMenergybound} gives $J^{\rm exp}_\varepsilon\ge\delta_\rho R_\rho$ on $A=R_\rho$. For the negative-energy assertion,
	$\Psi(s)\ge\frac16s^3$ for $s\ge0$ gives
	$$J^{\rm exp}_\varepsilon(\mathcal D_t(u,v))\le t^2A-\frac{\varepsilon\vartheta^2t^4}{12}\int(u^6+v^6)\to-\infty\,.$$
	The integral is positive and each energy is finite for the stated pair.
\end{proof}

Fix now $\vartheta_1,\vartheta_2>0$ and define $h_\vartheta$ and $H_\vartheta$ as in \eqref{TMnonlinearity},
so that $H_\vartheta'=h_\vartheta$. 

\begin{lemma}[Local differentiability near bounded profiles]\label{TMmapping}
	Let $w\in H^1(\R^2)\cap L^\infty(\R^2)$ and $\vartheta>0$.
	On a sufficiently small $H^1$ ball about $w$, the map
	$z\mapsto h_\vartheta(z)$ is $C^1$ into $L^2(\R^2)$, with derivative
	$\eta\mapsto h_\vartheta'(z)\eta$.
\end{lemma}
\begin{proof}
	For any fixed $L>0$, write $z=w+h$. The elementary estimate
	\[
	e^{Lz^2}-1
	\le e^{2L\|w\|_\infty^2}(e^{2Lh^2}-1)
	+(e^{2Lw^2}-1)
	\]
	and \eqref{Rufconstant} give a uniform bound for
	$\int(e^{Lz^2}-1)$ when $\|h\|_{H^1}<\rho$ and
	$2L\rho^2\le4\pi$. The last term is integrable since $w$ is bounded
	and square integrable. By choosing $L$ sufficiently large and then
	$\rho$ small, H\"older's inequality and the embeddings
	$H^1(\R^2)\hookrightarrow L^r(\R^2)$, $2\le r<\infty$, give
	uniform bounds for
	$h_\vartheta(z)$ in $L^2$, $h_\vartheta'(z)$ in $L^4$, and
	$h_\vartheta''(z)$ in $L^8$ on this ball. Indeed $h_\vartheta\in C^2(\R)$
	and, with constants depending on $\vartheta$,
	\[
	|h_\vartheta(t)|\le C|t|^3e^{\vartheta t^2},\quad
	|h_\vartheta'(t)|\le C|t|^2e^{\vartheta t^2},\quad
	|h_\vartheta''(t)|\le C(|t|+|t|^5)e^{\vartheta t^2}.
	\]
	Polynomial factors are controlled by Sobolev embedding, while the
	exponential minus $1$ is controlled by the preceding bound with
	a larger $L$.
	For $y,z$ in the ball, the mean-value formula and H\"older give
	\[
	\|h_\vartheta'(z)-h_\vartheta'(y)\|_4
	\le C\|z-y\|_8\le C\|z-y\|_{H^1}.
	\]
	Multiplication $L^4\times L^4\to L^2$ proves continuity of the
	claimed derivative in operator norm. The integral Taylor formula
	then bounds the $L^2$ remainder by $C\|z-y\|_{H^1}^2$ and proves
	Fr\'echet differentiability.
\end{proof}

Recall the definition of the functionals $F_{\rm e}$, $S_{\rm e}$, $J_{\rm e}$, and $P_{\rm e}$ in \eqref{TMpurefunctionals}. Note that $P_{\rm e}$ is the derivative of $J_{\rm e}(\mathcal D_t(u,v))$ at $t=1$ ($\mathcal D_t$ is defined in \eqref{dilation}) whenever differentiation is justified, and is the Poho\v zaev quantity used below. We also recall the notation in Theorem~\ref{introTMcompactness}
\begin{equation}\label{notation10.5}
	\Theta=\max\{\vartheta_1,\vartheta_2\}\,,\qquad m_0=\|Q_4\|_2^2\,,\qquad \kappa_a=\frac{\Theta K^4a}{m_0}\,.
\end{equation}

\begin{lemma}[An energy-level bound for the bilinear gradient]\label{TMlevelbound}
	Let $(u,v)\in\mathcal E_K$ with $B(u,v)=a$ and suppose that $F_{\rm e}(u,v),S_{\rm e}(u,v)$ are finite. Then
	\begin{equation*}
		(1-\kappa_a)A(u,v)\le2J_{\rm e}(u,v)-\frac12P_{\rm e}(u,v)\,.
	\end{equation*}
\end{lemma}
\begin{proof}
	Expand the nonnegative exponential series and, for $j\ge2$, set
	\[
	 F_j(u,v)=\frac1{2j!}\int[\vartheta_2^{j-1}u_+^{2j}+\vartheta_1^{j-1}v_+^{2j}]\,.
	\]
	Monotone convergence gives $F_{\rm e}=\sum_{j\ge2}F_j$ and
	$S_{\rm e}=\sum_{j\ge2}2jF_j$. Consequently
	\begin{equation}\label{eq_lem9.2}
		A(u,v)-2J_{\rm e}(u,v)+\frac12P_{\rm e}(u,v)=3F_{\rm e}(u,v)-\frac12S_{\rm e}(u,v)=\sum_{j\ge2}(3-j)F_j(u,v)\le F_2(u,v)\,.
	\end{equation}
	The sharp planar quartic inequality \cite{Weinstein} reads
	$\|w\|_4^4\le(2/m_0)\|w\|_2^2\|\nabla w\|_2^2$.
	Applied to the two positive parts, the component bounds imply
	$F_2\le\Theta K^4aA/m_0$, which together with \eqref{eq_lem9.2} proves the claim.
\end{proof}

	The exponential Gagliardo-Nirenberg inequality in Theorem \ref{introexponentialGN} supplies uniform control through the following compactness criterion for Poho\v zaev-Palais-Smale sequences for the functional $J_{\rm e}$ associated to problem \eqref{system_exp}, below a suitable energy threshold.
	\begin{proposition}[Compactness for Poho\v zaev-Palais-Smale sequences]\label{introTMcompactness}
		Let	$K\ge1$, $a>0$, and assume $\kappa_a<1$, see \eqref{notation10.5}. Suppose a radial sequence $(u_n,v_n)\in\mathcal E_K$ with
		$B(u_n,v_n)=a$, is such that $F_{\rm e}(u_n,v_n)$, $S_{\rm e}(u_n,v_n)$ are finite, and
		\begin{equation}\label{TMcompactlevel}
			J_{\rm e}(u_n,v_n)\to c\,,\quad\ P_{\rm e}(u_n,v_n)\to0\,,\quad\ 
			0<c<\frac{2\pi(1-\kappa_a)}{\Theta K^2}\,.
		\end{equation}
		If for some real $\lambda_n$ the two residuals
		\begin{equation}\label{TMcompactresiduals}
			-\Delta u_n+\lambda_nu_n-h_{\vartheta_1}(v_n)\to0\,,\qquad
			-\Delta v_n+\lambda_nv_n-h_{\vartheta_2}(u_n)\to0
		\end{equation}
		in $H^{-1}(\R^2)$, up to a subsequence, $(u_n,v_n)$ converges strongly in
		$H^1(\R^2)^2$ and $\lambda_n\to\lambda>0$. The limit is a positive
		radial classical exponentially decaying solution of bilinear
		mass $a$ and energy $c$ of \eqref{TMpuresystem}.
	\end{proposition}

\begin{proof}
	Lemma~\ref{TMlevelbound} yields
	\[
	 \limsup_{n\to+\infty}A(u_n,v_n)\le\frac{2c}{1-\kappa_a}<\frac{4\pi}{\Theta K^2}\,.
	\]
	Since $(u_n,v_n)\in\cE_K$ and $B(u_n,v_n)=a$, then the sequence is bounded in
	$H^1(\R^2)^2$ and $\Theta K^2A(u_n,v_n)\le4\pi\rho$ holds for some $\rho<1$ and $n$ large.
	Hence there exists a weakly convergent subsequence with limit $(u,v)$ and, by radial compactness, the convergence is strong in $L^t(\R^2)^2$ for every $t\in(2,+\infty)$ and almost everywhere. One may obtain the global statement by combining local Rellich compactness with the radial bound
	$|w(r)|\le C r^{-1/2}\|w\|_{H^1}$ and the uniform $L^2$ bound.
	
	Choose $1<r<\min\{2,1/\rho\}$ and then $b>r\Theta$ sufficiently
	close to $r\Theta$ so that $bK^2A(u_n,v_n)<4\pi$ uniformly.
	Theorem~\ref{introexponentialGN} bounds the integrals of
	$e^{b u_n^2}-1$ and $e^{b v_n^2}-1$ uniformly, while all polynomial factors in the nonlinearities are absorbed thanks to the strict inequality $b>r\Theta$. In particular, for $\vartheta\in\{\vartheta_1,\vartheta_2\}$, one has
	\[
	 \|h_\vartheta(u_n)\|_r+\|h_\vartheta'(u_n)\|_r
	 +\|h_\vartheta(v_n)\|_r+\|h_\vartheta'(v_n)\|_r\le C\,.
	\]
	At zero the respective powers are $|t|^{3r}$ and $|t|^{2r}$,
	so these bounds introduce no constant integrand on the infinite domain.
	Fatou's lemma gives the same bounds for the limit $(u,v)$.
	
	The nonnegative derivative $h_\vartheta'$ is nondecreasing. 
	By choosing $s=2r/(r-1)>2$, from mean value and Hölder’s inequalities we obtain
	\[\begin {aligned}
		\|h_\vartheta(u_n)-h_\vartheta(u)\|_{H^{-1}}
		&\le C\|h_\vartheta(u_n)-h_\vartheta(u)\|_{2r/(r+1)}\\
		&\le C\big(\|h_\vartheta'(u_n)\|_r+\|h_\vartheta'(u)\|_r\big)\|u_n-u\|_s\to0\,.
	\end{aligned}\]
	In the first inequality, we have also used $H^1(\R^2)\hookrightarrow L^s(\R^2)$, which, by duality, implies $L^{R^2}\hookrightarrow H^{-1}(\R^2)$. The same argument applies to $v_n$. Therefore, the nonlinearities converge strongly to $0$ in $H^-1(\R^2)$. Since $r'=r/(r-1)>2$, the analogous mean-value bound for $H_\vartheta$, using strong $L^{r'}$ convergence, gives $F_{\rm e}(u_n,v_n)\to F_{\rm e}(u,v)$. Strong $H^{-1}$ convergence of the nonlinearities and weak $H^1$
	convergence of the components also give $S_{\rm e}(u_n,v_n)\to S_{\rm e}(u,v)$.
	
	Testing now \eqref{TMcompactresiduals} against $v_n$ and $u_n$, respectively, and adding the resulting identities, we obtain
	$$2A(u_n,v_n)+2\lambda_n a-S_(u_n,v_n)=o(1)\,.$$
	On the other hand, the condition $P_(u_n,v_n)\to0$ yields
	$$2A(u_n,v_n)-S_(u_n,v_n)+2F_(u_n,v_n)=o(1)\,.$$
	Consequently, $\lambda_n a=F_(u_n,v_n)+o(1)$, so
	$\lambda_n\to\lambda=F_{\rm e}(u,v)/a\ge0$.
	If $F_{\rm e}(u,v)=0$, both positive parts of the limit vanish,
	and hence $S_{\rm e}(u,v)=0$. The Poho\v zaev limit would then imply
	$A(u_n,v_n)\to0$ and $J_{\rm e}(u_n,v_n)\to0$, contradicting $c>0$.
	Thus $\lambda>0$.
	
	The operator $-\Delta+\lambda:H^1\to H^{-1}$ is an isomorphism.
	Applying its inverse to \eqref{TMcompactresiduals} and replacing
	$\lambda_n$ by $\lambda$ one finds that the additional errors tend to zero in
	$H^{-1}$. The already proved nonlinear convergence now gives strong convergence in
	$H^1$ of both components. In particular $B(u,v)=a$, $J_{\rm e}(u,v)=c$, and the limiting equations hold. Testing the equations against the negative parts of the solutions, and using the maximum principle, we obtain strict positivity. Moreover, the nonlinear terms belong to $L^r$ for some $r>1$. Elliptic regularity in dimension two therefore yields boundedness, which in turn allows us to invoke classical regularity theory. Uniform continuity, together with square integrability, implies that both components vanish at infinity. Finally, since $h_\vartheta(t)/t\to0$ as $t\to0$, an exterior comparison argument applied to $u+v$ yields exponential decay, as in Theorem~\ref{profileregularity}.
\end{proof}
The same proof works when the radial residuals are measured in $H^{-1}_{\rm rad}$, since spherical averaging of a test function is a contraction in $H^1$. 

\subsection{A tilde-map construction on the mass constraint}\label{sectildelinking}
We apply the tilde-map construction of
Subsection~\ref{sectildepreliminaries} on $H=H^1_{\rm rad}(\R^2)$,
choosing
\begin{equation}\label{tildeweights}
 \vartheta_0=\sqrt{\vartheta_1\vartheta_2},\qquad
 b=(\vartheta_1/\vartheta_2)^{1/4},\qquad
 \ell=\min\{b^2,b^{-2}\}>0\,.
\end{equation}
This balances the exponential coefficients in the positive and negative
coordinates \eqref{tildecoordinates}. By Lemma~\ref{tildemaplemma},
\[
 \widetilde u=b^{-2}u\quad\mbox{and}\quad\mathcal T(u,v)=(b^2v,b^{-2}u)\,.
\]
Retain $r_a(w)=(2a+\|w\|_2^2)^{1/2}$ and write
$\mathbb S=\{\zeta\in H:\|\zeta\|_2=1\}$.
The parametrisation \eqref{tildechart} identifies
$\mathbb S\times H$ with the radial mass constraint $B=a$.
We establish the linking geometry and uniform bounds for its levels
below. In subsection~\ref{sectildecompactness} we prove the compactness
criterion, and in Subsection~\ref{secfullfibre} we construct a sequence
satisfying its assumptions using all negative directions.

We shall use $J_{\rm e}$ on this entire radial constraint, without imposing the component inequalities defining $\mathcal E_K$. Note that $J_{\rm e}:H^1(\R^2)^2\to\R$ is a $C^1$ functional. To see this, approximate any
$u\in H^1$ by a bounded smooth function $u_0$ in $H^1$.
For each fixed $L>0$, Lemma~\ref{TMmapping} provides uniform control of $e^{L(u_0+h)^2}-1$ whenever $\|h\|_{H^1}$ is sufficiently small, with a smallness threshold independent of $\|u_0\|_\infty$.
Choosing $u_0$ sufficiently close to $u$, the same estimates then hold throughout a neighbourhood of $u$. They provide local $C^1$ regularity of $h_\vartheta:H^1\to L^2$ and of the potential integral.

Let $W_m\subset H$ be any finite-dimensional subspace and define
\[
 \Sigma_{a,m}=\mathcal X_a(\mathbb S\times W_m)\,.
\]
The positive variable is unrestricted; only the negative directions
are approximated. One can choose nested finite-dimensional spaces
$W_n\subset C^\infty_{c,\rm rad}(\R^2)$ whose union is dense in $H$,
and denote the associated linear test spaces by
\begin{equation}\label{tildeGalerkinspace}
	E_n=\left\{\left(\frac b{\sqrt2}(\eta+\xi),
	\frac1{b\sqrt2}(\eta-\xi)\right):
	\eta\in H,\ \xi\in W_n\right\}.
\end{equation}
Their norms are inherited from $H^1(\R^2)^2$.


\begin{proof}[Proof of Theorem~\ref{introtildelinking}]
	Recall the definition of $R_0,\delta,\beta_0$ in \eqref{tildebarrierconstants}.
	Let $\phi\in H$ be a a nonnegative bounded and compactly supported radial function
	with $\|\phi\|_2=1$, and set
	\[
	 \phi_t(x)=t\phi(tx)\,,\quad\ 
	 Q_m=[t_-,t_+]\times\{w\in W_m:\|w\|_{H^1}\le R\}\,,\quad\ 
	 \chi_m(t,w)=\mathcal X_a(\phi_t,w)\,.
	\]
	Let $\Gamma_m$ be the continuous maps $\gamma:Q_m\to\Sigma_{a,m}$
	that agree with $\chi_m$ on $\partial Q_m$. We will choose the
	parameters independently of $W_m$ to obtain
	\begin{equation}\label{tildeminimax}
		\sup_{\partial Q_m}J_{\rm e}(\chi_m)<\beta_0\quad\mbox{and}\quad\beta_0\le c_m:=\inf_{\gamma\in\Gamma_m}\max_{Q_m}J_{\rm e}(\gamma)\le\overline c<4\pi/\vartheta_0\,.
	\end{equation}
	We treat separately the barrier, the cylinder estimates, and the linking argument.
	
	\emph{The positive-graph barrier.}
	At $w=0$ write $h=z/\sqrt2$, so $\|h\|_2^2=a$, $A=\|\nabla h\|_2^2$,
	and $(u,v)=(bh,b^{-1}h)$. The choice of $b$ gives
	\begin{equation*}
		F_{\rm e}(u,v)=(b^2+b^{-2})\int H_{\vartheta_0}(h)\le(b^2+b^{-2})\frac{C_\rho\vartheta_0}4aA\,,
	\end{equation*}
	for $\vartheta_0A\le4\pi\rho$, by \eqref{introtruncatedbound} applied to $(h,h)\in\mathcal E_1$,
	using $h_+$ in the potential and $h$ in its upper bound. Since $(b^2+b^{-2})\vartheta_0=\vartheta_1+\vartheta_2$ by \eqref{tildeweights}, then $J_{\rm e}(bh,b^{-1}h)\ge\delta A$. Hence the set
	$$\mathcal S_{a,R_0}=\{(bh,b^{-1}h):h\in H,\ \|h\|_2^2=a,\ \|\nabla h\|_2^2=R_0\}$$
	has energy at least $\beta_0$.
	
	\emph{Bounds on the mass-preserving cylinder.}
	For the function $\phi$ taken above, set
	$I_j=\int\phi^j>0$ for $j=4,6$. Convexity of $H_{\vartheta_0}$ on $\R$, see \eqref{TMnonlinearity}, including its zero extension on the negative half-line, gives
	for the coordinates \eqref{tildecoordinates}
	\begin{equation}\label{tildejensen}
		F_{\rm e}(u,v) =\int\left[b^2H_{\vartheta_0}\left(\frac{z+w}{\sqrt2}\right)+b^{-2}H_{\vartheta_0}\left(\frac{z-w}{\sqrt2}\right)\right]\ge 2\ell\int H_{\vartheta_0}\left(\frac{z}{\sqrt2}\right).
	\end{equation}
	Set $z=r_a(w)\phi_t$, where $r_a(w)^2=2a+\|w\|_2^2$ by \eqref{S_ra}. 
	The quartic and sixth-order terms of the
	potential respectively imply
	\begin{align}
		J_{\rm e}(\mathcal X_a(\phi_t,w))
		&\le t^2\left(\frac{\|\nabla\phi\|_2^2r_a(w)^2}2-\frac{\ell\vartheta_0I_4r_a(w)^4}8\right)-\frac{\|\nabla w\|_2^2}2
		\le\frac{\|\nabla\phi\|_2^4}{2\ell\vartheta_0I_4}t^2,\label{tildebottom}\\
		J_{\rm e}(\mathcal X_a(\phi_t,w))
		&\le\frac{\|\nabla\phi\|_2^2t^2r_a(w)^2}2-\frac{\|\nabla w\|_2^2}2-\frac{\ell\vartheta_0^2I_6t^4r_a(w)^6}{48}\,.
	\label{tildetop}
	\end{align}
	Choose $t_-$ so small that \eqref{tildebottom} is less than
	$\beta_0/2$ and $at_-^2\|\nabla\phi\|_2^2<R_0$. Since $r_a(w)^2\ge2a$, choose $t_+$
	so large that \eqref{tildetop} is nonpositive for all $w$ and
	$at_+^2\|\nabla\phi\|_2^2>R_0$.
	For $t\in[t_-,t_+]$, \eqref{tildetop} is bounded above by
	$$\frac{\|\nabla\phi\|_2^2t_+^2r_a(w)^2}2-\frac{\|\nabla w\|_2^2}2-\frac{\ell\vartheta_0^2I_6t_-^4r_a(w)^6}{48}\,.$$
	This tends uniformly to $-\infty$ as $\|w\|_{H^1}^2=(r_a(w)^2-2a)+\|\nabla w\|_2^2\to+\infty$. Thus a radius $R$,
	independent of $W_m$, makes the lateral boundary energy nonpositive. This proves the first inequality in \eqref{tildeminimax}.
	
	A more precise upper bound follows directly from \eqref{tildejensen}. Put $s=r_a(w)/\sqrt2\ge\sqrt a$ and $\xi=st$. Then
	\begin{equation}\label{tildeuniformupper}
		\begin{aligned}
			J_{\rm e}(\mathcal X_a(\phi_t,w))&\le\xi^2\|\nabla\phi\|_2^2-\frac{2\ell s^2}{\xi^2}\int H_{\vartheta_0}(\xi\phi)-\frac{\|\nabla w\|_2^2}2\\
	 		&\le\xi^2\|\nabla\phi\|_2^2-\frac{2\ell a}{\xi^2}\int H_{\vartheta_0}(\xi\phi)=:\mathcal U_{a,\phi}(\xi)\,.
		\end{aligned}
	\end{equation}
	One has $\mathcal U_{a,\phi}(\xi)\to0$ as $\xi\to0$ and $\mathcal U_{a,\phi}(\xi)\to-\infty$
	as $\xi\to+\infty$; the latter already follows from the sixth-order term. Its maximum is therefore finite and independent of $m$.
	We claim here that one can choose $\phi$ so that the maximum of $\mathcal U_{a,\phi}(\xi)$ is below $4\pi/\vartheta_0$, and we postpone the proof of this fact to Lemma~\ref{tildemoser} below.
	All the preceding choices can then be made for that fixed $\phi$.
	
	\emph{Linking on the constraint.}
	For $\gamma\in\Gamma_m$, denote its coordinates by $(z_\gamma,w_\gamma)\in H\times W_m$ and consider the map
	$$F_\gamma(t,w):=\left(w_\gamma(t,w),\frac12\|\nabla z_\gamma(t,w)\|_2^2-R_0\right)\in W_m\times\R\,.$$
	On \(\partial Q_m\), this map coincides with
	$$F_0(t,w)=\left(w,\frac12r_a(w)^2t^2\|\nabla\phi\|_2^2-R_0\right),$$
	which does not vanish there. Moreover, \(F_0\) has a unique zero in \(Q_m\), namely $(t,w)=\left(0,R_0/(a\|\nabla\phi\|_2^2))^{1/2}\right)$ and its local degree is nonzero.
	By homotopy invariance of the Brouwer degree, \(F_\gamma\) must therefore vanish at some interior point of \(Q_m\). The corresponding image under \(\gamma\) belongs to \(\mathcal S_{a,R_0}\), and hence $\max_{Q_m}J_{\rm e}(\gamma)\geq\beta_0$. Since the reference map belongs to \(\Gamma_m\), estimate \eqref{tildeuniformupper} also provides a uniform upper bound.

	In particular, \(\Sigma_{a,m}\) is complete with respect to its induced length metric. A standard deformation argument therefore yields a Cerami sequence \((Z_{m,k})\subset\Sigma_{a,m}\) at level \(c_m\), namely such that
	\begin{equation*}
		J_{\rm e}(Z_{m,k})\to c_m\,,\qquad(1+\|Z_{m,k}\|_{H^1\times H^1})\|d(J_{\rm e}|_{\Sigma_{a,m}})(Z_{m,k})\|\to0\,.
	\end{equation*}
		
	$$\left\|(J_{\rm e}'(Z_{m,k})+\lambda_{m,k}B'(Z_{m,k}))|_{E_m}\right\|=\left\|d(J_{\rm e}|_{\Sigma_{a,m}})(Z_{m,k})\right\|\,.$$
	Thus the Cerami sequence also satisfies the weighted Galerkin residual estimate required below. 
\end{proof}

\begin{lemma}[A uniform Moser upper bound for the linking levels]\label{tildemoser}
	For every $a>0$ and $\vartheta_1,\vartheta_2>0$, there exists a nonnegative radial $\phi\in H$ which is bounded and compactly supported with $\|\phi\|_2=1$, such that
	\begin{equation}\label{tildesharpupper}
		\sup_{\xi>0}\,\mathcal U_{a,\phi}(\xi)<\frac{4\pi}{\vartheta_0}\,,
	\end{equation}
	where $\mathcal U_{a,\phi}$ is defined in \eqref{tildeuniformupper}.
\end{lemma}

\begin{proof}
For $L>1$, take the standard radial Moser function
\[
 m_L(r)=\frac1{\sqrt{2\pi}}
 \begin{cases}
 \sqrt L\,,&0\le r\le e^{-L}\,,\\
 -L^{-1/2}\log r\,,&e^{-L}<r<1\,,\\
 0\,,&r\ge1\,.
 \end{cases}
\]
Then a direct computation shows that $\|\nabla m_L\|_2^2=1$ and $\|m_L\|_2^2\sim1/(4L)$ for $L\to+\infty$. Set 
$$\phi_L=\frac{m_L}{\|m_L\|_2}\quad\mbox{and}\quad y=\frac{\xi^2}{\|m_L\|_2^2}$$
and
\[
f_L(y):=y-\frac{2\ell a}{y\|m_L\|_2^2}\int H_{\vartheta_0}(\sqrt y\,m_L)\quad\mbox{for}\ \,y>0\,,
\]
The desired estimate \eqref{tildesharpupper} follows once we show that $\displaystyle\sup_{y>0}f_L(y)<\frac{4\pi}{\vartheta_0}$.

We first consider the case $y\ge\frac{4\pi}{\vartheta_0}$. On the disk $B_{e^{-L}}(0)$, we have $m_L^2=L/(2\pi)$, and hence
\[
\vartheta_0 y m_L^2=\frac{\vartheta_0 yL}{2\pi}\ge 2L\,.
\]
Thus, for $L$ sufficiently large, the exponential term dominates the lower-order terms in the definition of $H_{\vartheta_0}$. In this range of $L$, integrating over $B_{e^{-L}}(0)$ and recalling the basic estimate $e^t-1-t\ge e^t/2$, we obtain
\begin{equation*}
	\int H_{\vartheta_0}(\sqrt y\,m_L)\ge\frac{\pi}{4\vartheta_0}\exp\left(\frac{\vartheta_0L}{2\pi}\left(y-\frac{4\pi}{\vartheta_0}\right)\right),
\end{equation*}
which implies
\begin{equation}\label{eqL9.3}
	f_L(y)\leq y\Bigg(1-\frac{\ell a\pi}{2\vartheta_0\|m_L\|_2^2}\,\frac{\exp\left({\frac{\vartheta_0 L}{2\pi}\big(y-\frac{4\pi}{\vartheta_0}\big)}\right)}{y^2}\Bigg)\leq y\left(1-\frac{\ell a\vartheta_0}{32\pi\|m_L\|_2^2}\right),
\end{equation}
having substituted in the second term factor the value $y=\frac{4\pi}{\vartheta_0}$, since this term is increasing in $y$ on $\big[\frac{4\pi}{\vartheta_0},\infty\big)$, provided $L>\frac{\vartheta_0}{4\pi}$. Since \(\|m_L\|_2\to0\) as \(L\to+\infty\), we can choose $L$ sufficiently large so that $\frac{\ell a\vartheta_0}{32\pi\|m_L\|_2^2}>1$, which implies $f_L<0$ for every $y\geq\frac{4\pi}{\vartheta_0}$.

On the other hand, for $y\in\big(0,\frac{4\pi}{\vartheta_0}\big)$, the positivity of \(H_{\vartheta_0}\) directly gives $f_L(y)<y<T$. Since \(f_L\) extends continuously to $y=0$, with $f_L(0)=0$, it follows that $f_L$ attains a maximum on $\big[0,\frac{4\pi}{\vartheta_0}\big]$, and this maximum is strictly less than $T$. This concludes the proof of $\displaystyle\sup_{y>0}f_L(y)<\tfrac{4\pi}{\vartheta_0}$ and in turn of \eqref{tildesharpupper}.
\end{proof}

\subsection{Compactness results}\label{sectildecompactness}
We now prove compactness at the same level detected in Lemma \ref{tildemoser}.
The estimates apply to the full mass constraint and require no
component bounds of the form $\mathcal E_K$. They also allow residuals
to be tested only in the Galerkin spaces. The sequence must, however,
have an approximate Poho\v zaev identity and nonnegative multipliers.
The full-fibre construction in Subsection~\ref{secfullfibre} will
produce full-space sequences with exact Poho\v zaev identity and positive
multipliers, verifying these hypotheses.

For $Z=(u,v)\in H^1\times H^1$ introduce the notation
\begin{equation}\label{ProofThm1.15notation}
	\mathcal T Z=(b^2v,b^{-2}u),\quad\ M_b(Z)=b^{-2}\|u\|_2^2+b^2\|v\|_2^2,\quad\ G_b(Z)=b^{-2}\|\nabla u\|_2^2+b^2\|\nabla v\|_2^2\,,
\end{equation}
with $b$ defined in \eqref{tildeweights}. Clearly $\cT Z\in H^1\times H^1$. Applying the mass-tangent correction \eqref{tildemasscorrection} to $J_{\rm e}$ gives
\begin{equation}\label{tildetangentidentity}
 \begin{aligned}
 Y_Z&=\mathcal T Z-\frac{M_b(Z)}{2a}Z\,,\qquad B'(Z)[Y_Z]=0\,,\\
 J_{\rm e}'(Z)[Y_Z]
 &=G_b(Z)-\int\mathcal C_b(Z)-\frac{M_b(Z)}{2a}(2A-S_{\rm e})\,,
 \end{aligned}
\end{equation}
where
\begin{equation*}
	\mathcal C_b(u,v)=b^2v h_{\vartheta_2}(u)+b^{-2}u h_{\vartheta_1}(v)\,.
\end{equation*}
In the coordinates \eqref{tildecoordinates}, $\mathcal T$ sends
$(z,w)$ to $(z,-w)$. Both $Z$ and $\mathcal T Z$ therefore belong
to the same linear Galerkin space $E_n$ as $Z$, see \eqref{tildeGalerkinspace}.

\begin{lemma}\label{tildepointwise}
	For every $\varepsilon>0$ there exists  a positive constant $C_\varepsilon$ depending
	only on $\varepsilon,\vartheta_1,\vartheta_2$, such that, for all $u,v,t\in\R$,
	\begin{equation}\label{tildecrossbound}
		\mathcal C_b(u,v)\le(1+\varepsilon)[u h_{\vartheta_2}(u)+v h_{\vartheta_1}(v)]+C_\varepsilon[H_{\vartheta_2}(u)+H_{\vartheta_1}(v)]
	\end{equation}
	and, for $\vartheta\in\{\vartheta_1,\vartheta_2\}$,
	\begin{equation}\label{tildepotentialbound}
		H_\vartheta(t)\le\varepsilon t h_\vartheta(t)+C_\varepsilon t^2\,.
	\end{equation}
	In particular, $\mathcal C_b(u,v)\le C[u h_{\vartheta_2}(u)+v h_{\vartheta_1}(v)]$ with a fixed constant $C$.
\end{lemma}
\begin{proof}
Write $x=b^{-1}u$ and $y=bv$, so that both exponentials have same parameter $\vartheta_0$. Hence,
$$\mathcal C_b(u,v)=b^2y h_{\vartheta_0}(x)+b^{-2}x h_{\vartheta_0}(y)\quad\mbox{and}\quad u h_{\vartheta_2}(u)+v h_{\vartheta_1}(v)=b^2x h_{\vartheta_0}(x)+b^{-2}y h_{\vartheta_0}(y)\,.$$
Observe that, for $s,t\ge0$, one has
$$s h_{\vartheta_0}(t)\le(1+\varepsilon)t h_{\vartheta_0}(t)+C_\varepsilon H_{\vartheta_0}(s)\,.$$
Indeed, the first term suffices when $s\le(1+\varepsilon)t$. If otherwise $t<s/(1+\varepsilon)$, the ratio
$s h_{\vartheta_0}(s/(1+\varepsilon))/H_{\vartheta_0}(s)$
is bounded at zero and tends to zero at infinity. If either variable
is nonpositive, the same inequality is immediate. Applying it twice
proves \eqref{tildecrossbound}, with the fixed weights absorbed into
$C_\varepsilon$. The ratio $H_\vartheta(t)/(t h_\vartheta(t))$
tends to zero as $t\to+\infty$, whereas $H_\vartheta(t)=O(t^4)$
at zero; this proves \eqref{tildepotentialbound}. Finally
$4H_\vartheta(t)\le t h_\vartheta(t)$ gives the last assertion.
\end{proof}

\begin{lemma}[Exponential integrability with a quartic bound]\label{tildequarticvanishing}
	If $\alpha,L>0$ and $\alpha L<4\pi$, then
	\begin{equation}\label{tildequarticestimate}
		\int f^2(e^{\alpha f^2}-1)\le C_{\alpha,L}\|f\|_4^4\quad\text{whenever }f\in H^1(\R^2)\ \,\mbox{with}\ \,\|\nabla f\|_2^2\le L\,.
	\end{equation}
	In particular, if $\|\nabla f\|_2^2\le L<4\pi/\alpha$ uniformly, then the left-hand side tends to zero as $\|f\|_4\to0$.
\end{lemma}
\begin{proof}
	On $\{|f|\le1\}$ the integrand on the left is bounded by a constant times $|f|^4$.
	Consider now the region $\{|f|>1\}$ and fix $\beta\in(\alpha,4\pi/L)$. The standard estimate $t^2(e^{\alpha t^2}-1)\le C_{\alpha,\beta}e^{\beta(t-1)^2}$ for $t\ge1$ permits to absorb polynomial factors into the exponential. Schwarz rearrangement places $g$ in $H^1_0$ of a disk of area $|\{|f|>1\}|$ and does not increase its Dirichlet integral. The bounded-domain Moser inequality \cite{Moser1971}, rescaled to that disk, then yields
	$$\int_{\{|f|>1\}}e^{\beta(|f|-1)^2}\le C_{\beta,L}|\{|f|>1\}|\le C_{\beta,L}\int_{\{|f|>1\}}|f|^4\le C_{\beta,L}\|f\|_4^4\,,$$
	from which the assertion follows.
\end{proof}


\begin{proof}[Proof of Theorem~\ref{introtildecompactness}]

The Cerami condition in \eqref{tildeCPconditions} allows us to test against both \(Z_n\) and \(\mathcal T Z_n\), even before establishing boundedness of the sequence. Testing against \(Z_n\), we get
\begin{equation}\label{tildeARlimit}
	\lambda_n a=F_{\rm e}(Z_n)+o(1)\,,\quad A(Z_n)=c+F_{\rm e}(Z_n)+o(1)\,,\quad S_{\rm e}(Z_n)=2c+4F_{\rm e}(Z_n)+o(1)\,.
\end{equation}
Indeed, the second and the third identities follow by $J_{\rm e}(Z_n)=c+o(1)$ and $P_{\rm e}(Z_n)=o(1)$. Moreover, from the Cerami condition tested with $Z_n$ we get $\left|J_{\rm e}'(Z_n)[Z_n]+\lambda_n B'(Z_n)[Z_n]\right|\leq\left\|(J_{\rm e}'(Z_n)+\lambda_n B'(Z_n))|_{E_n}\right\|_{E_n^*}\|Z_n\|_{H^1\times H^1}=o(1)$, which, using $B'(Z_n)[Z_n]=2a$, can be rewritten as $2A(Z_n)-S_{\rm e}(Z_n)+2\lambda_n a=o(1)$. The first identity in \eqref{tildeARlimit} is then obtained by combining it with $P_{\rm e}(Z_n)=o(1)$.

Furthermore, test \eqref{tildeCPconditions} instead against $\mathcal T Z_n$, use $B'(Z_n)[\mathcal T Z_n]=M_b(Z_n)$ and the identity $J_{\rm e}'(Z_n)[\mathcal T Z_n]=G_b(Z_n)-\int\mathcal C_b(u_n,v_n)$ to obtain
\begin{equation}\label{tildecoercivetest}
	G_b(Z_n)+\lambda_n M_b(Z_n)=\int\mathcal C_b(u_n,v_n)+o(1)\le C S_{\rm e}(Z_n)+o(1)\,,
\end{equation}
by Lemma~\ref{tildepointwise}, applied pointwise and then integrated.

We first prove uniform boundedness of $(\lambda_n)_n$ and $(Z_n)_n$. Suppose by contradiction that $\lambda_n\to+\infty$, then $F_{\rm e}(Z_n)\to+\infty$
and $S_{\rm e}(Z_n)/F_{\rm e}(Z_n)\to4$ by \eqref{tildeARlimit}. Since $\lambda_n>0$ along that subsequence,
\eqref{tildecoercivetest} implies
$$\|u_n\|_2^2+\|v_n\|_2^2\le\frac C{\lambda_n}S_{\rm e}(Z_n)+o(1)\,.$$
Moreover, integrating of \eqref{tildepotentialbound} applied to both $t=u_n$ and $t=v_n$ yields
$$F_{\rm e}(Z_n)\leq\varepsilon S_{\rm e}(Z_n)+C_\varepsilon\left(\|u_n\|_2^2+\|v_n\|_2^2\right)\leq \left(\varepsilon+\frac{C_\varepsilon}{\lambda_n}\right)S_{\rm e}(Z_n)+o(1)\,,$$
which, for any choice of $\varepsilon<\frac14$, contradicts $S_{\rm e}(Z_n)\sim4F_{\rm e}(Z_n)$. As a result, $0\leq\lambda_n\leq C$ for some constant $C>0$, and this in turn implies that $F_{\rm e}(Z_n)$, $S_{\rm e}(Z_n)$, and $G_b(Z_n)$ are bounded by \eqref{tildeARlimit} and \eqref{tildecoercivetest}.

Suppose next that $\lambda_n\to0$, possibly along a subsequence. Then $F_{\rm e}(Z_n)\to0$, $A(Z_n)\to c$ and $S_{\rm e}(Z_n)\to2c$ by \eqref{tildeARlimit}. Since $\lambda_n\geq0$, by
\eqref{tildecrossbound} we deduce $\displaystyle\limsup_{n\to+\infty} G_b(Z_n)\le2(1+\varepsilon)c$ for every choice of $\varepsilon>0$.
Since
$$G_b(Z_n)=b^{-2}\|\nabla u_n\|_2^2+b^2\|\nabla v_n\|_2^2\ge 2\|\nabla u_n\|_2\|\nabla v_n\|_2=2A(Z_n)\,,$$
which holds thanks to the elementary estimate $b^{-2}x+b^2y\ge 2\sqrt{xy}$ for $x,y\ge0$, it follows that
$$G_b(Z_n)\to2c\quad\mbox{and}\quad\Big\|\nabla\Big(\frac{u_n}{b}-bv_n\Big)\Big\|_2^2=G_b(Z_n)-2A(Z_n)\to0\,.$$
Combining these with the reversed triangle inequality and the definition of $G_b$, we deduce
$$\|\nabla u_n\|_2^2\longrightarrow b^2c\,,\qquad \|\nabla v_n\|_2^2\longrightarrow b^{-2}c\,.$$
Recalling that $b^2=\sqrt{\vartheta_1/\vartheta_2}$ and $b^{-2}=\sqrt{\vartheta_2/\vartheta_1}$ by \eqref{tildeweights}, we conclude that
$$\vartheta_2\|\nabla u_n\|_2^2\longrightarrow\vartheta_0c<4\pi\,,\qquad \vartheta_1\|\nabla v_n\|_2^2\longrightarrow\vartheta_0c<4\pi\,,$$
since $c\in\big(0,\frac{4\pi}{\vartheta_0}\big)$ by assumption. Since $F_{\rm e}(Z_n)\to0$, we have also $\|(u_n)_+\|_4+\|(v_n)_+\|_4\to0$.
Lemma~\ref{tildequarticvanishing} then gives $S_{\rm e}(Z_n)\to0$, a contradiction. This proves a positive lower bound for $\lambda_n$, and \eqref{tildecoercivetest} then also bounds
$M_b(Z_n)$. We have therefore proved
\begin{equation}\label{tildeuniformbounds}
	0<\lambda_-\le\lambda_n\le\lambda_+<\infty\,,\qquad\|u_n\|_{H^1}+\|v_n\|_{H^1}\le C\,.
\end{equation}
Hence, we can extract a subsequence with
\begin{equation}\label{ProofThm1.15convergences}
	Z_n\rightharpoonup Z=(u,v)\in H^1_{\rm rad}(\R^2)^2\quad\mbox{and}\quad\lambda_n\to\lambda>0\,.
\end{equation}
Radial compactness gives strong $L^p$ convergence for every $p>2$ and almost-everywhere convergence. Boundedness of $S_{\rm e}(Z_n)$ implies then 
\begin{equation}\label{tildeweaknonlinearlimits}
 \begin{gathered}
 H_{\vartheta_2}(u_n)\to H_{\vartheta_2}(u),\quad
 H_{\vartheta_1}(v_n)\to H_{\vartheta_1}(v)\quad\text{in }L^1(\R^2)\,,\\
 h_{\vartheta_2}(u_n)\to h_{\vartheta_2}(u),\quad h_{\vartheta_1}(v_n)\to h_{\vartheta_1}(v)\quad\text{in }L^1_{\rm loc}\,.
 \end{gathered}
\end{equation}
We give the details for \(u_n\), the argument for \(v_n\) being identical. We first note that $H_\vartheta(t)=o(t h_\vartheta(t))$ as $t\to+\infty$. Thus, for \(R>0\),
$$\int_{\{u_n>R\}}H_\vartheta(u_n)\le\varepsilon_R\int u_nh_\vartheta(u_n)\le S_{\rm e}(Z_n)\leq C\varepsilon_R\,,$$
where \(\varepsilon_R\to0\) as \(R\to+\infty\). On the other hand, on \([0,R]\), \(H_\vartheta\) is Lipschitz with \(H_\vartheta'(t)=h_\vartheta(t)\le C_{\vartheta,R}t^3\). Hence strong convergence in \(L^4\) implies convergence in \(L^1\) of the truncated potentials. This proves \(H_\vartheta(u_n)\to H_\vartheta(u)\) in \(L^1(\mathbb R^2)\). Moreover, we have
$$\int_{\{u_n>R\}}h_\vartheta(u_n)\le\frac1R\int u_nh_\vartheta(u_n)\le\frac CR\,.$$
On any bounded set \(\Omega\), the truncated functions \(h_\vartheta(\min\{u_n,R\})\) converge a.e. to \(h_\vartheta(\min\{u,R\})\) and are uniformly bounded by \(R(e^{\vartheta R^2}-1)\). Dominated convergence therefore gives convergence in \(L^1(\Omega)\), while the preceding tail estimate allows us to let $R\to+\infty$. Thus \(h_\vartheta(u_n)\to h_\vartheta(u)\) in \(L^1_{\rm loc}(\mathbb R^2)\), completing the proof of \eqref{tildeweaknonlinearlimits}.

We now test \eqref{tildeCPconditions} with $(\varphi,0)$, where \(\varphi\in C^\infty_{c,\mathrm{rad}}(\R^2)\), and with $(0,\psi)$, with $\psi\in\bigcup_m W_m$, where these pairs are understood in the coordinates $(z,w)$ of \eqref{tildecoordinates}. Since the spaces $(W_n)_n$ are nested, each fixed test function belongs to the corresponding Galerkin space for all sufficiently large $n$. Thus,
$$\big(J_{\rm e}'(Z_n)+\lambda B'(Z_n)\big)[(\varphi,0)]=0,\qquad \big(J_{\rm e}'(Z_n)+\lambda_n B'(Z_n)\big)[(0,\psi)]=0\,.$$
Thanks to \eqref{ProofThm1.15convergences} and \eqref{tildeweaknonlinearlimits} we can pass to the limit and get
\begin{equation}\label{ProofThm1.15limiteq}
	\big(J_{\rm e}'(Z)+\lambda B'(Z)\big)[(\varphi,0)]=0,\qquad \big(J_{\rm e}'(Z)+\lambda B'(Z)\big)[(0,\psi)]=0\,.
\end{equation}
Since \(h_{\vartheta_2}(u),h_{\vartheta_1}(v)\in L^2(\R^2)\) by the Trudinger-Moser inequality, 
all terms in \eqref{ProofThm1.15limiteq} therefore define continuous linear functionals on \(H^1\). By density of \(C^\infty_{c,\mathrm{rad}}(\R^2)\) in the first coordinate and of \(\bigcup_m W_m\) in the second coordinate, the equations extend to all radial $H^1$ test functions. Finally, since $u$ and $v$ are radial, testing against an arbitrary smooth compactly supported function is equivalent to testing against its spherical average. Thus \eqref{ProofThm1.15limiteq} holds for arbitrary test functions, that is, $(u,v)$ solve the normalised exponential system \eqref{TMpuresystem}, hence $u,v\ge0$ by the comparison principle.
Elliptic regularity gives $u,v\in H^2\cap L^\infty$, followed by classical regularity and decay
to zero. Exterior comparison, using $h_\vartheta(t)=o(t)$ at zero,
gives exponential decay as in Lemma \ref{profileregularity}. 

Test the equations with $v$ and $u$, respectively and sum the two expressions; with our notation, this and the Poho\v zaev identity (see e.g. \cite{VdV}) read as
\begin{equation}\label{der+Pohozaev}
	S_{\rm e}(Z)=2A(Z)+2\lambda B(Z)\,,\qquad \lambda B(Z)=F_{\rm e}(Z)\,,
\end{equation}
respectively. It follows that
$$J_{\rm e}(Z)=A(Z)-F_{\rm e}(Z)=\frac12S_{\rm e}(Z)-2F_{\rm e}(Z)\geq0\,,$$
where the last inequality follows from \(t h_\vartheta(t)\ge4H_\vartheta(t)\) for \(t\ge0\). Moreover, \eqref{tildeARlimit} and \eqref{tildeweaknonlinearlimits} imply that \(F_{\rm e}(Z)=\lambda a\). Since \(\lambda>0\), the second identity in \eqref{der+Pohozaev} gives $B(Z)=a$. In particular, \(Z\) is nontrivial. The right-hand sides of the limiting system are nonnegative; hence, by the strong maximum principle, \(u,v>0\) in $\R^2$. 

To conclude the proof, it remains to prove the strong convergence $Z_n\to Z$ in $H^1(\R^2)^2$.
Define
\[
 \begin{aligned}
 D_M&=\lim_{n\to+\infty}(b^{-2}\|u_n-u\|_2^2+b^2\|v_n-v\|_2^2)\,,&
 D_A&=\lim_{n\to+\infty}(A(Z_n)-A(Z))=c-J_{\rm e}(Z)\,.
 \end{aligned}
\]
The testing identities and $B(Z)=a$ show that
$\displaystyle\lim_{n\to+\infty} S_{\rm e}(Z_n)-S_{\rm e}(Z)=2D_A$; Fatou's lemma then gives $D_A\ge0$. We claim
\begin{equation}\label{tildeconcentrationcross}
	\limsup_{n\to+\infty}\left(\int\mathcal C_b(u_n,v_n)-\int\mathcal C_b(u,v)\right)\le2D_A\,.
\end{equation}
To see this, use a continuous cutoff in
$\max\{|u_n|,|v_n|\}$, equal to zero below $R$ and one above $2R$.
The integrals of the complementary truncated cross term and of the
truncated $S$ integrand converge by strong $L^4$ convergence; on bounded
amplitudes these integrands are bounded by
$C_R(|u_n|+|v_n|)^4$. On the high-amplitude part,
\eqref{tildecrossbound} applies. Strong $L^1$ convergence of the
potentials in \eqref{tildeweaknonlinearlimits} makes its
$C_\varepsilon F$ contribution tend to zero as $R\to+\infty$,
uniformly in $n$. Subtract the limiting cross term, let $n\to+\infty$,
then $R\to+\infty$, and finally $\varepsilon\to0$.
This proves \eqref{tildeconcentrationcross}. Test now the limit equations \eqref{TMpuresystem} with $\cT Z$ and get
$$G_b(Z)+\lambda M_b(Z)=\int_\R\mathcal C_b(u,v)\,.$$
Subtracting this identity from \eqref{tildecoercivetest} and using \eqref{tildeconcentrationcross} one finds
\begin{equation}\label{da_qui}
	b^{-2}\lim_{n\to+\infty}\|\nabla(u_n-u)\|_2^2+b^2\lim_{n\to+\infty}\|\nabla(v_n-v)\|_2^2+\lambda D_M\le2D_A\,.
\end{equation}
On the other hand, by weak convergence and using the elementary arithmetic--geometric mean inequality $\sqrt{st}\leq\frac12(s+t)$ one obtains
$$|D_A|=\lim_{n\to\infty}
\left|\int\nabla(u_n-u)\nabla(v_n-v)\right|\leq\lim_{n\to\infty}\left(\frac1{2b^2}\|\nabla(u_n-u)\|_2^2+\frac{b^2}2\|\nabla(v_n-v)\|_2^2\right).$$
Combining the two inequalities, one gets
\begin{equation*}
	D_M=0\,,\quad 0\le D_A\le c<\frac{4\pi}{\vartheta_0}\,,\quad D_A=b^{-2}\!\lim_{n\to+\infty}\|\nabla(u_n-u)\|_2^2+b^2\!\lim_{n\to+\infty}\|\nabla(v_n-v)\|_2^2\,.
\end{equation*}
The first identity implies that $Z_N\to Z$ strongly in $L^2(\R^2)^2$. Moreover, from the equality case in the arithmetic--geometric mean inequality one further gets
\begin{equation}\label{tildedefectsynchronisation}
	D_A=b^{-2}\!\lim_{n\to+\infty}\|\nabla(u_n-u)\|_2^2=b^2\!\lim_{n\to+\infty}\|\nabla(v_n-v)\|_2^2\,,
\end{equation}
that is, weighted gradient residuals coincide 
after balancing the exponential coefficients. 
In terms of $\vartheta_1,\,\vartheta_2$ this reads as
\[
 \lim_{n\to+\infty}\vartheta_2\|\nabla(u_n-u)\|_2^2
 =\lim_{n\to+\infty}\vartheta_1\|\nabla(v_n-v)\|_2^2
 =\vartheta_0D_A<4\pi\,.
\]
Choose $q>1$ and $\eta>0$ so that both inequalities remain strict
after multiplication by $q(1+\eta)$. In this way
$$\sup_n\int_\R\left(e^{q\vartheta_2u_n^2}-1+e^{q\vartheta_1v_n^2}-1\right)\dd x<+\infty\,.$$
Here we used the inequality $u_n^2\le(1+\eta)(u_n-u)^2+(1+\eta^{-1})u^2$ and the analogous for $v_n$; we applied Adachi--Tanaka inequality \cite{AdTa} to \(u_n-u\) and \(v_n-v\), while we recalled that \(u,v\in H^1(\R^2)\) have finite exponential integrals of every order. 
For $1<r<\min\{q,2\}$, absorbing polynomials by the exponential growth gives uniformly
vanishing high-amplitude tails of $h_{\vartheta_2}(u_n)$ and
$h_{\vartheta_1}(v_n)$ in $L^r$. On bounded amplitudes, strong
$L^{3r}$ convergence and the cubic behaviour at zero give convergence
in $L^r$. Thus $\int\mathcal C_b(Z_n)\to\int\mathcal C_b(Z)$ converge strongly in $L^r$ and in
$H^{-1}$. Passing to the limit in \eqref{tildecoercivetest} now yields
\[
 G_b(Z_n)+\lambda_nM_b(Z_n)\longrightarrow G_b(Z)+\lambda M_b(Z)\,.
\]
Since $\lambda>0$, weak convergence and convergence of this Hilbert
norm imply strong convergence in $H^1(\R^2)^2$. Energy and mass pass
to the limit, completing the proof.
\end{proof}

\begin{corollary}[Compactness at the level threshold away from zero frequency]\label{tildeendpoint}
	The conclusion of Theorem~\ref{introtildecompactness} remains valid at
	$c=4\pi/\vartheta_0$ if $\liminf_n\lambda_n>0$ is assumed.
	For sequences satisfying the other hypotheses at this level,
	failure of compactness can therefore occur only along a subsequence
	with $\lambda_n\to0$.
\end{corollary}
\begin{proof}
	The estimate from above for $\lambda_n$ in Theorem~\ref{introtildecompactness} uses
	only finiteness of $c$, so it still applies. If one assumes the lower bound $\liminf_n\lambda_n>0$, then testing \eqref{tildeCPconditions} as in the proof of Theorem~\ref{introtildecompactness} gives the uniform $H^1$ bounds in \eqref{tildeuniformbounds}. The rest of the proof now follows without modifications, 
	thus the weak limit $Z=(u,v)$ is a positive solution of \eqref{TMpuresystem} with $B(u,v)=a$ and $\lambda>0$. Observe that the strict inequality $t h_\vartheta(t)>4H_\vartheta(t)$ for $t>0$ gives
	$$J_{\rm e}(Z)=\tfrac12S_{\rm e}(Z)-2F_{\rm e}(Z)>0\,.$$
	Reasoning as in \eqref{da_qui}--\eqref{tildedefectsynchronisation} one finds $D_M=0$, \eqref{tildedefectsynchronisation} holds, and 
	$$D_A=\frac{4\pi}{\vartheta_0}-J_{\rm e}(Z)<\frac{4\pi}{\vartheta_0}\,.$$
	The remainder, rather than the total sequence, is strictly below
	the scalar exponential threshold. The final uniform-integrability
	argument of Theorem~\ref{introtildecompactness} applies unchanged and
	gives strong convergence. The last assertion follows by applying
	this conclusion to every subsequence on which the frequencies have
	a positive lower bound.
\end{proof}

\subsection{Existence for normalised solutions in the range $0<a<a_c$: Proof of Theorem~\ref{introfullfibreexistence}}
\label{secfullfibre}
The finite-dimensional linking construction can be completed at the
level of the full negative space. The useful change is to incorporate
mass normalisation into the functional before maximising over a fibre.
In the plane, ordinary spatial dilation leaves the Dirichlet pairing
unchanged. This gives a smooth quotient functional even though the
dilation action itself need not be differentiable on $H^1$.
The reduction is related to the generalised Nehari method
\cite[Chapter 4]{SzulkinWeth2010}; the mass quotient and the curvature
estimate needed here are established below.

On the open set \(\{Z\in H\times H:B(Z)>0\}\), define
\begin{equation}\label{massquotient}
	\Phi_a(Z)=A(Z)-\frac a{B(Z)}F_{\rm e}(Z),
	\qquad
	\tau_Z=\sqrt{\frac{B(Z)}a},
	\qquad
	\mathcal N_aZ(x)=Z(\tau_Zx).
\end{equation}
By the scaling properties in dimension two,
\[
B(\mathcal N_aZ)=\tau_Z^{-2}B(Z)=a,\qquad
A(\mathcal N_aZ)=A(Z),\qquad
F_{\rm e}(\mathcal N_aZ)=\tau_Z^{-2}F_{\rm e}(Z).
\]
Consequently, \(\Phi_a(Z)=J_{\rm e}(\mathcal N_aZ)\). Moreover, \(\Phi_a\) is of class \(C^2\) on its domain. Indeed, the \(C^1\)-regularity of the maps \(h_\vartheta:H^1(\R^2)\to L^2(\R^2)\) established above implies that \(F_{\rm e}\) is \(C^2\), since
\[
F_{\rm e}'(u,v)[\varphi,\psi]=\int_{\R^2}\bigl(h_{\vartheta_2}(u)\varphi+h_{\vartheta_1}(v)\psi\bigr).
\]
The regularity of \(\Phi_a\) then follows from that of \(A\) and \(B\), since \(B\) is strictly positive on the domain.
Write
$$\mathcal L(z,w)=\left(\frac{b(z+w)}{\sqrt2},\,\frac{z-w}{b\sqrt2}\right)$$
and set
\[
 \mathbb S_+=\{\zeta\in H:\|\zeta\|_2=1,\ \zeta_+\not\equiv0\}\,,
 \qquad\mathcal K_\zeta=\{\mathcal L(r\zeta,w):r>0,\ w\in H,\ \|w\|_2<r\}\,.
\]
Note that for every $\zeta\in\mathbb S_+$, the set $\mathcal K_\zeta$ contains all pairs of the form $(0,w)$ for arbitrary $w\in H$. Recalling \eqref{tildechart}, its normalised image is
\begin{equation}\label{fullfibreimage}
	\mathcal N_a(\mathcal K_\zeta)=\{\mathcal X_a(\zeta_t,w):t>0,\ w\in H\}\,,\qquad \zeta_t(x)=t\zeta(tx)\,.
\end{equation}
To show this, observe that $B(\mathcal L(r\zeta,w))=\frac12(r^2-\|w\|_2^2)$. Given $Z=\mathcal L(r\zeta,w)\in\mathcal K_\zeta$, set $t=\tau_Z$. Since $B(\mathcal N_a Z)=a$, the scaling properties of the $L^2$ norm give
$$
\mathcal N_a Z
=\mathcal L\big(r\zeta(t\cdot),w(t\cdot)\big)
=\mathcal L\left(\frac rt\zeta_t,w(t\,\cdot)\right),
\qquad
\left(\frac rt\right)^2
=2a+\|w(t\cdot)\|_2^2.
$$
As $r/t>0$, the last identity yields $r/t=r_a(w(t\cdot))$, and hence $\mathcal N_a Z=\mathcal X_a(\zeta_t,w(t\cdot))$. Conversely, given any $t>0$ and $w\in H$, consider
$$
Z=\mathcal L\big(t r_a(w)\zeta,w(\cdot/t)\big).
$$
Since $\|w(\cdot/t)\|_2=t\|w\|_2<t r_a(w)$, we have $Z\in\mathcal K_\zeta$. Moreover, $B(Z)=at^2$, so that $\tau_Z=t$, and therefore
$$
\mathcal N_a Z
=\mathcal L\big(t r_a(w)\zeta(t\cdot),w\big)
=\mathcal X_a(\zeta_t,w).
$$
Thus the normalised image of $\mathcal K_\zeta$ is precisely \eqref{fullfibreimage}.

\begin{lemma}[A quartic barrier up to the cubic limiting mass]\label{cubicbarrier}
	Set $m_0=\|Q_4\|_2^2$, $\vartheta_0=\sqrt{\vartheta_1\vartheta_2}$, and
	\[
	 a_{\rm c}=m_0/\vartheta_0\,,\quad b_{\rm c}=(\vartheta_1/\vartheta_2)^{1/8}\,,\quad
	 \alpha_u=\vartheta_2b_{\rm c}^2\,,\quad
	 \alpha_v=\vartheta_1b_{\rm c}^{-2}.
	\]
	If $\|h\|_2^2=a$ and $0<\|\nabla h\|_2^2\le\frac{2\pi}{\max\{\alpha_u,\alpha_v\}}$, then
	\begin{equation}\label{quarticbarrierestimate}
		J_{\rm e}(b_{\rm c}h,b_{\rm c}^{-1}h)\ge \left(1-\frac a{a_{\rm c}}-C_*a\|\nabla h\|_2^2\right)\|\nabla h\|_2^2\,,
	\end{equation}
	where $C_{\rm R}$ is defined in \ref{Rufconstant} and
	$$C_* =\frac{C_{\rm R}}{16\pi^3}\left(\frac{\alpha_u^3}{\vartheta_2}+\frac{\alpha_v^3}{\vartheta_1}\right)>0\,.$$
	Consequently, for every $0<a<a_{\rm c}$ there exist $R_a,\beta_a>0$
	such that each full fibre $\mathcal N_a(\mathcal K_\zeta)$,
	$\zeta\in\mathbb S_+$, contains a pair of energy at least $\beta_a$.
	One may take
	\begin{equation}\label{cubicbarrierconstants}
		R_a=\min\left\{\frac{2\pi}{\max\{\alpha_u,\alpha_v\}},\frac{1-a/a_{\rm c}}{2C_*a}\right\},\qquad\beta_a=\frac12\left(1-\frac a{a_{\rm c}}\right)R_a\,.
	\end{equation}
\end{lemma}
\begin{proof}
	Recall that
	$H_\vartheta(t)=\frac{\vartheta}{4}t_+^4+\frac1{2\vartheta}R_3(\vartheta t_+^2),$ where $R_3(s)=e^s-1-s-\frac12s^2$. The scalar remainder estimate \eqref{introscalarquartic}
	proved in Corollary~\ref{scalartruncatedrefinement}, applied to
	$h_+$ gives
	\begin{equation}\label{sixthremainder}
		\int_{\R^2}R_3(\alpha h_+^2)\le\frac{C_{\rm R}\alpha^3}{8\pi^3}\|h_+\|_2^2\|\nabla h_+\|_2^4\le\frac{C_{\rm R}\alpha^3}{8\pi^3}a\|\nabla h\|_2^4,
	\end{equation}
	provided $\alpha\|\nabla h\|_2^2\le2\pi$. Here we used $\|h_+\|_2^2\le a$ and $\|\nabla h_+\|_2^2\le\|\nabla h\|_2^2$. Since $\vartheta_2b_{\rm c}^4
	=\vartheta_1b_{\rm c}^{-4}=\vartheta_0$ and
	$$F_{\rm e}(b_{\rm c}h,b_{\rm c}^{-1}h)=\frac{\vartheta_0}{2}\int_{\R^2}h_+^4+\frac1{2\vartheta_2}\int_{\R^2}R_3(\alpha_u h_+^2)+\frac1{2\vartheta_1}\int_{\R^2}R_3(\alpha_v h_+^2)\,,$$
	if $\max\{\alpha_u,\alpha_v\}\|\nabla h\|_2^2\le2\pi$, \eqref{sixthremainder} yields
	$$F_{\rm e}(b_{\rm c}h,b_{\rm c}^{-1}h)\le\frac{\vartheta_0}{2}\int_{\R^2}h_+^4+C_*a\|\nabla h\|_2^4\,.$$
	The sharp planar cubic Gagliardo--Nirenberg inequality \cite{Weinstein} gives
	$$\int_{\R^2}h_+^4\le\frac{2}{m_0}\|h_+\|_2^2\|\nabla h_+\|_2^2\le\frac{2a}{m_0}\|\nabla h\|_2^2\,.$$
	Thus
	$$F_{\rm e}(b_{\rm c}h,b_{\rm c}^{-1}h)\le\frac{a}{a_{\rm c}}\|\nabla h\|_2^2+C_*a\|\nabla h\|_2^4\,.$$
	Since $A(b_{\rm c}h,b_{\rm c}^{-1}h)=\|\nabla h\|_2^2$,subtracting this estimate from the quadratic term proves \eqref{quarticbarrierestimate}.
	We now show that every full fibre contains a point with uniformly positive energy. Fix $\zeta\in\mathbb S_+$ and choose $R_a$ and $\beta_a$ as in \eqref{cubicbarrierconstants}. 	Since $\|\zeta_t\|_2=1$ and $\|\nabla\zeta_t\|_2^2=t^2\|\nabla\zeta\|_2^2$, we can choose $t>0$ so that, for $h=\sqrt a\,\zeta_t$, one has $\|h\|_2^2=a$ and $\|\nabla h\|_2^2=R_a$. The choice of $R_a$ ensures that the hypotheses of \eqref{quarticbarrierestimate} are satisfied and that $C_*aR_a\le\frac12(1-a/a_{\rm c})$. Therefore
	$$J_{\rm e}(b_{\rm c}h,b_{\rm c}^{-1}h)\ge\left(1-\frac a{a_{\rm c}}-C_*aR_a\right)R_a\ge\beta_a\,.$$
	It remains to verify that this pair belongs to the full fibre $\mathcal N_a(\mathcal K_\zeta)$. Recalling \eqref{tildecoordinates}, 
	for $(u,v)=(b_{\rm c}h,b_{\rm c}^{-1}h)$, one has
	$$z=\frac{b_{\rm c}/b+b/b_{\rm c}}{\sqrt2}\,h\,,\qquad w=\frac{b_{\rm c}/b-b/b_{\rm c}}{\sqrt2}\,h\,.$$
	In particular, $z$ is a positive multiple of $\zeta_t$. Moreover, since $B(u,v)=a$, the identity
	$$B(\mathcal L(z,w))=\frac12\big(\|z\|_2^2-\|w\|_2^2\big)$$
	implies
	$$\|z\|_2^2=2a+\|w\|_2^2\,.$$
	As $\|\zeta_t\|_2=1$ and $z$ is a positive multiple of $\zeta_t$, this means that $z=r_a(w)\zeta_t$. Hence
	$$(b_{\rm c}h,b_{\rm c}^{-1}h)=\mathcal X_a(\zeta_t,w)\in\mathcal N_a(\mathcal K_\zeta)\,,$$
	where we used \eqref{fullfibreimage}. We have thus found, in every full fibre, a point with energy at least $\beta_a$, as claimed.
%
\end{proof}

\begin{lemma}\label{expcurvature}
	For $t>0$ and $\vartheta>0$,
	\begin{equation}\label{expcurvaturepointwise}
		h_\vartheta'(t)>\frac{h_\vartheta(t)}t+\frac{h_\vartheta(t)^2}{2H_\vartheta(t)}\,.
	\end{equation}
	Consequently, for any $Z=(u,v)$ with $F_{\rm e}(Z)>0$ one has
	\begin{equation}\label{expcurvatureintegral}
		I(Z):=\int\left[\frac{h_{\vartheta_2}(u)^2}{h_{\vartheta_2}'(u)}+\frac{h_{\vartheta_1}(v)^2}{h_{\vartheta_1}'(v)}\right]<\frac{2F_{\rm e}(Z)S_{\rm e}(Z)}{S_{\rm e}(Z)+2F_{\rm e}(Z)}\,.
	\end{equation}
	The ratios are defined to be zero on the nonpositive half-line.
	For every variation $V$ one has
	\[
	 F_{\rm e}''(Z)[V,V]\ge F_{\rm e}'(Z)[V]^2/I(Z).
	\]
\end{lemma}
\begin{proof}
	With $x=\vartheta t^2$, direct calculation gives
	\[
	 \left(h_\vartheta'(t)-\frac{h_\vartheta(t)}t\right)H_\vartheta(t)
	                   -\frac12h_\vartheta(t)^2
	       =\frac{x e^x}{\vartheta}(\sinh x-x)>0\,.
	\]
	Thus $h^2/h'<2H(th)/(th+2H)$. Observe that the function
	$(s,t)\mapsto st/(s+t)$ is concave and homogeneous of degree one
	on the nonnegative quadrant. Integrating over the two components
	therefore gives \eqref{expcurvatureintegral}, with strictness on the
	set where a component is positive. The final assertion is the
	weighted Cauchy--Schwarz inequality.
\end{proof}

\begin{proposition}[A unique full-fibre maximum]\label{fullfibremaximum}
	Assume $0<a<a_{\rm c}$ and recall the definition of $\beta_a$ in \eqref{cubicbarrierconstants}. For every $\zeta\in\mathbb S_+$,
	$\Phi_a|_{\mathcal K_\zeta}$ has a unique critical point $X_\zeta=\mathcal L(r_\zeta\zeta,w_\zeta)$, a global maximum. Its fibre Hessian is negative definite and
	coercive, and $(r_\zeta,w_\zeta)$ depends $C^1$ on $\zeta$. In particular
	\begin{equation}\label{fullreducedenergy}
		m_a(\zeta):=\Phi_a(X_\zeta)=\max_{t>0,\,w\in H}J_{\rm e}(\mathcal X_a(\zeta_t,w))\ge\beta_a
	\end{equation}
	and the map $\zeta\mapsto m_a(\zeta)$ is of class $C^1$. Moreover,
	\begin{equation}\label{exactfibreequations}
		\begin{gathered}
			B(\mathcal N_aX_\zeta)=a,\qquad P_{\rm e}(\mathcal N_aX_\zeta)=0\,,\qquad
			\lambda_\zeta:=\frac{F_{\rm e}(\mathcal N_aX_\zeta)}a>0\,,\\
			(J_{\rm e}'(\mathcal N_aX_\zeta)+\lambda_\zeta B'(\mathcal N_aX_\zeta))[V]=0
			\quad\text{for }V\in\operatorname{span}\{\mathcal N_aX_\zeta\}+\mathcal L(0,H)\,.
		\end{gathered}
	\end{equation}
\end{proposition}

\begin{proof}
	First, we investigate geometry and compactness on one fibre. Fix $\zeta\in\mathbb S_+$. In \eqref{fullfibreimage}, write
	$R=(2a+\|w\|_2^2)^{1/2}$ for the positive amplitude.
	The convexity estimate \eqref{tildejensen}, using the fourth and
	sixth powers in the exponential series, gives 
	\begin{equation}\label{fibrescoercivity}
		J_{\rm e}(\mathcal X_a(\zeta_t,w))\le t^2\left(\frac{\|\nabla\zeta\|_2^2R^2}{2}-\int\zeta_+^4R^4\right)-\int\zeta_+^6t^4R^6-\frac12\|\nabla w\|_2^2\,.
	\end{equation}
	The right-hand side is at most $Ct^2$, and it is nonpositive for all sufficiently large $t$, since $R^2\ge2a$. Thus a superlevel $\Phi_a\ge d>0$ has
	\[
	 0<t_d\le t\le T_d\,,\qquad R+\|w\|_{H^1}\le C_d\,.
	\]
	In the original variables $X=\mathcal L(r\zeta,\omega)$ these bounds
	give $r+\|\omega\|_{H^1}\le C_d$ and
	$0<c_d\le B(X)=at^2\le C_d$.
	These superlevels are complete closed subsets of the fibre coordinate
	space, separated from the boundary of its cone. Lemma~\ref{cubicbarrier}
	supplies a point with value at least $\beta_a$ on this full fibre; the supremum
	is finite by \eqref{fibrescoercivity}.
	
	We verify the Palais--Smale condition on every positive level of this
	fixed fibre. Let $X_j=\mathcal L(r_j\zeta,\omega_j)$ be such a
	sequence. The preceding bounds apply. Set
	\[
	 k_j=\frac a{B(X_j)}\quad\mbox{and}\quad\Lambda_j=\frac{F_{\rm e}(X_j)}{B(X_j)}\,.
	\]
	Both $k_j$ and $\Lambda_j$ are bounded above and away from zero.
	Indeed, its upper bound follows from $k_jF_{\rm e}(X_j)=A(X_j)-\Phi_a(X_j)$, and its lower bound follows
	from \eqref{tildejensen} and the positive lower bound for $r_j$.
	After passing to a subsequence, we may assume that $r_j\to r>0$, $k_j\to k>0$, and $\Lambda_j\to\Lambda>0$. For each $j$, the derivative of $\Phi_a$ with respect to the negative coordinate $\omega$, evaluated at $(r_j,\omega_j)$, is precisely $-E_j'(\omega_j)$, where
	$$E_j(\omega)=\frac12\|\nabla\omega\|_2^2+\frac{k_j\Lambda_j}{2}\|\omega\|_2^2+k_jF_{\rm e}(\mathcal L(r_j\zeta,\omega)),.$$
	Since $F_{\rm e}$ is convex, the functionals $E_j$ are uniformly strongly convex on $H$: indeed, the quadratic part has a uniform coercivity constant because $k_j\Lambda_j$ stays bounded away from zero. Moreover, $E_j$ converges to
	$$E(\omega)=\frac12\|\nabla\omega\|_2^2+\frac{k\Lambda}{2}\|\omega\|_2^2+kF_{\rm e}(\mathcal L(r\zeta,\omega))\,,$$
	which is coercive, convex and weakly lower semicontinuous, and therefore admits a unique minimiser $\omega_*$, for which $E'(\omega_*)=0$. By continuity of the nonlinear derivative with respect to the parameters, we have $E_j'(\omega_*)\to E'(\omega_*)=0$. On the other hand, the Palais--Smale condition gives $E_j'(\omega_j)\to0$. Uniform strong monotonicity now yields
	\begin{equation*}
		c\|\omega_j-\omega_*\|_{H^1}^2\le\langle E_j'(\omega_j)-E_j'(\omega_*),\omega_j-\omega_*\rangle\to0
	\end{equation*}
	for some $c>0$ independent of $j$. Hence $\omega_j\to\omega_*$ strongly in $H^1$, proving the Palais--Smale condition on the fixed fibre.
	Let $M=\sup_{\mathcal K_\zeta}\Phi_a$, which is finite and satisfies $M\ge\beta_a>0$. Fix $d\in(0,M)$ and consider the superlevel set $\mathcal A_d=\{X\in\mathcal K_\zeta:\Phi_a(X)\ge d\}$,
	which is complete with respect to the metric induced by the product norm of $\mathbb R\times H$ as previously shown. 
	
	Let $(X_n)_n\subset\mathcal A_d$ be a maximizing sequence. Applying Ekeland's variational principle to $-\Phi_a$ on $\mathcal A_d$, 
	we obtain a sequence $(Y_n)_n\subset\mathcal A_d$ such that
	$$\Phi_a(Y_n)\ge\Phi_a(X_n)\to M\quad\mbox{and}\quad\|\Phi_a'|_{\mathcal K_\zeta}(Y_n)\|_{(\mathbb R\times H)^*}\to0\,.$$
	Indeed, since $M>d$, the points $Y_n$ eventually belong to the interior of $\mathcal A_d$ relative to the fibre, so the variational inequality provided by Ekeland's principle yields the full fibre derivative estimate. Thus $(Y_n)_n$ is a Palais--Smale sequence at level $M>0$. By the Palais--Smale condition established above, a subsequence converges strongly in the fibre coordinates to some $X_\zeta\in\mathcal K_\zeta$. By continuity of $\Phi_a$, we conclude that $\Phi_a(X_\zeta)=M$, and hence the supremum is attained.


	We show next that every fibre critical point is a nondegenerate strict maximum. Let $X$ be any fibre critical point and set
	$$k=\frac a{B(X)}\quad\mbox{and}\quad\Lambda=\frac{F_{\rm e}(X)}{B(X)}\,.$$
	Since $X$ is an admissible fibre variation, $\Phi_a'(X)[X]=0$ with $A'(X)[X]=2A(X)$, $B'(X)[X]=2B(X)$ and $F_{\rm e}'(X)[X]=S_{\rm e}(X)$, yields
	$$2A(X)-kS_{\rm e}(X)+2kF_{\rm e}(X)=0\,.$$
	It follows that
	$$\Phi_a(X)=A(X)-kF_{\rm e}(X)=k\left(\frac{S_{\rm e}(X)}2-2F_{\rm e}(X)\right)>0\,.$$
	Indeed, $S_{\rm e}(X)>4F_{\rm e}(X)>0$, since $t h_\vartheta(t)>4H_\vartheta(t)$ for every $t>0$. We now consider the quadratic form $Q(Z)=A(Z)+k\Lambda B(Z)$, keeping $k$ and $\Lambda$ fixed when differentiating. Writing a fibre variation as $V=\mathcal L(s\zeta,\eta)$, we have
	$$Q(V)=\frac12\left[s^2\left(\|\nabla\zeta\|_2^2+k\Lambda\right)-\big(\|\nabla\eta\|_2^2-k\Lambda\|\eta\|_2^2\big)\right].$$
	Thus $Q$ has one positive direction, corresponding to the scalar variable $s$, and a negative coercive part in the $\eta$-variable, since $k\Lambda>0$. Moreover, the criticality identity gives
	$$Q(X)=A(X)+kF_{\rm e}(X)=\frac k2S_{\rm e}(X),\qquad Q'(X)[V]=A'(X)[V]+k\Lambda B'(X)[V]=kF_{\rm e}'(X)[V]$$
	for every fibre variation $V$.
		For any fibre variation $V$, set $V^\perp=V-\frac{F_{\rm e}'(X)[V]}{S_{\rm e}(X)}X$. The identities $Q''(X)[X,X]=kS_{\rm e}(X)>0$ and $Q''(X)[X,V]=kF_{\rm e}'(X)[V]$ imply that $Q''(X)[X,V^\perp]=0$. Thus $V^\perp$ belongs to the $Q''(X)$-orthogonal complement of $X$, and therefore $Q''(X)[V^\perp,V^\perp]\le0$. Expanding this inequality gives
	$$Q''(X)[V,V]=\frac{kF_{\rm e}'(X)[V]^2}{S_{\rm e}(X)}-\mathcal D_Q(V)\,,$$
	where
	$$\mathcal D_Q(V):=-Q''(X)\left[V-\frac{F_{\rm e}'(X)[V]}{S_{\rm e}(X)}X,\;V-\frac{F_{\rm e}'(X)[V]}{S_{\rm e}(X)}X\right]\ge0\,.$$
	In particular, $\mathcal D_Q$ vanishes precisely on $\mathbb R X$ and is coercive on the $Q$-orthogonal complement of $X$. Differentiating $\Phi_a$ twice and using $\Phi_a'(X)[V]=0$ for every $V\in\mathcal L(\mathbb R\zeta,H)$, we obtain
	$$\Phi_a''(X)[V,V]=Q''[V,V]-kF_{\rm e}''(X)[V,V]+\frac{2k}{B(X)}B'(X)[V]\big(F_{\rm e}'(X)[V]-\Lambda B'(X)[V]\big).$$
	By Lemma~\ref{expcurvature}, $F_{\rm e}''(X)[V,V]\ge F_{\rm e}'(X)[V]^2/I(X)$. Combining this estimate with the identity for $Q''[V,V]$ and completing the square in the last term, we obtain
	\begin{equation}\label{fibreHessian}
		\begin{split}
			\Phi_a''(X)[V,V]\le&-\mathcal D_Q(V)-k\left(\frac1{I(X)}-\frac1{S_{\rm e}(X)}-\frac1{2F_{\rm e}(X)}\right)F_{\rm e}'(X)[V]^2\\
			&-\frac{2k\Lambda}{B(X)}\left(B'(X)[V]-\frac{F_{\rm e}'(X)[V]}{2\Lambda}\right)^2.
		\end{split}
	\end{equation}
	The coefficient in the middle term is strictly positive by Lemma~\ref{expcurvature}. Moreover, the first two terms on the right-hand side control the full fibre norm: indeed, the second term controls the component parallel to $X$, since $F_{\rm e}'(X)[X]=S_{\rm e}(X)>0$, while $\mathcal D_Q$ is coercive on the $Q$-orthogonal complement of $X$. Consequently, the fibre Hessian is negative definite and coercive.


	We now turn to uniqueness. The positive set of $\Phi_a$ on the fibre is path connected. To see this, write $X=\mathcal L(r\zeta,r\eta)$, with $r>0$ and $\|\eta\|_2<1$, and denote by $F_j(\zeta,\eta)$ the homogeneous term of degree $2j$ in the expansion of $F_{\rm e}(\mathcal L(\zeta,\eta))$. Setting $A_\eta=\frac12(\|\nabla\zeta\|_2^2-\|\nabla\eta\|_2^2)$ and $B_\eta=\frac12(1-\|\eta\|_2^2)$, we obtain
	$$\Phi_a(\mathcal L(r\zeta,r\eta))=r^2\left(A_\eta-a\frac{F_2(\zeta,\eta)}{B_\eta}\right)-\frac{a}{B_\eta}\sum_{j\ge3}r^{2j-2}F_j(\zeta,\eta).$$
	Since the higher-order terms are nonnegative, the set of positive amplitudes is a nonempty interval with left endpoint zero precisely when $A_\eta,\,B_\eta>0$ and $\sqrt{A_\eta B_\eta}>\sqrt{aF_2(\zeta,\eta)}$. The set of $\eta$ satisfying this inequality is convex: the map $\eta\mapsto\sqrt{A_\eta B_\eta}$ is concave, whereas $\eta\mapsto\sqrt{F_2(\zeta,\eta)}$ is convex, being a weighted $L^4$ norm squared of the positive parts of affine functions of $\eta$. 
	Thus, given any two points in the positive set of $\Phi_a$, we can first decrease their amplitudes to a common sufficiently small positive value, then join their $\eta$-coordinates by the line segment contained in the admissible set, and finally increase the amplitude to reach the second point. This constructs a path entirely contained in the positive set of $\Phi_a$.

	Suppose by contradiction that two distinct fibre maxima exist, they can be joined in the positive set of $\Phi_a$ by a path contained in that set. Applying the mountain-pass deformation argument (see \cite[Chapter 1]{Willem1996}) to $-\Phi_a$, the corresponding minimax level is strictly negative. The Palais--Smale condition established above therefore applies at this minimax level, yielding a critical point of mountain-pass type, which hence cannot be a nondegenerate strict local minimum of $-\Phi_a$. This contradicts the negative coercivity of the fibre Hessian of $\Phi_a$ established in \eqref{fibreHessian}, and proves uniqueness of the fibre critical point. Finally, the implicit function theorem yields local $C^1$ dependence of this critical point on $\zeta$. Uniqueness ensures that these local maps agree on their overlaps, giving a globally defined $C^1$ map and proving \eqref{fullreducedenergy}.

	Since
	$$\Phi_a'(X)=A'(X)-\frac{a}{B(X)}F_{\rm e}'(X)+\frac{aF_{\rm e}(X)}{B(X)^2}B'(X)$$
	for every fibre variation $V\in\mathcal L(\mathbb R\zeta,H)$, the criticality condition $\Phi_a'(X)[V]=0$, after the change of variables $x\mapsto\tau_Xx$, becomes the last identity in \eqref{exactfibreequations}. Here we use the invariance of $A$ under spatial dilations in $\R^2$ and the fact that the integrals defining $F_{\rm e}'$ and $B'$ scale by the factor $\tau_X^{-2}=a/B(X)$. The condition $\Phi_a'(X)[X]=0$ gives $2A(X)-\frac{a}{B(X)}S_{\rm e}(X)+\frac{2a}{B(X)}F_{\rm e}(X)=0$, which is $P_{\rm e}(\mathcal N_aX)=0$.
\end{proof}

\begin{proof}[Proof of Theorem~\ref{introfullfibreexistence}]
	Define the reduced level by
	\begin{equation*}
		c_a:=\inf_{\zeta\in\mathbb S_+}m_a(\zeta)\,.
	\end{equation*}
	The lower bound $c_a\ge\beta_a$ follows from \eqref{fullreducedenergy}, while the upper bounds hold by recalling \eqref{fullreducedenergy} and using \eqref{tildeuniformupper} and \eqref{tildesharpupper} for a fixed bounded and compactly supported radial function $\phi\in\mathbb S_+$, and taking the supremum over the fibre parameters:
	\[
	c_a\leq m_a(\phi)\le\sup_{\xi>0}\mathcal U_{a,\phi}(\xi)<\frac{4\pi}{\vartheta_0}\,.
	\]
	We next construct a sequence satisfying the full residual condition, while keeping the two identities in \eqref{exactfibreequations} exact. Since $\Phi_a$ is invariant under spatial dilations, we have
	\begin{equation}\label{reduceddilationinvariance}
		m_a(\zeta_t)=m_a(\zeta),\quad\mbox{for}\ \ \zeta_t(x)=t\zeta(tx)\,.
	\end{equation}
	Choose a minimizing sequence $(\zeta_n^0)_n\subset\mathbb S_+$ for $c_a$. By \eqref{reduceddilationinvariance}, we may assume $\|\nabla\zeta_n^0\|_2=1$. Extend $m_a$ to the whole sphere $\mathbb S=\{\zeta\in H:\|\zeta\|_2=1\}$ by setting $m_a=+\infty$ on $\mathbb S\setminus\mathbb S_+$. This extension is lower semicontinuous with respect to the $H^1$ metric.
	Indeed, if $\zeta_n\to\zeta\leq0$ in $H^1$, for every fixed $r>0$, continuity of $\Phi_a$ gives
	\[
	m_a(\zeta_n)\ge\Phi_a(\mathcal L(r\zeta_n,0))\to\Phi_a(\mathcal L(r\zeta,0))=\frac{r^2}2\|\nabla\zeta\|_2^2\,.
	\]
	Since $\|\zeta\|_2=1$, we have $\|\nabla\zeta\|_2>0$. Letting $r\to+\infty$ therefore yields $m_a(\zeta_n)\to+\infty$, whenever $\zeta_n\to\zeta\in\mathbb S\setminus\mathbb S+$, as required. The sphere $\mathbb S$ is complete with respect to the distance induced by $H^1$. Ekeland's variational principle, starting with sufficiently accurate $\zeta_n^0$, gives
	$(\zeta_n)_n\subset\mathbb S_+$ such that
	\begin{equation}\label{reducedPSsequence}
	 m_a(\zeta_n)\to c_a\,,\qquad
	 \|\zeta_n-\zeta_n^0\|_{H^1}\to0\,,\qquad
	 \|dm_a(\zeta_n)\|_{(T_{\zeta_n}\mathbb S)^*}\to0\,.
	\end{equation}
	In particular $\|\nabla\zeta_n\|_2\to1$ and $(\zeta_n)$ is bounded in $H^1$. For each $n$, let $X_n=\mathcal L(r_n\zeta_n,w_n)$ be the unique maximum of $\Phi_a$ on $\mathcal K_{\zeta_n}$ given by Proposition~\ref{fullfibremaximum}, and set
	\[
	U_n=\mathcal N_aX_n,\qquad \tau_n=\sqrt{B(X_n)/a},\qquad \lambda_n=\frac{F_{\rm e}(U_n)}a>0\,.
	\]
	By \eqref{exactfibreequations}, $P_{\rm e}(U_n)=0$, and
	\[
	(J_{\rm e}'(U_n)+\lambda_nB'(U_n))[V]=0
	\]
	for every $V\in\operatorname{span}\{U_n\}+\mathcal L(0,H)$. In particular, the residual vanishes when tested against $U_n$ and $\mathcal TU_n$, since $\mathcal TU_n-U_n\in\mathcal L(0,H)$. The estimates producing the uniform bounds in the proof of Theorem~\ref{introtildecompactness} use only these two tests, the energy identity, the Poho\v zaev identity and the positivity of the multiplier. They therefore apply here, without requiring the full residual estimate, and yield
	\begin{equation}\label{fullfibrebounds}
		\|U_n\|_{H^1\times H^1}\le C\,,\qquad0<\lambda_-\le\lambda_n\le\lambda_+\,.
	\end{equation}
	Let $z_n$ denote the positive coordinate of $U_n$, so that $U_n=\mathcal L(z_n,\widetilde w_n)$ for some $\widetilde w_n\in H$. The dilation and mass identities give
\[
\|\nabla z_n\|_2^2=r_n^2\|\nabla\zeta_n\|_2^2\,,\qquad \|z_n\|_2=\frac{r_n}{\tau_n}\,,\qquad \|z_n\|_2^2\ge2a\,.
\]
	Moreover, since $A(U_n)\le\frac12\|\nabla z_n\|_2^2$ and $J_{\rm e}(U_n)=m_a(\zeta_n)\to c_a\ge\beta_a$, we have, for all sufficiently large $n$,
	\[
	\|\nabla z_n\|_2^2\ge2A(U_n)\ge2J_{\rm e}(U_n)\ge\beta_a\,.
	\]
	Together with \eqref{fullfibrebounds} and $\|\nabla\zeta_n\|_2\to1$, these estimates show that $r_n$ and $\tau_n$ are bounded and bounded away from zero.

	We now show that $(U_n)$ satisfies the full residual condition. Since $\zeta\mapsto(r_\zeta,w_\zeta)$ is of class $C^1$, for every $\eta\in T_{\zeta_n}\mathbb S$ the chain rule gives
	\[
	\dd m_a(\zeta_n)[\eta]=\Phi_a'(X_n)[\mathcal L(r_n\eta+dr_{\zeta_n}[\eta]\zeta_n,dw_{\zeta_n}[\eta])]\,.
	\]
	The last two terms inside $\mathcal L$ represent a variation within the fibre $\mathcal K_{\zeta_n}$ and hence give no contribution, since $X_n$ is a fibre critical point. Therefore,
	\[
	\dd m_a(\zeta_n)[\eta]=r_n\Phi_a'(X_n)[\mathcal L(\eta,0)]\,.
	\]
	The derivative $\Phi_a'(X_n)$ vanishes on $\mathcal L(\mathbb R\zeta_n,H)$, by the criticality of $X_n$ on its fibre. For an arbitrary $\eta\in H$, write
	\[
	\eta=\left(\eta-\left(\int_{\mathbb R^2}\zeta_n\eta\right)\zeta_n\right)+\left(\int_{\mathbb R^2}\zeta_n\eta\right)\zeta_n\,.
	\]
	The first term belongs to $T_{\zeta_n}\mathbb S$, while the second gives a direction on which $\Phi_a'(X_n)$ vanishes. Since $(\zeta_n)$ is bounded in $H^1$, the projection onto $T_{\zeta_n}\mathbb S$ is uniformly bounded in $H^1$. It follows from \eqref{reducedPSsequence} and the positive lower bound for $r_n$ that $\|\Phi_a'(X_n)\|_{(H\times H)^*}\to0$. For every $V\in H\times H$, one has
	\[
	\Phi_a'(X_n)[V]=(J_{\rm e}'(U_n)+\lambda_nB'(U_n))[V(\tau_n\,\cdot)]\,.
	\]
	Since $\tau_n$ and $1/\tau_n$ are uniformly bounded, the maps $V\mapsto V(\tau_n\,\cdot)$ are uniformly bounded isomorphisms of $H\times H$. Hence
	\[
	\|(J_{\rm e}'(U_n)+\lambda_nB'(U_n))\|_{(H^1\times H^1)^*}\to0\,.
	\]
	In view of \eqref{fullfibrebounds}, this also yields the weighted Cerami condition. 
	We can therefore apply Theorem~\ref{introtildecompactness}: up to a subsequence, $U_n$ converges strongly in $H^1(\R^2)^2$ to a positive radial classical solution $(u,v)$ of the system, with $B(u,v)=a$ and $J_{\rm e}(u,v)=c_a$.

	 It remains to prove that this solution has least energy among all positive radial solutions with mass $a$. Let $(u,v)$ be any such solution. Testing the system against $(u,v)$ and using the Pohožaev identity, we obtain
	\[
	2A(u,v)+2\lambda a=S_{\rm e}(u,v)\,,\qquad 2A(u,v)-S_{\rm e}(u,v)+2F_{\rm e}(u,v)=0\,.
	\]
	Consequently, $\lambda a=F_{\rm e}(u,v)$. Introduce the coordinates $(z,w)$ defined by $(u,v)=\mathcal L(z,w)$, and set $\zeta=z/\|z\|_2$. Since $u,v>0$, we have $\zeta\in\mathbb S_+$. Moreover, $B(u,v)=a$ implies $\|z\|_2^2=2a+\|w\|_2^2$, so $(u,v)=\mathcal X_a(\zeta,w)$. The equation and the identity $\lambda a=F_{\rm e}(u,v)$ imply, by the quotient derivative formula, that $\Phi_a'(u,v)$ vanishes on $\operatorname{span}\{u,v\}+\mathcal L(0,H)$. Thus $(u,v)$ is a critical point of $\Phi_a$ on $\mathcal K_\zeta$. By Proposition~\ref{fullfibremaximum}, it must be the unique maximum on that fibre. Therefore,
	\[
	J_{\rm e}(u,v)=\Phi_a(u,v)=m_a(\zeta)\ge c_a\,.
	\]
	This proves the least-energy property and completes the proof.
\end{proof}

\section{Branching solutions at the mass-threshold $a_c$}\label{secTMsolutions_ac}
We now 
turn to local solution construction. The pure exponential branch is treated before perturbations of a power system.

\subsection{Pure exponential systems at small frequency}\label{Section_11.2}

\begin{proof}[Proof of Theorem~\ref{intropureexponential}]
In addition to the notations $\tau,a_{\rm c},\kappa$ from the statement, set
\begin{equation}\label{TMpureconstants}
	c_u=(\vartheta_1\vartheta_2^3)^{-1/8},\qquad 	c_v=(\vartheta_1^3\vartheta_2)^{-1/8},\quad\mbox{so that}\quad c_uc_v=(\vartheta_1\vartheta_2)^{-1/2}\,.
\end{equation}
We look for solutions branching from $(U_0,V_0)=(Q_4,Q_4)$, so we introduce the rescaled functions
\begin{equation}\label{TMpurescaling}
	u_\lambda(x)=c_u\sqrt\lambda\,U_\lambda(\sqrt\lambda x)\,,\qquad
	v_\lambda(x)=c_v\sqrt\lambda\,V_\lambda(\sqrt\lambda x)\,.
\end{equation}
To retain a nonnegative right-hand side before positivity of solutions is known, $s,t\in\R$, define
$$g_s(t)=t_+^3\int_0^1e^{srt_+^2}\dd r\,.$$
Thus $g_0(t)=t_+^3$, and for $s\ne0$,
$g_s(t)=t_+(e^{st_+^2}-1)/s$.
Substitution of \eqref{TMpurescaling} into the positive-part version
of \eqref{TMpuresystem} gives
\begin{equation}\label{TMpurerescaled}
	-\Delta U_\lambda+U_\lambda=g_{\lambda\tau}(V)\,,\qquad
	-\Delta V_\lambda+V_\lambda=g_{\lambda/\tau}(U)\,.
\end{equation}
The choice of $c_u,c_v$ above ensures that the cubic terms in the rescaled equations have coefficient one, since $\vartheta_1c_v^3/c_u=\vartheta_2c_u^3/c_v=1$, and
$\vartheta_1c_v^2=\tau$, $\vartheta_2c_u^2=\tau^{-1}$. At $\lambda=0$, this system reduces to
$$-\Delta U_0+U_0=(V_0)_+^3,\qquad-\Delta V_0+V_0=(U_0)_+^3,$$
which is solved by $(U,V)=(Q,Q)$, where we write $Q=Q_4$ for simplicity. Aiming at branching from this solution, define the implicit function $\cF:\R\times(H^2_{\rm rad})^2\to(L^2_{\rm rad})^2$ by
$$\cF(\lambda,U,V)=\big(-\Delta U+U-g_{\lambda\tau}(V)\,,\ -\Delta V+V-g_{\lambda/\tau}(U)\big).$$
We claim that $\mathcal F$ is of class $C^1$ in a neighbourhood of $(0,Q,Q)$. Indeed, on bounded intervals of $s$ and $t$, the function $g_s(t)$ and its first derivatives are continuous with uniformly bounded second derivatives locally in $s$. Taylor's formula with integral remainder yields
$$\|g_s(V+h)-g_s(V)-\partial_tg_s(V)h\|_2\le C\|h\|_\infty\|h\|_2\le C\|h\|_{H^2}^2\,,$$
and similarly for the derivative in $s$.  At $(0,Q,Q)$, the derivative in $(U,V)$ is diagonalised, by summing and subtracting the two components into
\begin{equation}\label{L+-}
	L_+=-\Delta+1-3Q^2,\qquad L_-=-\Delta+1+3Q^2.
\end{equation}
These are the symmetric and antisymmetric parts of
\eqref{L_lambda} at $p=q=4$ and $\lambda=1$. We claim that 
both $L_+$ and $L_-$ are isomorphisms from $H^2_{\rm rad}$ to $L^2_{\rm rad}$, and we postpone its verification at the end of this proof. Hence the system linearisation is an isomorphism from $(H^2_{\rm rad})^2$ to $(L^2_{\rm rad})^2$. 
Since $\mathcal F(0,Q,Q)=0$, the implicit function theorem then gives a unique local $C^1$ branch $\lambda\mapsto(U_\lambda,V_\lambda)$ such that $(U_0,V_0)=(Q,Q)$, defined for $\lambda$ in an open interval containing zero.
For $\lambda>0$ the right-hand sides in \eqref{TMpurerescaled} are nonnegative, hence the maximum principle implies that $U_\lambda,V_\lambda>0$, which are
classical and exponential decaying arguing as in Lemma \eqref{profileregularity}. Rescaling back by \eqref{TMpurescaling} gives positive radial classical exponentially decaying solutions of the original system.

	Let us now compute the first-order variation of the mass. Define
	\begin{equation}\label{a_lambda}
		a(\lambda)=B(u_\lambda,v_\lambda)=\frac1{\sqrt{\vartheta_1\vartheta_2}}\int U_\lambda V_\lambda
	\end{equation}
	by \eqref{TMpurescaling} and \eqref{TMpureconstants}.
	This formula provides a $C^1$ extension of the mass to $\lambda=0$. In particular,
	$$a(0)=\frac1{\sqrt{\vartheta_1\vartheta_2}}\int Q^2=a_{\rm c}\,.$$
	We now compute $a'(0)$: differentiating \eqref{a_lambda} yields
	$$a'(0)=\frac1{\sqrt{\vartheta_1\vartheta_2}}\int\left(\dot U_0Q+Q\dot V_0\right)=\frac1{\sqrt{\vartheta_1\vartheta_2}}\int QS\,,$$ 
	where we denote by $\dot U_0$ and $\dot V_0$ the derivatives of the solution branch with respect to $\lambda$, evaluated at $\lambda=0$, and we set $S=\dot U_0+\dot V_0$. 
	To evaluate this integral, we exploit the scaling identity of the Townes profile $Q$. For $\omega>0$, define
	$$
	Q_\omega(x)=\sqrt\omega\,Q(\sqrt\omega x\,)\,,
	$$
	which solves $-\Delta Q_\omega+\omega Q_\omega=Q_\omega^3$ by \eqref{Q4}. Differentiating with respect to $\omega$ at $\omega=1$, and using
	$$
	\left.\partial_\omega Q_\omega\right|_{\omega=1}=\frac12(Q+x\cdot\nabla Q) Q,
	$$
	we obtain
	\begin{equation}\label{Lambda}
		L_+(Q+x\cdot\nabla Q)=-2Q\,.
	\end{equation}
	On the other hand, since
	$$
	\left.\partial_sg_s(Q)\right|_{s=0}=Q^5\int_0^1r\dd r=\frac12Q^5,
	$$
	differentiating \eqref{TMpurerescaled} at $\lambda=0$ gives
	$$(-\Delta+1)\dot U_0=3Q^2\dot V_0+\frac\tau2Q^5\,,\qquad(-\Delta+1)\dot V_0=3Q^2\dot U_0+\frac1{2\tau}Q^5\,.$$
	Adding these equations, we find
	\begin{equation}\label{eq:S}
		L_+S=\frac{\tau+\tau^{-1}}2Q^5.
	\end{equation}
	By self-adjointness of $L_+$ and \eqref{Lambda}--\eqref{eq:S}, it follows that
	\begin{equation*}
		\begin{split}
			a'(0)&=\frac1{\sqrt{\vartheta_1\vartheta_2}}\int QS=-\frac1{2\sqrt{\vartheta_1\vartheta_2}}\int(Q+x\cdot\nabla Q)L_+S=-\frac{\tau+\tau^{-1}}{4\sqrt{\vartheta_1\vartheta_2}}\int(Q+x\cdot\nabla Q)Q^5\\
			&=-\frac{\tau+\tau^{-1}}{4\sqrt{\vartheta_1\vartheta_2}}\left(\int Q^6+\frac16\int x\cdot\nabla(Q^6)\right)=-\frac{\tau+\tau^{-1}}{4\sqrt{\vartheta_1\vartheta_2}}\left(\int Q^6-\frac13\int Q^6\right)
		\end{split}
	\end{equation*}
	by integration by parts
	and the identity $\int x\cdot\nabla f=-2\int f$ in dimension two. We conclude that
	$$
	a'(0)=-\frac{\tau+\tau^{-1}}{6\sqrt{\vartheta_1\vartheta_2}}\int Q^6=-\kappa.
	$$
	Since $a$ is $C^1$ up to $\lambda=0$, this proves \eqref{TMpuremass}. By continuity, $a'(\lambda)<0$ for sufficiently small positive $\lambda$. Thus $a(\lambda)$ is strictly decreasing and maps an interval $(0,\lambda_*)$ onto $(a_{\rm c}-\delta,a_{\rm c})$, for some $\delta>0$. Its inverse provides a frequency $\lambda=\lambda(a)$ for every mass in this interval. The expansion \eqref{TMpuremass} gives
	$$
	\lambda(a)=\frac{a_{\rm c}-a}{\kappa}+o(a_{\rm c}-a)
	$$
	as $a\nearrow a_{\rm c}$. Finally, the stated uniqueness is a direct consequence of the implicit function theorem and this scalar inversion.
	
	It remains to prove the fact that $L_\pm$ defined in \eqref{L+-} are isomorphisms  from $H^2_{\rm rad}$ to $L^2_{\rm rad}$. Concerning $L_+$, in \eqref{antisymmetricexponential} we already proved its coercivity. On the other hand, write $L_-=L_0-\cK=L_0(Id-L_0^{-1}\cK)$, where $L_0=-\Delta+\lambda$ is invertible since $\lambda>0$, and $\cK:\phi\mapsto3Q^2\phi$. Combining $\|Q\|_{L^\infty(\R^2\setminus B_R(0))}\to0$ as $R\to+\infty$ with the compact embedding $H^1(B_R(0))\hookrightarrow L^2(B_R(0))$, it is easy to see that $\cK$ is a compact operator, and so is $L_0^{-1}\cK$. This shows that $Id-L_0^{-1}\cK$ is a Fredholm operator of index $0$ and the same then holds for $L_-$. The radial nondegeneracy then implies that $L_-$ is also surjective, which makes $L_-$ an isomorphism.
\end{proof}

\subsection{Local normalised solutions with exponential terms}\label{Section_11.3}

\begin{proof}[Proof of Theorem~\ref{introexponentialcontinuation}]
	We introduce the map $\cF:(2,\infty)^2\times\R^2\times(H^1_{\rm rad}(\R^2))^2\times\R\to(H^{-1}_{\rm rad}(\R^2))^2\times\R$ by
	\begin{equation*}
		\cF(p,q,\varepsilon_1,\varepsilon_2,u,v,\lambda)=
		\begin{pmatrix}
			-\Delta u+\lambda u-v_+^{p-1}-\varepsilon_1h_{\vartheta_1}(v)\\
			-\Delta v+\lambda v-u_+^{q-1}-\varepsilon_2h_{\vartheta_2}(u)\\
			B(u,v)-a
		\end{pmatrix}.
	\end{equation*}
	By Lemma~\ref{TMmapping}, applied at the bounded base components $u_0$ and $v_0$, the exponential nonlinearities $h_\vartheta$ define a $C^1$ map in a neighbourhood of the base point. Together with the differentiability estimates for the power nonlinearities with respect to both the functions and the exponents established in the proof of Theorem~\ref{intropowercontinuation}, this shows that $\mathcal F$ is $C^1$ near $(p_0,q_0,0,0,u_0,v_0,\lambda_0)$. At $\varepsilon_1=\varepsilon_2=0$, its derivative with respect to $(u,v,\lambda)$ at the base point coincides with $\mathcal A_{\lambda_0}$ defined in \eqref{augmentedoperator}, which is an isomorphism provided $\frac1{p_0}+\frac1{q_0}\neq\frac12$ holds by Lemma~\ref{linearisation}. The implicit function theorem therefore yields a unique local $C^1$ branch of solutions parametrised by $(p,q,\varepsilon_1,\varepsilon_2)$. 
	Finally, shrinking the parameter neighbourhood if necessary, continuity ensures that $\lambda>0$ and that both components remain nontrivial, hence positivity and decay follow as in the proof of Theorem~\ref{intropureexponential}.
\end{proof}


\subsection{Exponential mass response on the planar critical curve}\label{Section_exp_response}
On $1/p+1/q=1/2$, the power family has constant bilinear mass
$a_*$.
We use the fixed-frequency isomorphism in Lemma~\ref{linearisation} and the
symmetric crossed operator from Subsection~\ref{Section_massresponse}
to compute the response to exponential perturbations.

\begin{proof}[Proof of Proposition~\ref{introcriticalresponse}]

Fix $\lambda_0>0$ and keep $p,q$ fixed on the critical curve.
Consider the map $\cF:\R^3\times(H^1_{\rm rad}(\R^2))^2\to(H^{-1}_{\rm rad}(\R^2))^2$ given by
\[
\mathcal F(\lambda,\varepsilon_1,\varepsilon_2,u,v)=
\begin{pmatrix}
	-\Delta u+\lambda u-v_+^{p-1}-\varepsilon_1h_{\vartheta_1}(v)\\
	-\Delta v+\lambda v-u_+^{q-1}-\varepsilon_2h_{\vartheta_2}(u)
\end{pmatrix}.
\]
Since the base solutions $u_{\lambda_0},v_{\lambda_0}$ given by Theorem \ref{introexistencescaling} are bounded, Lemma~\ref{TMmapping}, together with the power mapping estimates	in Section~\ref{Section_branches_powers}, ensures that $\mathcal F$ is $C^1$ in a neighbourhood of
$(\lambda_0,0,0,u_{\lambda_0},v_{\lambda_0})$.
Its derivative with respect to $(u,v)$ at this point is
$\mathcal L_{\lambda_0}$, which is an isomorphism by the first
step of the proof of Lemma~\ref{linearisation}; this remains
true also on the critical curve.
The implicit function theorem therefore provides a unique local
$C^1$ branch
$(\lambda,\varepsilon_1,\varepsilon_2)\mapsto
(u_{\lambda,\varepsilon},v_{\lambda,\varepsilon})$, which
coincides with $(u_\lambda,v_\lambda)$ when $\varepsilon_1=\varepsilon_2=0$.
For nonnegative perturbation coefficients, positivity,
classical regularity and exponential decay follow as in the
proof of Theorem~\ref{introexponentialcontinuation}.

To compute the first-order variation of the mass, write
$(h_i,k_i)=\partial_{\varepsilon_i}
(u_{\lambda,\varepsilon},v_{\lambda,\varepsilon})|_{\varepsilon=0}$
and introduce the linearised operator
$$\mathcal H_\lambda(h,k)=
 \begin{pmatrix}
 (-\Delta+\lambda)k-(q-1)u_\lambda^{q-2}h\\
 (-\Delta+\lambda)h-(p-1)v_\lambda^{p-2}k
 \end{pmatrix}.$$
Differentiating the equations with respect to frequency and to the
two perturbation parameters gives
 \begin{equation*}
 \mathcal H_\lambda(\partial_\lambda u_\lambda, \partial_\lambda v_\lambda)=-(v_\lambda,u_\lambda)\,,\quad
 \mathcal H_\lambda(h_1,k_1)=(0,h_{\vartheta_1}(v_\lambda))\,,\quad
 \mathcal H_\lambda(h_2,k_2)=(h_{\vartheta_2}(u_\lambda),0)\,.
 \end{equation*}
The operator $\mathcal H_\lambda$ is symmetric with respect to
the $H^{-1}$--$H^1$ duality pairing. Testing the first identity
with $(h_1,k_1)$ and using the second one (and similarly with the roles exchanged) we obtain
\begin{equation}\label{TMmassderivative}
 \begin{aligned}
 \left.\partial_{\varepsilon_1}B\right|_{\varepsilon=0}
 &=\int(v_\lambda h_1+u_\lambda k_1)=-\int h_{\vartheta_1}(v_\lambda)\partial_\lambda v_\lambda
   =-\frac{\dd}{\dd\lambda}\int H_{\vartheta_1}(v_\lambda)\,,\\
 \left.\partial_{\varepsilon_2}B\right|_{\varepsilon=0}
 &=-\frac{\dd}{\dd\lambda}\int H_{\vartheta_2}(u_\lambda)\,.
 \end{aligned}
\end{equation}
These are the exponential counterparts of the two identities in


It remains to evaluate these derivatives using the frequency
scaling. Setting $w_\lambda(x)=\lambda^sW(\sqrt\lambda x)$, recalling
$H_\vartheta(t)=\frac1{2\vartheta}\Psi(\vartheta t^2)$, writing $z=\vartheta\lambda^{2s}W^2$ and differentiating, one finds
$$\frac{\dd}{\dd\lambda}\int H_\vartheta(w_\lambda)=\frac{\dd}{\dd\lambda}\left(\frac1{2\vartheta\lambda}
\int\Psi(\vartheta\lambda^{2s}W^2)\right)=\frac1{2\vartheta\lambda^2}\int\big[2sz(e^z-1)-\Psi(z)\big]=D_{\vartheta,s}[W](\lambda)\,,$$
i.e. \eqref{TMcoefficient}. 
The differentiation under the integral sign is justified on
compact intervals of positive frequencies: indeed, $W$ is
bounded and exponentially decaying, and the integrand and its
$\lambda$-derivative are bounded by a constant times $W^4$.
Applying this formula to $v_\lambda$ and $u_\lambda$ in
\eqref{TMmassderivative} gives respectively
$-D_{\vartheta_1,s_v}[V](\lambda)$ and
$-D_{\vartheta_2,s_u}[U](\lambda)$. 
Moreover, the critical scaling $s_u+s_v=1$ implies $B(u_\lambda,v_\lambda)=\lambda^{s_u+s_v-1}\int UV=a_*$, hence a first-order expansion of the mass with respect to $(\varepsilon_1,\varepsilon_2)$, together with \eqref{TMmassderivative}, now gives \eqref{TMexpansion}. $C^1$ regularity in all parameters, makes the remainder locally uniform in $\lambda$.

We finally determine the signs of the coefficients in the expansion \eqref{TMexpansion} from the absolutely convergent series
\begin{equation}\label{TMresponseseries}
	D_{\vartheta,s}[W](\lambda)=\sum_{j=2}^{\infty}\frac{\vartheta^{j-1}}{2j!}(2sj-1)\lambda^{2sj-2}\int W^{2j}.
\end{equation}
If $s\ge1/4$, all terms are nonnegative, and the terms with
$j\ge3$ are strictly positive because $W$ is nontrivial.
Consequently, $D_{\vartheta,s}[W](\lambda)>0$ for every
$\lambda>0$. Since $s_v=2/p$ and $s_u=2/q$, both coefficients
in \eqref{TMexpansion} are strictly positive whenever
$p,q\le8$, including the endpoints. Their positive lower bounds on a small
compact frequency interval, together with the uniform remainder,
give $B(u_{\lambda,\varepsilon},v_{\lambda,\varepsilon})<a_*$
for every sufficiently small nonzero $\varepsilon_1,\varepsilon_2\ge0$. Local completeness of the implicit-function branch
excludes any other nearby radial solution of mass $a_*$ for these
parameters. 
	Finally, if $0<s<1/4$, the first term of
	\eqref{TMresponseseries} dominates as $\lambda\to0$.
	More precisely,
	$$D_{\vartheta,s}[W](\lambda)=\frac{\vartheta(4s-1)}4\lambda^{4s-2}\int W^4+O(\lambda^{6s-2}).$$
	The coefficient is therefore negative at sufficiently small
	positive frequencies. Taking $s=s_v=2/p$ when $p>8$, or $s=s_u=2/q$
	when $q>8$, proves the remaining assertions.
\end{proof}

On the diagonal $p=q=4$, one has $U=V=Q_4$ and $s_u=s_v=1/2$.
Formula~\eqref{TMcoefficient} then reduces to
\[
 D_{\vartheta,1/2}[Q_4](\lambda)=\frac1{2\vartheta\lambda^2}\int\big[(\vartheta\lambda Q_4^2-1)e^{\vartheta\lambda Q_4^2}+1\big]>0\,,
\]
recovering the cubic response. Off the diagonal, the factor
$2sj-1$ in \eqref{TMresponseseries} is exactly the planar
power-response exponent in Proposition~\ref{introcriticalresponseNgeq3}
for the perturbing power $\ell=2j$. Thus the exponential response
combines all the even-power responses with their exponential-series
weights. If $p>8$, for example, at a sufficiently small fixed
$\lambda>0$ the branch with $\varepsilon_1>0$, $\varepsilon_2=0$
has mass greater than $a_*$ for small $\varepsilon_1$.
The segment $\max\{p,q\}\le8$ is consequently the largest part of
the critical curve on which every nonzero nonnegative perturbation
has a strictly negative first-order mass response at every positive
frequency. 

\section*{Acknowledgments}
\noindent The authors thank Professor Xu Zhang (Central South University) for fruitful discussions. G. Romani is partially supported by the INdAM-GNAMPA 2026 project \textit{Structural degeneracy and criticality in (sub)elliptic PDEs} (E53C25002010001).

\providecommand{\bysame}{\leavevmode\hbox to3em{\hrulefill}\thinspace}
\providecommand{\MR}{\relax\ifhmode\unskip\space\fi MR }
\providecommand{\MRhref}[2]{%
  \href{http://www.ams.org/mathscinet-getitem?mr=#1}{#2}
}
\providecommand{\href}[2]{#2}


\begin{thebibliography}{10}
	
\bibitem{AdTa}
Shinji Adachi and Kazunaga Tanaka, \emph{Trudinger type inequalities in $\mathbb{R}^N$ and their best exponents}, Proc. Amer. Math. Soc., \textbf{128} (2000), 2051--2057,
\href{https://doi.org/10.1090/S0002-9939-99-05180-1}{doi:10.1090/S0002-9939-99-05180-1}.

\bibitem{ACS}
Yulian An, Jann-Long Chern, and Junping Shi, \emph{Uniqueness of positive
  solutions to some coupled cooperative variational elliptic systems}, Trans.
  Amer. Math. Soc. \textbf{370} (2018), 5209--5243,
  \href{https://doi.org/10.1090/tran/7207}{doi:10.1090/tran/7207}.

\bibitem{BariczPonnusamy}
 \'Arp\'ad Baricz and Saminathan Ponnusamy,
\emph{On Tur\'an type inequalities for modified Bessel functions},
Proc. Amer. Math. Soc. \textbf{141} (2013), 523--532,
\href{https://doi.org/10.1090/S0002-9939-2012-11325-5}{doi:10.1090/S0002-9939-2012-11325-5}.

\bibitem{BartschJeanjean}
Thomas Bartsch and Louis Jeanjean, \emph{Normalized solutions for nonlinear
  {Schr\"odinger} systems}, Proc. Roy. Soc. Edinburgh Sect. A \textbf{148}
  (2018), 225--242,
  \href{https://doi.org/10.1017/S0308210517000087}{doi:10.1017/S0308210517000087}.

\bibitem{BJS}
Thomas Bartsch, Louis Jeanjean, and Nicola Soave, \emph{Normalized solutions
  for a system of coupled cubic {Schr\"odinger} equations on {$\R^3$}},
  J. Math. Pures Appl. \textbf{106} (2016), 583--614,
  \href{https://doi.org/10.1016/j.matpur.2016.03.004}{doi:10.1016/j.matpur.2016.03.004}.

\bibitem{BS}
Thomas Bartsch and Nicola Soave, \emph{A natural constraint approach to
  normalized solutions of nonlinear {Schr\"odinger} equations and systems}, J.
  Funct. Anal. \textbf{272} (2017), 4998--5037,
  \href{https://doi.org/10.1016/j.jfa.2017.01.025}{doi:10.1016/j.jfa.2017.01.025}.

\bibitem{BScorrection}
\bysame, \emph{Correction to: ``{A} natural constraint approach to normalized
  solutions of nonlinear {Schr\"odinger} equations and systems''}, J. Funct.
  Anal. \textbf{275} (2018), 516--521,
  \href{https://doi.org/10.1016/j.jfa.2018.02.007}{doi:10.1016/j.jfa.2018.02.007}.

\bibitem{BMR}
Denis Bonheure, Ederson Moreira~dos Santos, and Miguel Ramos, \emph{Ground
  state and non-ground state solutions of some strongly coupled elliptic
  systems}, Trans. Amer. Math. Soc. \textbf{364} (2012), 447--491,
  \href{https://doi.org/10.1090/S0002-9947-2011-05452-8}{doi:10.1090/S0002-9947-2011-05452-8}.

\bibitem{BMT}
Denis Bonheure, Ederson Moreira~dos Santos, and Hugo Tavares, \emph{Hamiltonian
  elliptic systems: a guide to variational frameworks}, Portugal. Math.
  \textbf{71} (2014), 301--395,
  \href{https://doi.org/10.4171/PM/1954}{doi:10.4171/PM/1954}.

\bibitem{Brezis}
Haim Brezis, \emph{Functional analysis, sobolev spaces and partial differential
  equations}, Universitext, Springer, New York, 2011,
  \href{https://doi.org/10.1007/978-0-387-70914-7}{doi:10.1007/978-0-387-70914-7}.

\bibitem{BuscaSirakov}
J\'er\^ome Busca and Boyan Sirakov, \emph{Symmetry results for semilinear
  elliptic systems in the whole space}, J. Differential Equations \textbf{163}
  (2000), 41--56,
  \href{https://doi.org/10.1006/jdeq.1999.3701}{doi:10.1006/jdeq.1999.3701}.

\bibitem{CST}
Daniele Cassani, Federica Sani, and Cristina Tarsi, \emph{Equivalent {Moser}
  type inequalities in {$\R^2$} and the zero mass case}, J. Funct. Anal.
  \textbf{267} (2014), no.~11, 4236--4263,
  \href{https://doi.org/10.1016/j.jfa.2014.09.022}{doi:10.1016/j.jfa.2014.09.022}.

\bibitem{CassaniTarsi2015}
Daniele Cassani and Cristina Tarsi, \emph{Existence of solitary waves for
  supercritical {Schr\"odinger} systems in dimension two}, Calc. Var.
  Partial Differential Equations \textbf{54} (2015), no.~2, 1673--1704,
  \href{https://doi.org/10.1007/s00526-015-0840-3}{doi:10.1007/s00526-015-0840-3}.

\bibitem{CassaniTavaresZhang2020}
Daniele Cassani, Hugo Tavares, and Jianjun Zhang,
\emph{Bose fluids and positive solutions to weakly coupled systems with
critical growth in dimension two}, J. Differential Equations
\textbf{269} (2020), no.~3, 2328--2385,
\href{https://doi.org/10.1016/j.jde.2020.01.036}{doi:10.1016/j.jde.2020.01.036}.

\bibitem{CassaniZhang2017}
Daniele Cassani and Jianjun Zhang,
\emph{A priori estimates and positivity for semiclassical ground states
for systems of critical {Schr\"odinger} equations in dimension two},
Comm. Partial Differential Equations \textbf{42} (2017), no.~5, 655--702,
\href{https://doi.org/10.1080/03605302.2017.1295062}{doi:10.1080/03605302.2017.1295062}.

\bibitem{CassaniXuZhang}
Daniele~Cassani and Xu~Zhang, \emph{Uniqueness and nondegeneracy for critical fractional {Lane--Emden}
  systems}, Preprint 2026.
  
\bibitem{Cazenave}
Thierry Cazenave, \emph{Semilinear {Schr\"odinger} equations}, Courant Lecture
  Notes in Mathematics, vol.~10, New York University, Courant Institute of
  Mathematical Sciences and American Mathematical Society, 2003,
  \href{https://doi.org/10.1090/cln/010}{doi:10.1090/cln/010}.

\bibitem{CazenaveLions1982}
Thierry Cazenave and Pierre-Louis Lions,
\emph{Orbital stability of standing waves for some nonlinear
{Schr\"odinger} equations}, Comm. Math. Phys. \textbf{85} (1982),
no.~4, 549--561,
\href{https://doi.org/10.1007/BF01403504}{doi:10.1007/BF01403504}.

\bibitem{CGNT}
Shu-Ming Chang, Stephen Gustafson, Kenji Nakanishi, and Tai-Peng Tsai,
  \emph{Spectra of linearized operators for {NLS} solitary waves}, SIAM J.
  Math. Anal. \textbf{39} (2007), no.~4, 1070--1111,
  \href{https://doi.org/10.1137/050648389}{doi:10.1137/050648389}.

\bibitem{ChangLiuYan2023}
Xiaojun Chang, Manting Liu, and Duokui Yan,
\emph{Normalized ground state solutions of nonlinear {Schr\"odinger}
equations involving exponential critical growth},
J. Geom. Anal. \textbf{33} (2023), article 83,
\href{https://doi.org/10.1007/s12220-022-01130-8}{doi:10.1007/s12220-022-01130-8}.

\bibitem{CXZ}
Jing Chen, Zheng Xie, and Xinghua Zhang, \emph{Normalized ground state
  solutions for the {Schr\"odinger} systems with critical exponential growth in
  {$\R^2$}}, J. Math. Anal. Appl. \textbf{540} (2024), no.~1, 128628,
  \href{https://doi.org/10.1016/j.jmaa.2024.128628}{doi:10.1016/j.jmaa.2024.128628}.

\bibitem{CZZ}
Zhen Chen, Xuexiu Zhong, and Wenming Zou, \emph{Normalized solutions for
  nonlinear {Schr\"odinger} systems with special mass-mixed terms: {The} linear
  couple case}, 2021,
  \href{https://arxiv.org/abs/2107.12564}{arXiv:2107.12564}.

\bibitem{CCST}
Zhi-You Chen, Jann-Long Chern, Junping Shi, and Yong-Li Tang, \emph{On the
  uniqueness and structure of solutions to a coupled elliptic system}, J.
  Differential Equations \textbf{249} (2010), no.~12, 3419--3442,
  \href{https://doi.org/10.1016/j.jde.2010.09.001}{doi:10.1016/j.jde.2010.09.001}.

\bibitem{ChiaoGarmireTownes1964}
Raymond~Y. Chiao, Elsa Garmire, and Charles~H. Townes,
\emph{Self-trapping of optical beams}, Phys. Rev. Lett. \textbf{13}
(1964), 479--482,
\href{https://doi.org/10.1103/PhysRevLett.13.479}{doi:10.1103/PhysRevLett.13.479}.

\bibitem{dF}
Djairo G.~De~Figueiredo, \emph{Semilinear elliptic systems: existence,
  multiplicity, symmetry of solutions}, Handbook of Differential Equations:
  Stationary Partial Differential Equations, vol.~5, Elsevier, New York, 2008,
  pp.~1--48.

\bibitem{FigueiredoDoORuf2005}
Djairo~G. de Figueiredo, Jo\~ao Marcos do \'O, and Bernhard Ruf,
\emph{An Orlicz-space approach to superlinear elliptic systems},
J. Funct. Anal. \textbf{224} (2005), no.~2, 471--496,
\href{https://doi.org/10.1016/j.jfa.2004.09.008}{doi:10.1016/j.jfa.2004.09.008}.

\bibitem{FY}
Djairo~G. de~Figueiredo and Jianfu Yang, \emph{Decay, symmetry and existence of
  solutions of semilinear elliptic systems}, Nonlinear Anal. \textbf{33}
  (1998), 211--234,
  \href{https://doi.org/10.1016/S0362-546X(97)00548-8}{doi:10.1016/S0362-546X(97)00548-8}.

\bibitem{DerezinskiRichard}
Jan Derezi\'nski and Serge Richard,
\emph{On Schr\"odinger operators with inverse square potentials on the half-line},
Ann. Henri Poincar\'e \textbf{18} (2017), 869--928.

\bibitem{FFRS}
Alberto Fiorenza, Maria~Rosaria Formica, Tom\'a\v{s} Roskovec, and Filip
  Soudsk\'y, \emph{Detailed proof of classical {Gagliardo--Nirenberg}
  interpolation inequality with historical remarks}, Z. Anal. Anwend.
  \textbf{40} (2021), no.~2, 217--236,
  \href{https://doi.org/10.4171/ZAA/1681}{doi:10.4171/ZAA/1681}.

\bibitem{FrankKimPistoia}
Rupert~L. Frank, Seunghyeok Kim, and Angela Pistoia, \emph{Non-degeneracy for
  the critical {Lane}--{Emden} system}, Proceedings of the American
  Mathematical Society \textbf{149} (2021), no.~1, 265--278,
  \href{https://doi.org/10.1090/proc/15217}{doi:10.1090/proc/15217}.

\bibitem{GilbargTrudinger}
David Gilbarg and Neil~S. Trudinger, \emph{Elliptic partial differential
  equations of second order}, 2 ed., Classics in Mathematics, Springer, Berlin,
  2001,
  \href{https://doi.org/10.1007/978-3-642-61798-0}{doi:10.1007/978-3-642-61798-0}.

\bibitem{GJ}
Tianxiang Gou and Louis Jeanjean, \emph{Multiple positive normalized solutions
  for nonlinear {Schr\"odinger} systems}, Nonlinearity \textbf{31} (2018),
  no.~5, 2319--2345,
  \href{https://doi.org/10.1088/1361-6544/aab0bf}{doi:10.1088/1361-6544/aab0bf}.

\bibitem{GrillakisShatahStrauss1987}
Manoussos Grillakis, Jalal Shatah, and Walter Strauss,
\emph{Stability theory of solitary waves in the presence of symmetry, I},
J. Funct. Anal. \textbf{74} (1987), no.~1, 160--197,
\href{https://doi.org/10.1016/0022-1236(87)90044-9}{doi:10.1016/0022-1236(87)90044-9}.

\bibitem{Jeanjean}
Louis Jeanjean, \emph{Existence of solutions with prescribed norm for
  semilinear elliptic equations}, Nonlinear Anal. \textbf{28} (1997),
  1633--1659,
  \href{https://doi.org/10.1016/S0362-546X(96)00021-1}{doi:10.1016/S0362-546X(96)00021-1}.

\bibitem{KevrekidisFrantzeskakis2016}
P.~G. Kevrekidis and D.~J. Frantzeskakis, \emph{Solitons in coupled nonlinear
  {Schr\"odinger} models: A survey of recent developments}, Reviews in Physics
  \textbf{1} (2016), 140--153,
  \href{https://doi.org/10.1016/j.revip.2016.07.002}{doi:10.1016/j.revip.2016.07.002}.

\bibitem{Kwong}
Man~Kam Kwong, \emph{Uniqueness of positive solutions of {$\Delta u-u+u^p=0$}
  in {$\R^n$}}, Arch. Rational Mech. Anal. \textbf{105} (1989),
  243--266, \href{https://doi.org/10.1007/BF00251502}{doi:10.1007/BF00251502}.

\bibitem{LiebLoss}
Elliott~H. Lieb and Michael Loss, \emph{Analysis}, 2 ed., Graduate Studies in
  Mathematics, vol.~14, American Mathematical Society, Providence, RI, 2001,
  \href{https://bookstore.ams.org/view?ProductCode=GSM%2F14.R}{Publisher
  record}.

\bibitem{Lions1}
Pierre-Louis Lions, \emph{The concentration-compactness principle in the
  calculus of variations. the locally compact case, part 1}, Ann. Inst. H.
  Poincar\'e Anal. Non Lin\'eaire \textbf{1} (1984), 109--145,
  \href{https://www.numdam.org/item/AIHPC_1984__1_2_109_0/}{Numdam full text}.

\bibitem{Lions2}
\bysame, \emph{The concentration-compactness principle in the calculus of
  variations. the locally compact case, part 2}, Ann. Inst. H. Poincar\'e Anal.
  Non Lin\'eaire \textbf{1} (1984), 223--283,
  \href{https://www.numdam.org/item/AIHPC_1984__1_4_223_0/}{Numdam full text}.

\bibitem{MC}
Li~Ma and Dezhong Chen, \emph{Radial symmetry and uniqueness for positive
  solutions of a {Schr\"odinger} type system}, Math. Comput. Modelling
  \textbf{49} (2009), 379--385,
  \href{https://doi.org/10.1016/j.mcm.2008.06.010}{doi:10.1016/j.mcm.2008.06.010}.

\bibitem{MZ}
Li~Ma and Lin Zhao, \emph{Uniqueness of ground states of some coupled nonlinear
  {Schr\"odinger} systems and their application}, Journal of Differential
  Equations \textbf{245} (2008), no.~9, 2551--2565,
  \href{https://doi.org/10.1016/j.jde.2008.04.008}{doi:10.1016/j.jde.2008.04.008}.

\bibitem{MaiaMiyagaki}
Braulio B.~V. Maia and Ol\'{i}mpio~H. Miyagaki, \emph{Normalized solutions for
  a {Hamiltonian} elliptic system on {$\R^N$}}, Proc. Roy. Soc.
  Edinburgh Sect. A (2026), 1--22, 
  \href{https://doi.org/10.1017/prm.2026.10143}{doi:10.1017/prm.2026.10143}.

\bibitem{Merle1992}
Frank Merle,
\emph{On uniqueness and continuation properties after blow-up time of
self-similar solutions of nonlinear {Schr\"odinger} equation with
critical exponent and critical mass}, Comm. Pure Appl. Math.
\textbf{45} (1992), no.~2, 203--254,
\href{https://doi.org/10.1002/cpa.3160450204}{doi:10.1002/cpa.3160450204}.

\bibitem{MerleRaphael2005}
Frank Merle and Pierre Rapha\"el,
\emph{The blow-up dynamic and upper bound on the blow-up rate for
critical nonlinear {Schr\"odinger} equation}, Ann. of Math. (2)
\textbf{161} (2005), no.~1, 157--222,
\href{https://doi.org/10.4007/annals.2005.161.157}{doi:10.4007/annals.2005.161.157}.

\bibitem{Mitidieri1993}
Enzo Mitidieri, \emph{A {Rellich} type identity and applications},
  Comm. Partial Differential Equations \textbf{18} (1993), no.~1--2,
  125--151,
  \href{https://doi.org/10.1080/03605309308820923}{doi:10.1080/03605309308820923}.

\bibitem{Mitidieri}
Enzo Mitidieri, \emph{Non-existence of positive solutions of semilinear
  elliptic systems in {$\R^N$}}, Differential Integral Equations
  \textbf{9} (1996), no.~3, 465--479,
  \href{https://doi.org/10.57262/die/1367969966}{doi:10.57262/die/1367969966}.

\bibitem{Moser1971}
J\"urgen Moser, \emph{A sharp form of an inequality by N. Trudinger},
Indiana Univ. Math. J. \textbf{20} (1971), no.~11, 1077--1092,
\href{https://doi.org/10.1512/iumj.1971.20.20101}
{doi:10.1512/iumj.1971.20.20101}.


\bibitem{NIST}
{National Institute of Standards and Technology}, \emph{{Digital Library of
  Mathematical Functions}}, Chapter 10, Bessel functions, in particular
  Sections 10.25, 10.28, 10.30 and 10.40.
  \href{https://dlmf.nist.gov/10}{dlmf.nist.gov/10}.

\bibitem{N}
Louis Nirenberg, \emph{On elliptic partial differential equations}, Annali
  della Scuola Normale Superiore di Pisa - Scienze Fisiche e Matematiche, Serie
  3 \textbf{13} (1959), no.~2, 115--162,
  \href{https://www.numdam.org/item/ASNSP_1959_3_13_2_115_0/}{Numdam full
  text}.

\bibitem{PucciSerrin1998}
Patrizia Pucci and James Serrin,
\emph{Uniqueness of ground states for quasilinear elliptic equations
in the exponential case}, Indiana Univ. Math. J. \textbf{47}
(1998), no.~2, 529--539,
\href{https://doi.org/10.1512/iumj.1998.47.2045}{doi:10.1512/iumj.1998.47.2045}.

\bibitem{QTZ}
Dongdong Qin, Xianhua Tang, and Jian Zhang, \emph{Ground states for planar
  {Hamiltonian} elliptic systems with critical exponential growth}, J.
  Differential Equations \textbf{308} (2022), 130--159,
  \href{https://doi.org/10.1016/j.jde.2021.10.063}{doi:10.1016/j.jde.2021.10.063}.

\bibitem{Ruf}
Bernhard Ruf, \emph{A sharp {Trudinger--Moser} type inequality for unbounded
  domains in {$\R^2$}}, J. Funct. Anal. \textbf{219} (2005), no.~2,
  340--367,
  \href{https://doi.org/10.1016/j.jfa.2004.06.013}{doi:10.1016/j.jfa.2004.06.013}.

\bibitem{ShakedShanthikumar}
Moshe Shaked and J. George Shanthikumar,
\emph{Stochastic Orders},
Springer Series in Statistics,
Springer, New York, 2007.

\bibitem{Sirakov}
Boyan Sirakov, \emph{On the existence of solutions of {Hamiltonian} elliptic
  systems in {$\R^N$}}, Adv. Differential Equations \textbf{5} (2000),
  1445--1464,
  \href{https://doi.org/10.57262/ade/1356651229}{doi:10.57262/ade/1356651229}.


\bibitem{SzulkinWeth2010}
Andrzej Szulkin and Tobias Weth, \emph{The method of Nehari manifold},
in \emph{Handbook of Nonconvex Analysis and Applications},
D.~Y. Gao and D. Motreanu (eds.), International Press, Somerville, MA,
2010, 597--632,
\href{https://staff.math.su.se/andrzejs/publications/Nehari.pdf}{Author manuscript}.

\bibitem{VdV}
Robert C.A.M. van der Vorst, \emph{Variational Identities and Applications to Differential Systems}, Arch. Ration. Mech. Anal., 116 (1991), 375--398, \href{https://doi.org/10.1007/BF00375674}{doi:10.1007/BF00375674}.

\bibitem{Weinstein}
Michael~I. Weinstein, \emph{Nonlinear {Schr\"odinger} equations and sharp
  interpolation estimates}, Comm. Math. Phys. \textbf{87} (1983), 567--576,
  \href{https://doi.org/10.1007/BF01208265}{doi:10.1007/BF01208265}.

\bibitem{Weinstein1986}
Michael~I. Weinstein,
\emph{Lyapunov stability of ground states of nonlinear dispersive evolution equations}, Comm. Pure Appl. Math. \textbf{39} (1986), no.~1, 51--67,
\href{https://doi.org/10.1002/cpa.3160390103}{doi:10.1002/cpa.3160390103}.

\bibitem{Willem1996}
Michel Willem, \emph{Minimax Theorems}, Progress in Nonlinear
Differential Equations and their Applications, vol.~24,
Birkh\"auser, Boston, 1996,
\href{https://doi.org/10.1007/978-1-4612-4146-1}{doi:10.1007/978-1-4612-4146-1}.

\bibitem{ZhangYuZhang}
Hui Zhang, Yuanyang Yu, and Jianjun Zhang, \emph{Normalized solutions for
  {Hamiltonian} elliptic systems with general nonlinearities}, 2026, Discrete
  Contin. Dyn. Syst. Ser. S, 2027, 33: 531-545
  \href{https://doi.org/10.3934/dcdss.2026178}{doi:10.3934/dcdss.2026178}.

\end{thebibliography}
\end{document}